\documentclass[12pt,a4paper,reqno]{amsart}
\usepackage{amsthm,amsfonts,amsmath,amssymb,latexsym}
\usepackage{epsfig,graphics,color}
\usepackage[all]{xy}
\usepackage[breaklinks=true]{hyperref}
\usepackage{stmaryrd}
\usepackage{verbatim}
\usepackage{bm}
\usepackage{mathabx}
\usepackage{enumitem}
\usepackage{ifthen}
\usepackage{tikz-cd}
 
\newlength{\XHeight}
\newlength{\XWidth}
\setlist[itemize,1]{leftmargin=\dimexpr 26pt-.1in}

\newtheorem{PARA}{}[section]
\newtheorem{theorem}[PARA]{Theorem}
\newtheorem{corollary}[PARA]{Corollary}
\newtheorem{lemma}[PARA]{Lemma}
\newtheorem{proposition}[PARA]{Proposition}
\newtheorem{definition}[PARA]{Definition}
\newtheorem{definition-proposition}[PARA]{Definition-Proposition}

\theoremstyle{definition}
\newtheorem{remark}[PARA]{Remark}
\theoremstyle{theorem}
\newtheorem{example}[PARA]{Example}
\newcommand{\para}{\begin{PARA}\rm}
\newcommand{\arap}{\end{PARA}\rm}
\newcommand{\dfn}{\begin{definition}\rm}
\newcommand{\nfd}{\end{definition}\rm}
\newcommand{\rmk}{\begin{remark}\rm}
\newcommand{\kmr}{\end{remark}\rm}
\newcommand{\xmpl}{\begin{example}\rm}
\newcommand{\lpmx}{\end{example}\rm}
\newcommand{\cA}{\mathcal{A}}
\newcommand{\cB}{\mathcal{B}}

\newcommand{\cC}{\mathcal{C}}
\newcommand{\cD}{\mathcal{D}}
\newcommand{\cF}{\mathcal{F}}

\newcommand{\cH}{\mathcal{H}}

\newcommand{\cK}{\mathcal{K}}

\newcommand{\cM}{\mathcal{M}}

\newcommand{\cP}{\mathcal{P}}

\newcommand{\uzeta}{\underline{\zeta}}

\newcommand{\one}
{{{\mathchoice \mathrm{ 1\mskip-4mu l} \mathrm{ 1\mskip-4mu l}
\mathrm{ 1\mskip-4.5mu l} \mathrm{ 1\mskip-5mu l}}}}

\renewcommand{\H}{{\mathbb{H}}}

\newcommand{\R}{{\mathbb{R}}}

\newcommand{\Z}{{\mathbb{Z}}}
\newcommand{\Id}{\operatorname{Id}}

\newcommand{\ev}{\mathrm{ev}}

\newcommand{\Hom}{\operatorname{Hom}}

\newcommand{\CZ}{\mathrm{CZ}}
\def\NABLA#1{{\mathop{\nabla\kern-.5ex\lower1ex\hbox{$#1$}}}}
\def\Nabla#1{\nabla\kern-.5ex{}_{#1}}
\def\Tabla#1{\Tilde\nabla\kern-.5ex{}_{#1}}
\renewcommand{\Tilde}{\widetilde}

\newcommand{\p}{{\partial}}

\newcommand{\wh}{\widehat}
\newcommand{\ol}{\overline}
\newcommand{\MM}{\mathcal{M}}

\newcommand{\wt}{\widetilde}

\newcommand*\Ham[3]{\ensuremath{
       \settowidth{\XWidth}{$\lambda$}
       \setlength{\XHeight}{\XWidth}
       \tikz[baseline]{
         \draw (-\XWidth,2\XHeight) -- (0,2\XHeight) -- (2\XWidth,-.5\XHeight) -- (3\XWidth,.5\XHeight) -- (4\XWidth,-.75\XHeight);
         \node at (1.5\XWidth,1.3\XHeight)[scale=0.5] {#1};
         \node at (2.9\XWidth,-.3\XHeight)[scale=0.5] {#2};
         \node at (4.3\XWidth,-.3\XHeight)[scale=0.5] {#3};
       }
     }
}

\newcommand*\ham{\ensuremath{
	\setlength{\XWidth}{.2em}
       \setlength{\XHeight}{\XWidth}
       \tikz[baseline]{
         \draw (-\XWidth,2\XHeight) -- (0,2\XHeight) -- (2\XWidth,-.5\XHeight) -- (3\XWidth,.5\XHeight) -- (4\XWidth,-.75\XHeight);
       }
     }
}

\newcommand*\Hamlambdamu[2]{\ensuremath{
       \settowidth{\XWidth}{$\lambda$}
       \setlength{\XHeight}{\XWidth}
       \tikz[baseline]{
         \draw (-\XWidth,2\XHeight) -- (0,2\XHeight) -- (2\XWidth,-.5\XHeight) -- (3\XWidth,.5\XHeight);
         \node at (1.5\XWidth,1.3\XHeight)[scale=0.5] {#1};
         \node at (2.9\XWidth,-.3\XHeight)[scale=0.5] {#2};
       }
     }
}

\newcommand*\hamlambdamu{\ensuremath{
	\setlength{\XWidth}{.2em}
       \setlength{\XHeight}{\XWidth}
       \tikz[baseline]{
         \draw (-\XWidth,2\XHeight) -- (0,2\XHeight) -- (2\XWidth,-.5\XHeight) -- (3\XWidth,.5\XHeight);
       }
     }
}

\newcommand*\HamLlambda[1]{\ensuremath{
       \settowidth{\XWidth}{$\lambda$}
       \setlength{\XHeight}{\XWidth}
       \tikz[baseline]{
         \draw (-\XWidth,2\XHeight) -- (0,2\XHeight) -- (2\XWidth,-.5\XHeight);
         \node at (1.5\XWidth,1.3\XHeight)[scale=0.5] {#1};
       }
     }
}

\newcommand*\hamLlambda{\ensuremath{
	\setlength{\XWidth}{.2em}
       \setlength{\XHeight}{\XWidth}
       \tikz[baseline]{
         \draw (-\XWidth,2\XHeight) -- (0,2\XHeight) -- (2\XWidth,-.5\XHeight);
       }
     }
}

\newcommand*\Hamvee[2]{\ensuremath{
       \settowidth{\XWidth}{$\lambda$}
       \setlength{\XHeight}{\XWidth}
       \tikz[baseline]{
         \draw (1\XWidth,.75\XHeight) -- (2\XWidth,-.5\XHeight) -- (3\XWidth,.5\XHeight);
         \node at (1.8\XWidth,1\XHeight)[scale=0.5] {#1};
         \node at (2.9\XWidth,-.3\XHeight)[scale=0.5] {#2};
       }
     }
}

\newcommand*\hamvee{\ensuremath{
	\setlength{\XWidth}{.2em}
       \setlength{\XHeight}{\XWidth}
       \tikz[baseline]{
         \draw (1\XWidth,.75\XHeight) -- (2\XWidth,-.5\XHeight) -- (3\XWidth,.5\XHeight);
       }
     }
}

\newcommand*\Hamwedge[2]{\ensuremath{
       \settowidth{\XWidth}{$\lambda$}
       \setlength{\XHeight}{\XWidth}
       \tikz[baseline]{
         \draw (2\XWidth,-.5\XHeight) -- (3\XWidth,.5\XHeight) -- (4\XWidth,-.75\XHeight);
         \node at (2.9\XWidth,-.3\XHeight)[scale=0.5] {#1};
         \node at (4.3\XWidth,-.3\XHeight)[scale=0.5] {#2};
       }
     }
}

\newcommand*\hamwedge{\ensuremath{
	\setlength{\XWidth}{.2em}
       \setlength{\XHeight}{\XWidth}
       \tikz[baseline]{
         \draw (2\XWidth,-.5\XHeight) -- (3\XWidth,.5\XHeight) -- (4\XWidth,-.75\XHeight);
       }
     }
}

\newcommand*\tikzprod[7]{\ensuremath{
	\setlength{\XWidth}{2em}
	\setlength{\XHeight}{\XWidth}
	\tikz[baseline]{
	 \ifthenelse{#2=1}
	 	{\draw[-<] (0\XWidth,0\XHeight) -- (-\XWidth,\XHeight);
		\node [left] at (-\XWidth,\XHeight) {$\cA$}
		}
		{\draw[->] (0\XWidth,0\XHeight) -- (-\XWidth,\XHeight)
		\node [left] at (-\XWidth,\XHeight) {$\cA^\vee$}
		}; 
	 \ifthenelse{#3=1}
	 	{\draw[-<] (0\XWidth,0\XHeight) -- (\XWidth,\XHeight);
		\node [right] at (\XWidth,\XHeight) {$\cA$}
		}
		{\draw[->] (0\XWidth,0\XHeight) -- (\XWidth,\XHeight)
		\node [right] at (\XWidth,\XHeight) {$\cA^\vee$}
		};
	 \ifthenelse{#4=1}
	 	{\draw[-<] (0\XWidth,0\XHeight) -- (0\XWidth,-\XHeight);
		\node [right] at (0\XWidth,-\XHeight) {$\cA$}
		}
		{\draw[->] (0\XWidth,0\XHeight) -- (0\XWidth,-\XHeight);
		\node [right] at (0\XWidth,-\XHeight) {$\cA^\vee$}
		};
	 \draw[-<] (0\XWidth,0\XHeight) -- (\XWidth,\XHeight); \draw[->] (0\XWidth,0\XHeight) -- (0\XWidth,-\XHeight);
	 \node [right] at (0\XWidth,0\XHeight)[scale=0.8] {#1};
	 \node [right] at (-.7\XWidth,.75\XHeight)[scale=.6] {#5};
	 \node [left] at (.7\XWidth,.75\XHeight)[scale=.6] {#6};
	 \node [left] at (0\XWidth,-.6\XHeight)[scale=.6] {#7}
	 }
	 }
	 } 

\newcommand{\umu}{{\underline{\mu}}}
\newcommand{\um}{{\underline{m}}}
\newcommand{\utau}{{\underline{\tau}}}
\newcommand{\usigma}{{\underline{\sigma}}}
\newcommand{\ubeta}{{\underline{\beta}}}
\newcommand{\uc}{{\underline{c}}}
\renewcommand{\uzeta}{{\underline{\zeta}}}

\begin{document}

\title[Multiplicative structures on cones and duality]{Multiplicative structures on cones and duality}
\author{Kai Cieliebak}
\address{Universit\"at Augsburg \newline Universit\"atsstrasse 14, D-86159 Augsburg, Germany}
\email{kai.cieliebak@math.uni-augsburg.de}
\author{Alexandru Oancea}
\address{
Universit\'e de Strasbourg \newline 
Institut de recherche math\'ematique avanc\'ee, IRMA \newline
Strasbourg, France}
\email{oancea@unistra.fr}
\date{\today}


\begin{abstract} 
We initiate the study of multiplicative structures on cones and show that cones of Floer continuation maps fit naturally in this framework. We apply this to give a new description of the multiplicative structure on Rabinowitz Floer homology and cohomology, and to give a new proof of the Poincar\'e duality theorem which relates the two. 
The underlying algebraic structure admits two incarnations, both new, which we study and compare: on the one hand the structure of $A_2^+$-algebra on the space $\cA$ of Floer chains, and on the other hand the structure of $A_2$-algebra involving $\cA$, its dual $\cA^\vee$ and a continuation map from $\cA^\vee$ to $\cA$.
\end{abstract}

\maketitle

\setcounter{tocdepth}{1}
\tableofcontents


\section{Introduction}\label{sec:introduction}


Rabinowitz Floer homology was originally defined as the Floer homology of the Rabinowitz action functional~\cite{Cieliebak-Frauenfelder}. An alternative description as ``V-shaped symplectic homology'' was found in~\cite{Cieliebak-Frauenfelder-Oancea}, relating Rabinowitz Floer homology to symplectic homology and cohomology. In~\cite{CO} and~\cite{Venkatesh} yet another description of Rabinowitz Floer homology was introduced, as the homology of the cone of a Floer continuation map. This proved crucial to understand its functoriality properties~\cite{CO} and to extend the definition to non-exact settings~\cite{Venkatesh}.  

Surprisingly, these three points of view are not at all equally well suited for the study of algebraic structures. Currently, the most versatile version seems to be V-shaped symplectic homology: it carries a product~\cite{CO} which is a straightforward adaptation of the classical pair-of-pants product in Floer theory, and also a secondary coproduct which together with the product defines a graded Frobenius algebra structure~\cite{CHO-PD}. A similar statement holds for V-shaped symplectic cohomology, the two being related by Poincar\'e duality~\cite{CHO-PD}. 
A definition of such product and coproduct structures in the original setting of the Rabinowitz action functional is not available, although~\cite{Abbondandolo-Merry} and~\cite{Frauenfelder-Weber} may independently lead there.

In this paper we define and study product structures on Rabinowitz Floer homology from the perspective of cones. Besides another proof of the Poincar\'e duality theorem~\cite{CHO-PD} in this framework, we obtain new insights into such structures that are not available by other approaches. Notions and results from this paper serve as inputs in several other articles: in~\cite{CHO-MorseFloerGH} to relate the graded Frobenius algebra structure on Rabinowitz Floer homology of a unit cotangent bundle to that on Rabinowitz loop homology, and in~\cite{CHO-PD,CHO-reducedSH} to prove the splitting theorem for Rabinowitz Floer homology in terms of symplectic homology and cohomology. As such, this paper plays a key role in the series of articles~\cite{CHO-index,CHO-PD,CHO-MorseFloerGH,CHOS-Cross,CO-algebra,CHO-reducedSH} on Poincar\'e duality for loop spaces and its applications.


{\bf Results of the paper}.
The cone of a chain map $c:\cM\to\cA$ is $Cone(c)=\cA\oplus \cM[-1]$
with differential $\p_{Cone(c)}(a,x)=(\p_\cA a + c(x), -\p_\cM x)$. 
In~\S\ref{sec:cones} we begin by spelling out the data corresponding to an $A_\infty$-structure on the cone of a chain map $c:\cM\to\cA$ such that $\cA\subset Cone(c)$ is an $A_\infty$-subalgebra. We call such data an {\em $A_\infty$-triple}. Ignoring higher homotopies, this leads to the notion of an {\em $A_2$-triple} $(\cM,c,\cA)$. It consists of a degree $0$ chain map $c:\cM\to \cA$ together with bilinear maps  
$\mu:\cA\otimes\cA\to\cA$, $m_L:\cA\otimes \cM\to \cM$, $m_R:\cM\otimes\cA\to \cM$
of degree $0$, $\tau_L:\cM\otimes \cA\to \cA$, $\tau_R:\cA\otimes \cM\to \cA$, $\sigma:\cM\otimes \cM\to \cM$ of degree $1$ and $\beta:\cM\otimes\cM\to \cA$ of degree $2$ satisfying suitable relations (see Definition~\ref{defi:pre-subalgebra}). 
In our first result we denote an element of $Cone(c)=\cA\oplus \cM[-1]$ by $(a,\bar x)$, where $\bar x\in\cM[-1]$ is the shift of an element $x\in\cM$. 

{\bf Theorem A }(= Propositions~\ref{prop:product_on_cone}, \ref{prop:arity2homotopyinvariance} and~\ref{prop:A2mor-to-cone})
{\em For an $A_2$-triple, the formula 
\begin{eqnarray*} 
\lefteqn{m \big((a,\bar x),(a',\bar x')\big)} \\
 & = \big( \mu(a,a')+(-1)^{|a|}\tau_R(a,x') + \tau_L(x,a') - (-1)^{|\bar x|}\beta(x,x'), \nonumber \\
 &  \qquad \qquad  \qquad \qquad (-1)^{|a|}\overline{m_L(a,x')}+\overline{m_R(x,a')}-(-1)^{|\bar x|}\overline{\sigma(x,x')}\big) \nonumber
\end{eqnarray*}
defines a degree $0$ bilinear product on $Cone(c)$ which is a chain map and thus descends to homology. This product is functorial with respect to homotopy retracts and morphisms of $A_2$-triples.} 

An important part of the discussion are sign conventions related to multilinear degree shifts, tensor products, and algebraic duals, which we defer to Appendices~\ref{sec:shifts} and~\ref{sec:dg_linear_algebra}. 

In~\S\ref{sec:cones_continuation} we apply the results of~\S\ref{sec:cones} to Floer theory. This is based on the observation that Floer continuation maps give rise to $A_\infty$-triples which are canonically defined up to homotopy equivalence. We prove in this paper only the arity 2 version of this statement:

{\bf Theorem B }(= Proposition~\ref{prop:Hamiltonian-pre-subalgebra}).
{\em Floer continuation maps give rise to $A_2$-triples which are canonically defined up to homotopy equivalence.} 

The construction involves moduli spaces of solutions of Floer equations parametrized by simplices of dimension $0$, $1$, and $2$. The construction of the $A_\infty$-triple
involves moduli spaces parametrized by higher dimensional simplices, much in the manner of~\cite{Ekholm-Oancea}, which should correspond to the assocoipahedra of Poirier and Tradler~\cite{Poirier-Tradler}.
See the discussion below and Remark~\ref{rmk:TQFT++}. 

Let now $W$ be a Liouville domain of dimension $2n$ with trivial canonical bundle, $SH_*$ symplectic homology graded by Conley-Zehnder indices, and $S\H_* = SH_{*+n}$ its degree shifted version. The power of the cone perspective derives from its applicability to various families of Floer continuation maps. Applying the construction of a product on the cone to the three families of Floer continuation maps $\{K_\lambda\}$, $\cH_\vee$ and $\cH_\wedge$ shown in Figure~\ref{fig:3families} gives rise to three rings $S\H_*(\{K_\lambda\})$, $S\H_*(\cH_\vee)$ and $S\H_*(\cH_\wedge)$.
\begin{figure}
\begin{center}
\includegraphics[width=\textwidth]{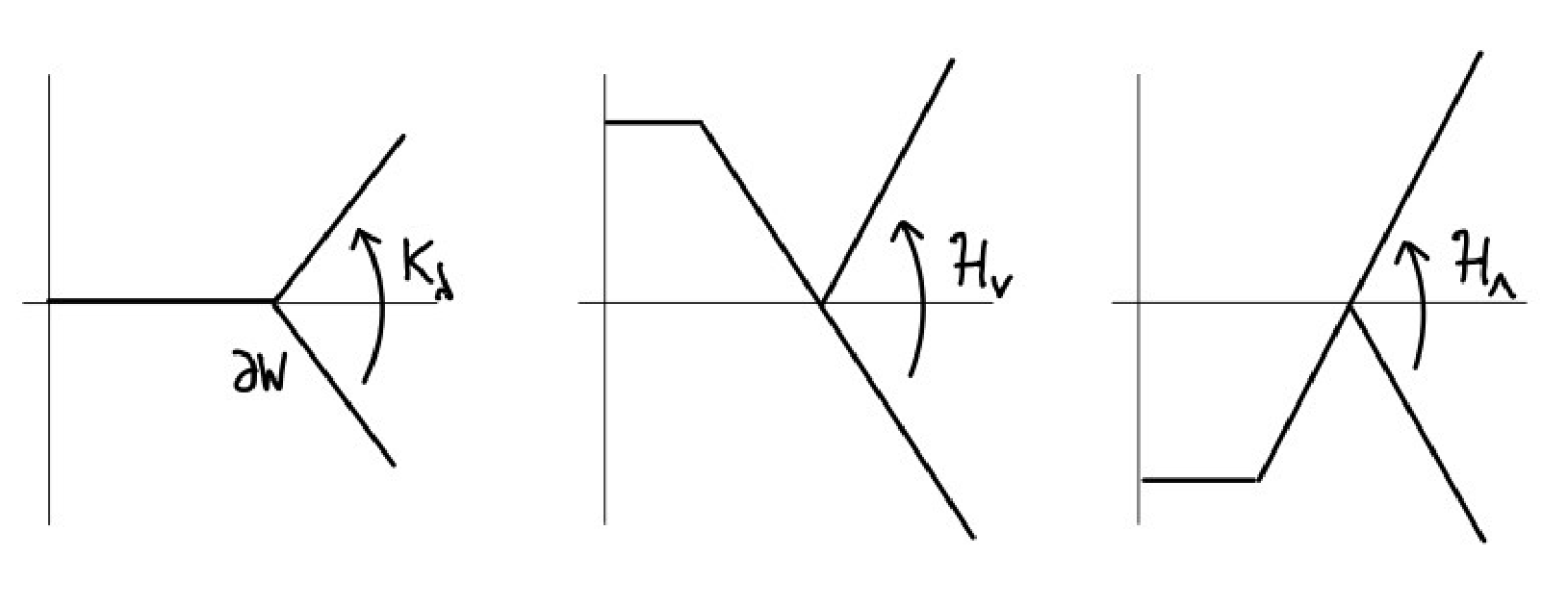}
\vspace{-.5cm}
\caption{The families of Floer continuation maps $\{K_\lambda\}$, $\cH_\vee$ and $\cH_\wedge$.}
\label{fig:3families}
\end{center}
\end{figure}
The first two of these recover Rabinowitz Floer homology $S\H_*(\p W)$ with its product defined in~\cite{CHO-PD}:

{\bf Theorem C. }
{\em We have canonical ring isomorphisms
$$
S\H_*(\{K_\lambda\})\simeq S\H_*(\cH_\vee) \simeq S\H_*(\p W).
$$}

The third one corresponds to Rabinowitz Floer cohomology $SH^{1-n-*}(\p W)$ with its product defined in~\cite{CHO-PD}:  

{\bf Theorem D. }
{\em We have canonical ring isomorphisms
$$
S\H_*(\cH_\wedge)\simeq \wh{SH}_{*+n-1}(\p W) \simeq SH^{1-n-*}(\p W). 
$$}

Here $\wh{SH}_*(\p W) = SH_*(V,\p V)$ for the symplectic cobordism $V=[0,1]\times\p W$ in the terminology of~\cite{CO}. In~\S\ref{sec:alternative_product} and~\S\ref{sec:unit} we give direct definitions of the product and its unit on $\wh{SH}_{*+n-1}(\p W)$, and thus on Rabinowitz Floer cohomology. 
%

{\bf Theorem E (Cone duality). } 
{\em There is a canonical ring isomorphism 
$$
S\H_*(\cH_\vee)\simeq S\H_*(\cH_\wedge). 
$$}

The above ring isomorphisms are summarized in the following diagram:
$$
\xymatrix
@R=20pt
@C=50pt
{
   & & SH^{1-n-*}(\p W) \ar@{=}_{\mbox{\scriptsize Thm.~\ref{thm:coh-product}\quad }}[ld] \ar@{=}^{\mbox{\scriptsize Prop.~\ref{prop:products_coincide}}}[d] \\
   & S\H_*(\p W) \ar@{=}_{\mbox{\scriptsize Thm.~\ref{thm:Rabinowitz-cone}}}[d] \ar@{=}_-{\mbox{\scriptsize Thm.~\ref{thm:PD-cones}}}[r] & \wh{SH}_{*+n-1}(\p W) \ar@{=}^{\mbox{\scriptsize Thm.~\ref{thm:PD-cones2}}}[d] \\
   & S\H_*(\cH_\vee) \ar@{=}_{\mbox{\scriptsize Thm.~\ref{thm:Rabinowitz-cone}}}[d] \ar@{=}_{\mbox{\scriptsize Thm.~\ref{thm:duality-cone}}}[r] & S\H_*(\cH_\wedge) \\
   & S\H_*(\{K_\lambda\}) 
}
$$
Theorems C, D and E together give a new proof of the Poincar\'e duality theorem from~\cite{CHO-PD} (which we restate in this paper as Theorem~\ref{thm:coh-product}). Moreover, the preceding theorems are crucial inputs for the proof of the splitting theorem for Rabinowitz Floer homology (\cite[Proposition~7.6]{CHO-reducedSH} and~\cite[Theorem~1.5]{CHO-PD}). In turn, this splitting theorem serves as a computational tool for the multiplicative structure on Rabinowitz Floer homology, see~\cite[Theorem~7.8]{CHO-reducedSH} and~\cite[\S2]{CHO-PD}. 

The explicit description of the product on $\wh{SH}_{*+n-1}(\p W)$ allows us to relate it to the {\em varying weights secondary product}, originally constructed by Seidel and further explored in~\cite{Ekholm-Oancea,AS-product-structures}: 

{\bf Theorem F }(= Proposition~\ref{prop:cont-weight}).
{\em On negative action symplectic homology $\wh{SH}_{*+n-1}^{<0}(\p W) \simeq SH^{1-n-*}_{>0}(\p W)$ the product in Theorem D agrees with the varying weights secondary product.} 

Rabinowitz Floer homology arises from an $A_2$-triple $(\cM,c,\cA)$ in which $\cM$ is the algebraic dual $\cA^\vee$ of $\cA$. This allows one to encode all the structure on $\cA$ alone. We formalize this in~\S\ref{sec:A2+} in the notion of an {\em $A_2^+$-structure} on $\cA$. It consists of a degree $0$ copairing $c_0\in\cA\otimes \cA$, a degree $1$ secondary copairing $Q_0\in \cA\otimes\cA$, a degree $0$ product $\mu:\cA\otimes \cA\to \cA$, a degree $1$ secondary coproduct $\lambda:\cA\to \cA\otimes \cA$ and a degree $2$ cubic vector $B\in\cA\otimes\cA\otimes \cA$ satisfying suitable relations (see Definition~\ref{defi:A2+structure}). 

{\bf Theorem G }(= Propositions~\ref{prop:TQFT+} and~\ref{prop:A2+toA2-mor}). 
{\em An $A_2^+$-structure on $\cA$ gives rise to an $A_2$-triple $(\cA^\vee,c,\cA)$, and a morphism of $A_2^+$-algebras gives rise to a correspondence of $A_2$-triples.}

As a consequence, a quasi-isomorphism of $A_2^+$-algebras induces a ring isomorphism between the homologies of the cones (Corollary~\ref{cor:A2+toA2-mor}). This result is a crucial input for the proof of the Frobenius algebra isomorphism between Rabinowitz Floer homology of a unit cotangent bundle and Rabinowitz loop homology in~\cite{CHO-MorseFloerGH}. Here the difficulty lies in the fact that Rabinowitz homology is an amalgamation of homology and cohomology, which are covariant resp.~contravariant under the Viterbo isomorphism. It is overcome by encoding the whole structure as an $A_2^+$-structure on homology alone and using its functoriality above. 

Since the notion of an $A_2^+$-structure is self-dual, in addition to a product $\boldsymbol{\mu}$ it also induces a coproduct $\boldsymbol{\lambda}$ on its cone.
This product and coproduct can be read as a coproduct $\boldsymbol{\mu}^\vee$ and product $\boldsymbol{\lambda}^\vee$ on the dual of the cone. The following result can be viewed as an algebraic counterpart of the Poincar\'e Duality Theorem in~\cite{CHO-PD}.

{\bf Theorem H (Duality for $A_2^+$-algebras) }(= Theorem~\ref{thm:PD-A2+}).
{\em Let $\cA$ be an $A_2^+$-algebra which is free and finite dimensional in each degree. Then we have a canonical isomorphism 
$$
  \Bigl(H_*\bigl(Cone(c)\bigr),\boldsymbol{\mu},\boldsymbol{\lambda}\Bigr)\cong
  \Bigl(H_{*-1}\bigl(Cone(c)^\vee\bigr),\boldsymbol{\lambda}^\vee \tau,\tau \boldsymbol{\mu}^\vee\Bigr)
$$
intertwining the products and coproducts, where $\tau$ is the map flipping the tensor factors with signs.}

{\bf Extensions. }
The results in this paper have several straightforward extensions. 
(a) In this paper we treat explicitly only products, but there is an entirely parallel discussion of coproducts. 
(b) Sara Venkatesh defined in~\cite{Venkatesh} Rabinowitz Floer homology in non-exact settings as the cone of such a continuation map at the additive level. The perspective on multiplicative structures that we adopt here is robust and would carry over to such situations.
(c) We formulate all our results for closed strings, in terms of symplectic homology, but they have natural open string counterparts with Lagangian boundary conditions in terms of wrapped Floer homology. See~\cite{CO,CHO-PD} for a discussion of such extensions, in particular regarding the gradings.

{\bf Relation to other work and future directions. }
The notion of an $A_2^+$-algebra encodes only operations with up to $3$ inputs and outputs, which suffices for our applications on the level of homology. Considering Floer moduli spaces with more inputs and outputs leads to a homotopical version of an $A_2^+$-algebra, which we may call an ``$A_\infty^+$-algebra''. Special cases of this structure have appeared in the literature under various names: $V_\infty$-algebras in Tradler--Zeinalian~\cite{Tradler-Zeinalian}, double Poisson algebras in Van den Bergh~\cite{VandenBergh08}, and pre-Calabi-Yau algebras in Iyudu--Kontsevich--Vlassopoulos~\cite{Iyudu-Kontsevich-Vlassopoulos21} and Kontsevich-Takeda-Vlassopoulos~\cite{KTV}. In fact, Leray and Vallette have proved that curved versions of these three structures are equivalent~\cite{Leray-Vallette}. The presence of the copairing $c_0$
in Definition~\ref{defi:A2+structure} means that our structure is also curved. What distinguishes it from the other structures is the presence of the secondary copairing $Q_0$ which measures the failure of $c_0$ to be symmetric. This is dictated to us by the applications in symplectic homology and string topology, where a nontrivial $Q_0$ occurs and is responsible for various subtleties concerning the algebraic structure on reduced homology~\cite{CHO-reducedSH}. 

Poirier and Tradler have shown that the dioperad $V_\infty$ is governed by a family of polytopes called ``assocoipahedra"~\cite{Poirier-Tradler, Poirier-Tradler-Koszulity}. In view of the preceding discussion, we expect that these can be realized as Floer moduli spaces in the absence of $Q_0$. This connection was first proposed by T.~Mazuir after observing the equivalence of Figure 7 in~\cite{CHO-reducedSH} with Figure 9 in~\cite{Poirier-Tradler}. It is an interesting challenge to find polyhedral descriptions of the Floer moduli spaces in the presence of $Q_0$ and study the resulting dioperad. 

It is proved in~\cite{CHO-PD} that Rabinowitz Floer homology of the boundary of a Liouville domain carries the structure of a graded Frobenius algebra. This structure emanates from the cone description of Rabinowitz Floer homology if we enhance the underlying $A_2^+$-algebra to include operations with $4$ inputs and outputs. Based on this, we conjecture that an $A_\infty^+$-algebra gives rise to a homotopy version of a graded Frobenius algebra on its cone. 
The development of this homotopy theory is the subject of ongoing work by Mazuir~\cite{Mazuir-in-progress}.

It has long been observed that the definition of Rabinowitz Floer homology resembles that of Tate homology, see e.g.~\cite{ACF14} for an application of this point of view in equivariant homology. 
We conjecture that our Poincar\'e Duality Theorem~\ref{thm:coh-product} is an instantiation of Tate duality and we intend to explore this in future work. For cotangent bundles of simply connected manifolds, the structure we find on Rabinowitz Floer homology should coincide with the structure found by Rivera and Wang~\cite{Rivera-Wang} on singular Hochschild cohomology of the dga of cochains. 
Our work should also be related to earlier work of Kaufmann such as~\cite{Kaufmann07}, and to Kontsevich graph complexes as studied by Merkulov--Willwacher~\cite{Merkulov-Willwacher}.

In the closed string case, Rabinowitz Floer homology also carries a natural BV operator, see~\cite{CHOS-Cross}. To complete the algebraic picture, one should therefore incorporate BV-structures in the formalisms developed in this paper. The $S^1$-equivariant versions of these structures should give rise to a refinement of the ${\rm IBL}_\infty$ structures from~\cite{Cieliebak-Fukaya-Latschev}. 

The wrapped Floer homology groups have a categorical refinement given by the wrapped Fukaya category~\cite{Abouzaid-Seidel}. Recently Ganatra-Gao-Venkatesh~\cite{GGV} and Bae-Jeong-Kim~\cite{BJK} have defined a categorical refinement of the Lagrangian Rabinowitz Floer homology groups called Rabinowitz Fukaya category. The cone perspective on Rabinowitz Floer homology plays a key role, and it is an interesting question to compare the moduli spaces used in those constructions with the moduli spaces from this paper. In~\cite{BJK} it is proved that, under reasonable assumptions, the Rabinowitz Fukaya category of a Liouville domain of dimension $2n$ is $(n-1)$-Calabi-Yau. This raises the question whether, and how, such a statement implies our Poincar\'e duality theorem. Legout~\cite{Legout} had previously defined a categorical refinement of the Cthulhu homology of Chantraine-Dimitroglou Rizell-Ghiggini-Golovko~\cite{CDRGG20}, in the setting of Lagrangian cobordisms. We conjecture that, when transposed to the context of Liouville domains, the Legout category is equivalent to the Rabinowitz Fukaya category.

{\bf Acknowledgements. } This paper is a split-off from our collaboration with Nancy Hingston on Poincar\'e duality. Without her far reaching vision this could not have come into being. The authors benefited from discussions with M.~Abouzaid, B.~Chantraine, F.~Chapoton, V.~Dotsenko, P.~Ghiggini, R.~Kaufmann, T.~Mazuir, M.~Rivera, and S.~Venkatesh.
The first author thanks Stanford University, Institut Mittag--Leffler, and the Institute for Advanced Study for their hospitality over the duration of this project.
The second author acknowledges the hospitality of Helmut Hofer and IAS in 2017, 2019, and 2022.

This work has benefited from financial support via the grants MICROLOCAL ANR-15-CE40-0007, ENUMGEOM ANR-18-CE40-0009 and COSY ANR-21-CE40-0002. In its late stages, it has also benefited from support provided to the second author by the University of Strasbourg Institute for Advanced Study (USIAS) for a Fellowship, within the French national programme ``Investment for the future" (IdEx-Unistra).

\section{Products on cones}\label{sec:cones}

Throughout this section we use homological conventions and coefficients in a principal ideal domain $R$. Recall that, given a degree $0$ chain map $c:\cM_*\to \cA_*$ between chain complexes whose differentials have degree $-1$,
the cone of a chain map $c:\cM\to\cA$ is the chain complex
$$
Cone(c)=\cA\oplus \cM[-1],\qquad \p_{Cone(c)} = \begin{pmatrix}
  \p_\cA & c \\ 0 & -\p_\cM \end{pmatrix}
$$
with $\cM[-1]_*=\cM_{*-1}$ and $\p_{\cM[-1]}=-\p_\cM$. 
The inclusion $\cA\hookrightarrow Cone(c)$, $a\mapsto (a,0)$ is a chain map. We shall freely refer in this section to the notation of Appendix~\ref{sec:shifts} regarding degree shifts in the multilinear setting.

\begin{definition} \label{defi:Ainfty}
An \emph{$A_\infty$-algebra} $\cA$ is a $\Z$-graded $R$-module endowed with a collection of maps $\mu^d_\cA:\cA[-1]^{\otimes d}\to \cA[-1]$, $d\ge 1$ of degree $-1$ satisfying the relations
\begin{equation} \label{eq:Ainfty}
\sum_{i+j=k+1}\sum_{t=1}^i \mu^i_\cA(\one^{\otimes \, t-1} \otimes \mu^j_\cA\otimes\one^{\otimes \, i-t})=0, \qquad k\ge 1.
\end{equation}
\end{definition}
In the absence of the shift the structure maps $\mu^d_\cA:\cA^{\otimes d}\to \cA$ have degree $d-2$. The functional relation translates by evaluation into the relation 
\begin{align*}
\sum_{i+j=k+1}\sum_{t=1}^i & (-1)^{\|a_1\|+\dots +\|a_{t-1}\|} \\
& \mu^i_\cA(a_1,\dots, a_{t-1}, \mu^j_\cA(a_t,\dots, a_{t+j-1}),a_{t+j},\dots a_i)=0, \quad k\ge 1,
\end{align*}
where $\|a\|=|a|+1$ denotes the shifted degree. The data of an $A_\infty$-algebra on a $\Z$-graded $R$-module $\cA$ is equivalent to the data of a square-zero coderivation on ${\overline T}^c(\cA[-1])=\oplus_{k\ge 1}\cA[-1]^{\otimes k}$, the reduced tensor coalgebra on $\cA[-1]$. This is because each $\mu^d_\cA$ can be uniquely extended as a coderivation $\overline\mu^d_\cA:\overline T^c(\cA[-1])\to \overline T^c(\cA[-1])$ and, setting $\overline \mu_\cA=\sum_d \overline \mu^d_\cA$, the $A_\infty$-relations are equivalent to $\overline \mu_\cA\circ \overline\mu_\cA=0$.
\begin{definition} \label{defi:Ainfty-mor}
Given two $A_\infty$-algebras $\cA$ and $\cA'$, an \emph{$A_\infty$-morphism} $\cF:\cA\to \cA'$ is a collection of maps $\cF^d:\cA[-1]^{\otimes d}\to \cA'[-1]$, $d\ge 1$ of degree $0$ which satisfy the relations 
\begin{equation} \label{eq:Ainfty-mor}
 \sum_{i=1}^d \sum_{j,k,\ell}\cF^i(\one^{\otimes \, j} \otimes \mu_\cA^k\otimes \one^{\otimes \, \ell}) = 
 \sum_r \sum_{i_1,\dots,i_r}\mu_{\cA'}^r(\cF^{i_1}\otimes \dots \otimes \cF^{i_r}), \quad d\ge 1.
\end{equation}  
\end{definition}
The data of the collection $\{\cF^d\}$ is equivalent to the data of a chain map $\overline \cF: ({\overline T}^c(\cA[-1]),\overline \mu_\cA) \to ({\overline T}^c(\cA'[-1]), \overline \mu_{\cA'})$, i.e. $\overline \cF \circ \overline \mu_\cA = \overline \mu_{\cA'}\circ \overline \cF$.

If $\mu^d_\cA=0$ for $d>2$ the $A_\infty$-structure determines a dga structure on $\cA$ by the formulas 
\begin{align}\label{eq:dga-from-Ainfty}
\p_\cA& =-\mu^1_\cA=\mu^1_\cA[1],\cr
a\cdot a' & = (-1)^{|a|}\mu^2_\cA(a,a')=\mu^2_\cA[1,1;1](a,a'). 
\end{align}
See Appendix~\ref{sec:shifts} for an explanation of the notation $\mu^2_\cA[1,1;1]$.
In the case where $\mu^d_\cA=0$ for $d>3$ the same formulas determine an associative-up-to-homotopy differential graded algebra structure on $\cA$. More precisely, with $h(a_1,a_2,a_3)=(-1)^{|a_2|}\mu^3_\cA(a_1,a_2,a_3)$, i.e. $h=\mu^3_\cA[1,1,1;1]$ in the notation of Appendix~\ref{sec:shifts}, we have
\footnote{The conventions for the definition of an $A_\infty$-algebra vary greatly throughout the literature, and the difference stems mainly from the point of view adopted: either square-zero coderivation on $\overline T^c(\cA[-1])$, or homotopy relaxation of dga structure. All conventions are equivalent to one another by suitable sign changes, and the short note by Polishchuk~\cite{Polishchuk} contains a useful comparison. Our convention derives the functional relation~\eqref{eq:Ainfty}, which involves no signs, from the condition that the associated coderivation on the reduced tensor coalgebra on $\cA[-1]$ squares to zero. This is the point of view of Fukaya-Oh-Ohta-Ono~\cite[Definition~3.2.3]{FOOO1} and Seidel~[(2.1)]\cite{Seidel-HHA}. Seidel's convention in~\cite[(1.2)]{Seidel-book} is essentially the same except that the maps $\mu^d_\cA$ are viewed as acting from the right, so one passes from one convention to the other by defining $\tilde \mu^d_\cA(a_d,\dots,a_1)=\mu^d_\cA(a_1,\dots,a_d)$. 

The conventions of Lef\`evre-Hasegawa~\cite[1.2.1.1]{Lefevre-Hasegawa} and Markl~\cite[(2)]{Markl-transfer}, which coincide, are such that the first three structure maps directly define on $\cA$ an associative-up-to-homotopy differential graded algebra structure. They are the same as the original one of Stasheff~\cite{Stasheff-Ainfty2}, cf.~\cite{Polishchuk}. Lef\`evre-Hasegawa writes down 
in~\cite[Lemma~1.2.2.1]{Lefevre-Hasegawa} a transformation through which this point of view is equivalent to ours, and that transformation inspired our treatment of shifts in Appendix~\ref{sec:shifts}. 

The conventions by which one associates to an $A_\infty$-algebra in the sense of Definition~\ref{defi:Ainfty} an associative-up-to-homotopy dga structure vary greatly as well. Ours is different from both the one of Fukaya-Oh-Ohta-Ono~\cite[(3.2.5)]{FOOO1} and the one of Seidel~\cite[(1.3)]{Seidel-book}. We favour it because it fits into a systematic procedure of shifting multilinear maps, cf. Appendix~\ref{sec:shifts}, and also because it realizes $(\cA,\p_\cA)$ as the shift of $(\cA[-1], \mu^1_\cA)$. } 
\begin{equation} \label{eq:associator-from-Ainfty}
(a_1\cdot a_2)\cdot a_3 - a_1\cdot(a_2\cdot a_3) = [\p_\cA,h].
\end{equation}
Similarly, given an $A_\infty$-morphism $\cF:\cA\to \cA'$ the map $f=\cF^1[1;1]:\cA\to \cA'$ acts on elements as $f(a)=\cF^1(a)$ and is a chain map which intertwines the products $m_\cA=\mu^2_{\cA}[1,1;1]$ and $m_{\cA'}=\mu^2_{\cA'}[1,1;1]$ on $\cA$ and $\cA'$ up to a homotopy $k=\cF^2[1,1;1]$, i.e. 
$$
m_{\cA'}(f\otimes f) - f m_\cA = [\p, k]. 
$$
Our discussion evolves around the following refinement of the notion of an $A_\infty$-algebra and $A_\infty$-morphism.

\begin{definition}  \label{defi:Ainfty-subalgebra-extension-triple}
An \emph{$A_\infty$-triple} $(\cM,c,\cA)$ consists of an $A_\infty$-algebra $\cA$, a chain complex $\cM$ and a degree $0$ chain map 
$$
c:\cM\to \cA,
$$ 
together with an $A_\infty$-algebra structure on $Cone(c)$ extending the $A_\infty$-algebra structure on $\cA$ via the inclusion $\cA\hookrightarrow Cone(c)=\cA\oplus \cM[-1]$, and such that $\mu^1_{Cone(c)}=\p_{Cone(c)}[-1]$.
\end{definition}

\begin{definition} \label{defi:Ainfty-subalgebra-extension-triple-morphism}
An \emph{$A_\infty$-morphism} $\cF:(\cM,c,\cA)\to (\cM',c',\cA')$ between $A_\infty$-triples is an $A_\infty$-morphism $Cone(c)\to Cone(c')$ which restricts to an $A_\infty$-morphism $\cA\to \cA'$. 
\end{definition}

\subsection{The point of view of split pairs of $A_\infty$-algebras} \label{sec:Ainfty-pairs}

The data of an $A_\infty$-triple is equivalent to the data of an \emph{$R$-split pair of $A_\infty$-algebras}, or simply \emph{split pair of $A_\infty$-algebras}, meaning an inclusion of $A_\infty$-algebras $\cA\subset \cB$ together with a splitting (as graded $R$-modules) $s:\cB/\cA\to \cB$ of the short exact sequence $0\to \cA\to \cB\to \cB/\cA\to 0$. Given such a pair we write the differential on $\cB$ in upper triangular form with respect to the decomposition $\cB=\cA\oplus (\cB/\cA)$ induced by the splitting, we set $\cM=(\cB/\cA)[1]$ and we define $c:\cM\to \cA$ to be the $(\cB/\cA,\cA)$-component of the differential. Then there is an induced structure of $A_\infty$-triple on $(\cM,c,\cA)$ such that $Cone(c)=\cB$. Conversely, given an $A_\infty$-triple $(\cM,c,\cA)$ the inclusion $\cA\subset Cone(c)$ is obviously split. 

The key property of $A_\infty$-structures is that they obey the Homotopy Transfer Theorem, see~\cite{Markl-transfer} for its most general form and~\cite[\S9.4]{Loday-Vallette} for a contextual discussion. In the context of split pairs of $A_\infty$-algebras $\cA\subset \cB$ the homotopy transfer theorem adapts in an obvious way by considering only maps which are upper triangular with respect to the decomposition $\cB=\cA\oplus \cB/\cA$ provided by the splitting.  

\begin{definition}\label{def:homotopy-retract}
Let $\cA\subset \cB$ and $\cA'\subset \cB'$ be split pairs of chain complexes. We say that the pair $\cA'\subset\cB'$ is a {\em homotopy retract} of the pair $\cA\subset\cB$ if there are maps 
%
$$
\begin{tikzcd}
\arrow[loop left,"H"]\cB \arrow[r, shift left, "P"] & \arrow[l,shift left,"I"] \cB'
\end{tikzcd}
$$
which are upper triangular with respect to the decompositions $\cB=\cA\oplus \cB/\cA$ and $\cB'=\cA'\oplus \cB'/\cA'$ provided by the splittings and such that 
$$
[\p,H]=\one - IP.
$$ 
\end{definition}

The proof of the next theorem is the same as that of~\cite[Theorem~5]{Markl-transfer}. It specifically uses the upper triangular form of the maps $P,I,H$, and also the explicit formulas provided by Markl in~\cite{Markl-transfer}. 

\begin{theorem}[Homotopy transfer for split pairs]Given a homotopy retraction of pairs as above, and given an $A_\infty$-algebra structure on $\cB$ such that $\cA$ is a subalgebra, there is an $A_\infty$-structure on $\cB'$ such that $\cA'$ is a subalgebra, and there are extensions of $P,I$ to $A_\infty$-morphisms of pairs $\tilde P, \tilde I$ and of $H$ to an $A_\infty$-homotopy $\tilde H$ between $\tilde I \tilde P$ and $\one$ which preserves $\cA$.  \qed\end{theorem}

The transferred $A_\infty$-structure and the extensions $\tilde P, \tilde I, \tilde H$ are described very explicitly in terms of summation over trees whose vertices are at least trivalent, see~\cite{Markl-transfer} and also~\cite{KS-HMS-torus}. 

One particularly relevant situation is that in which the maps $P,I$ defining a homotopy retract are actually homotopy inverses, with the homotopy $H':\cB'\to\cB'$ such that $[\p,H']=\one-PI$ being also in upper triangular form. The next result is the analogue of~\cite[Proposition~12]{Markl-transfer}.  

\begin{proposition} \label{prop:transfer-homotopy-equivalence}
In the preceding situation, the homotopy $H'$ can be extended to an $A_\infty$-homotopy $\tilde H'$ between $\tilde P\tilde I$ and $\one$ which preserves $\cA$ if 
$$
[PH-H'P]=0\in H_1(Hom(\cB,\cB')),
$$
and $PH-H'P$ admits a primitive which is upper triangular. \qed
\end{proposition}

\subsection{The point of view of $A_\infty$-bimodules} \label{sec:infty-bimodules}

Recall that an $A_\infty$-algebra $\cB$ is a graded $R$-module with operations $\mu^d_\cB:\cB[-1]^{\otimes d}\to \cB[-1]$, $d\ge 1$ of degree $-1$ subject to the relations~\eqref{eq:Ainfty}. 

The data of an $A_\infty$-triple $(\cM,c,\cA)$ can then be explicitly encoded in two collections of operations 
$$
m^{i_1|j_1|\dots|i_k|j_k} : \cA[-1]^{\otimes i_1}\otimes \cM[-2]^{\otimes j_1}\otimes\dots \otimes \cA[-1]^{\otimes i_k}\otimes \cM[-2]^{\otimes j_k}\to \cM[-2], 
$$
and
$$
\tau^{i_1|j_1|\dots|i_k|j_k}: \cA[-1]^{\otimes i_1}\otimes \cM[-2]^{\otimes j_1}\otimes\dots \otimes \cA[-1]^{\otimes i_k}\otimes \cM[-2]^{\otimes j_k}\to \cA[-1]
$$
of degree $-1$, indexed by tuples of non-negative integers $i_1,j_1,\dots,i_k,j_k$ such that the intermediate indices $j_1,i_2,j_2,\dots,i_k$ are nonzero---a convention which we adopt for non-redundancy of the notation--- and such that the following conditions hold: 
\begin{itemize}
\item $m^{d|0}=0$ and $\tau^{d|0}=\mu^d_{\cA}$ for all $d\ge 1$, where $\mu^d_\cA$, $d\ge 1$ are the $A_\infty$-operations for $\cA$; 
\item $\tau^{0|1}=-c[-2;-1]:\cM[-2]\to \cA[-1]$ and $m^{0|1}=\p_{\cM[-2]}$;
\item the operations $\mu^d:Cone(c)[-1]^{\otimes d}\to Cone(c)[-1]$, $d\ge 1$ given by 
\begin{align*}
\mu^d|\cA[-1]^{\otimes i_1}\otimes \cM[-2]^{\otimes j_1}\otimes\dots \otimes & \cA[-1]^{\otimes i_k} \otimes \cM[-2]^{\otimes j_k} \\
& = \tau^{i_1|j_1|\dots|i_k|j_k}\oplus m^{i_1|j_1|\dots|i_k|j_k}
\end{align*}
define an $A_\infty$-algebra structure on $Cone(c)$. 
\end{itemize}

The collections of operations $\{m^{i_1|j_1|\dots|i_k|j_k}\}$ and $\{\tau^{i_1|j_1|\dots|i_k|j_k}\}$ can be further partitioned according to the value of $\mathfrak{j}=j_1+j_2+\dots+j_k$. It is instructive to spell out the meaning of the sub-collections which correspond to the first two values of $\mathfrak{j}$.
\begin{itemize}
\item The case $\mathfrak{j}=0$ is covered by the first of the three conditions above. This expresses the fact that $\cA$ is an $A_\infty$-subalgebra of $Cone(c)$.  
\item The case $\mathfrak{j}=1$ exhibits two sub-collections. The first one is  
$$
m^{i_1|1|i_2}:\cA[-1]^{\otimes i_1}\otimes \cM[-2]\otimes \cA[-1]^{\otimes i_2}\to \cM[-2]
$$
for $i_1,i_2\ge 0$ and exhibits $\cM[-1]$ as an $A_\infty$-bimodule over $\cA$. Here we slightly deviate from the above notational convention by allowing $i_2=0$.
The second sub-collection is 
$$
\tau^{i_1|1|i_2}:\cA[-1]^{\otimes i_1}\otimes \cM[-2]\otimes \cA[-1]^{\otimes i_2}\to \cA[-1]
$$
for $i_1,i_2\ge 0$.
This describes an $\cA$-bimodule $A_\infty$-homomorphism $\cM[-1]\to \cA[-1]$, whose first term is $-c[-1]:\cM[-1]\to \cA[-1]$.  
\end{itemize}
From this perspective, the data of an $A_\infty$-triple can be equivalently rephrased as consisting of an $A_\infty$-algebra $\cA$, of an $A_\infty$-bimodule $\cM[-1]$, and of an $\cA$-bimodule $A_\infty$-homomorphism $\cM[-1]\to \cA[-1]$ whose first term is $-c[-1]:\cM[-1]\to\cA[-1]$, together with collections of maps $\{m^{i_1|j_1|\dots|i_k|j_k}\}$ and $\{\tau^{i_1|j_1|\dots|i_k|j_k}\}$ as above which extend the given data and define an $A_\infty$-structure on $Cone(c)$. 

The discussion of the Homotopy Transfer Theorem from~\S\ref{sec:Ainfty-pairs} has an obvious counterpart in this context, except that the amount of combinatorial data that one needs to track is significantly larger. For further use we spell out the corresponding notion of homotopy retract. 

\begin{definition} \label{defi:homotopy_retract_triple}
A triple $(\cM',c',\cA')$ is a \emph{homotopy retract} of a triple $(\cM,c,\cA)$ if the pair $\cA'\subset Cone(c')$ is a homotopy retract of the pair $\cA\subset Cone(c)$ in the sense of Definition~\ref{def:homotopy-retract}. 
\end{definition}

Let us write in upper triangular form the maps $P,I,H$ involved in the definition as 
$$
P=\left(\begin{array}{cc}p & \cK \\ 0 & \pi\end{array}\right),\qquad I=\left(\begin{array}{cc}i & \cH \\0 & \iota \end{array}\right),\qquad H=\left(\begin{array}{cc}h & a \\0 & -\chi\end{array}\right),
$$
so that we obtain the diagram 
\begin{equation}\label{eq:homotopy-retract-triples}
\begin{tikzcd}[sep=huge]
\arrow[loop left,"\chi"]\cM \arrow[r, shift left, "\pi"] \arrow[d,"c"'] \arrow[dr,"\cK"' near start] & \arrow[l,shift left,"\iota"] \cM' \arrow[d,"c' "] \arrow[dl,"\cH" near start] \\
\arrow[loop left,"h"]\cA \arrow[r, shift left, "p"] & \arrow[l,shift left,"i"] \cA' 
\end{tikzcd}
\end{equation}
It is straightforward to check that the homotopy retract condition is equivalent to the following:
\begin{itemize}
\item ($P$, $I$ are chain maps) 
$$
[\p,\cH]=ic'-c\iota,\qquad [\p,\cK]=pc-c'\pi. 
$$
\item ($H$ is a homotopy between $IP$ and $\one$, i.e. $\one-IP=[\p,H]$) 
$$
\one-\iota\pi = [\p_\cM,\chi], \qquad \one-ip=[\p_\cA,h], 
$$
$$
[\p,a]=c\chi - hc - i\cK - \cH\pi.
$$
\end{itemize}

\subsection{$A_2$-triples} \label{sec:arity2}

Of particular interest for us will be the operations of arity $d=2$. The previous discussion provides degree $-1$ maps
\begin{align}\label{eq:mu-mL-mR}
\tau^{2|0}=  \mu^2_\cA & :\cA[-1]\otimes \cA[-1]\to \cA[-1], \cr 
m^{1|1|0} & :\cA[-1]\otimes  \cM[-2]\to \cM[-2],\cr  
m^{0|1|1} & :\cM[-2]\otimes \cA[-1]\to \cM[-2]
\end{align}
which, after an appropriate shift discussed below, induce an algebra structure on $H(\cA)$ and also an $H(\cA)$-bimodule structure on $H(\cM)$. We also have the degree $-1$ maps  
\begin{align}\label{eq:tauR-tauL}
\tau^{1|1|0} & :\cA[-1]\otimes \cM[-2]\to \cA[-1],\cr 
\tau^{0|1|1} & :\cM[-2]\otimes \cA[-1]\to \cA[-1]
\end{align}
which, also after an appropriate shift, provide homotopies ensuring that $c$ induces in homology a bimodule map $H(\cM)\to H(\cA)$. There are two more degree $-1$ operations of arity $2$, namely 
\begin{equation}\label{eq:sigma}
m^{0|2|0} : \cM[-2]\otimes \cM[-2]\to \cM[-2]
\end{equation}
and 
\begin{equation}\label{eq:beta}
\tau^{0|2|0}:\cM[-2]\otimes \cM[-2]\to \cA[-1]. 
\end{equation}

\begin{definition} \label{defi:pre-subalgebra}
An \emph{$A_2$-triple} $(\cM,c,\cA)$ consists of an associative up to homotopy dg algebra $(\cA,\mu)$, of a chain complex $\cM$ and of a degree $0$ chain map $c:\cM\to \cA$, together with bilinear maps 
$$
m_L:\cA\otimes \cM\to \cM,\qquad m_R:\cM\otimes\cA\to \cM
$$
of degree $0$, 
$$
\tau_L:\cM\otimes \cA\to \cA,\qquad \tau_R:\cA\otimes \cM\to \cA,
$$
$$
\sigma:\cM\otimes \cM\to \cM 
$$
of degree $1$, and 
$$
\beta:\cM\otimes\cM\to \cA
$$
of degree $2$, subject to the following conditions: 
$$
[\p,\mu]=0,\qquad [\p,m_L]=0,\qquad [\p,m_R]=0,
$$
$$
[\p,\tau_L]=\mu(c\otimes\one)-cm_R,\qquad [\p,\tau_R]=\mu(\one\otimes c)-cm_L, 
$$
$$
[\p,\sigma]=m_R(\one\otimes c)-m_L(c\otimes\one),
$$
and 
$$
[\p,\beta] = -c\sigma+\tau_R(c\otimes\one)-\tau_L(\one\otimes c). 
$$
\end{definition}

The brackets are understood with respect to the indicated degrees, e.g. $[\p,m_L]= \p_\cM m_L - m_L(\p_\cA\otimes \one_\cM) - m_L(\one_\cA\otimes \p_\cM)$ and $[\p,\tau_L]=\p_A\tau_L+\tau_L(\p_\cM\otimes \one_\cA)+\tau_L(\one_\cM\otimes \p_\cA)$. 

The definition is motivated by Lemma~\ref{lem:extension-triples-from-Ainfty} and Proposition~\ref{prop:product_on_cone} below. 

\begin{lemma} \label{lem:extension-triples-from-Ainfty}
Given an $A_\infty$-triple $(\cM,c,\cA)$, the arity $2$ operations from (\ref{eq:mu-mL-mR}--\ref{eq:beta}) induce canonically the structure of an $A_2$-triple on $(\cM,c,\cA)$. 
\end{lemma}

\begin{proof} We define $\mu$, $m_L$ etc. by suitable shifts of the arity $2$ operations of the $A_\infty$-triple. As explained in Appendix~\ref{sec:shifts}, the order of successive shifts matters. Since our goal is to define an algebra structure on the cone, we first shift uniformly all the arity $2$ operations by $[1,1;1]$ so that all inputs and outputs are tensor products of $\cA$ and $\cM[-1]$; we then further shift by $+1$ on the $\cM[-1]$-factor in order to obtain arity $2$ operations whose inputs and outputs are tensor products of $\cA$ and $\cM$. This means that we define (with $\mu^2_\cA=\tau^{2|0}$) 
$$
\mu=\mu^2_\cA[1,1;1],
$$
$$
m_L=m^{1|1|0}[1,1;1][0,1;1],\qquad m_R=m^{0|1|1}[1,1;1][1,0;1],
$$
$$
\tau_R=\tau^{1|1|0}[1,1;1][0,1;0],\qquad \tau_L=\tau^{0|1|1}[1,1;1][1,0;0],
$$ 
$$
\sigma=m^{0|2|0}[1,1;1][1,1;1],\qquad \beta=\tau^{0|2|0}[1,1;1][1,1;0].
$$

For further use, it is also convenient to define the maps (of degree $0$)
$$
\umu=\mu=\mu^2_\cA[1,1;1],
$$
$$
\um_L=m^{1|1|0}[1,1;1],\qquad \um_R=m^{0|1|1}[1,1;1],
$$
$$
\utau_R=\tau^{1|1|0}[1,1;1],\qquad \utau_L=\tau^{0|1|1}[1,1;1],
$$ 
$$
\usigma=m^{0|2|0}[1,1;1],\qquad \ubeta=\tau^{0|2|0}[1,1;1].
$$

We claim that the maps $\mu,m_L,m_R,\tau_L,\tau_R,\sigma,\beta$ define the structure of an $A_2$-triple on $(\cM,c,\cA)$. Denoting $\cB=\cA\oplus \cM[-1]$, the proof consists of a direct verification by decomposing the $A_\infty$-relation 
\begin{equation} \label{eq:mu12}
\mu^1_\cB\mu^2_\cB+\mu^2_\cB(\mu^1_\cB\otimes \one) + \mu^2_\cB(\one\otimes \mu^1_\cB)=0
\end{equation}
into components. We use that $\mu^2_\cB=(\tau^{2|0}+\tau^{1|1|0}+\tau^{0|1|1}+\tau^{0|2|0},m^{1|1|0}+m^{0|1|1}+m^{0|2|0})$ and 
$$
\mu^1_\cB=\p_{Cone(c)}[-1]=\left(\begin{array}{cc}\p_{\cA[-1]}& -c[-2;-1]\\0 & \p_{\cM[-2]}\end{array}\right).
$$  
While the verification is straightforward, the signs are subtle and for this reason we give the proof in detail. 

\smallskip 

\noindent {\it Step~1. We prove that the maps $\mu=\umu,\um_L,\um_R,\utau_L,\utau_R,\usigma,\ubeta$ satisfy the relations 
$$
[\p,\mu]=0,\qquad [\p,\um_L]=0,\qquad [\p,\um_R]=0,
$$
$$
[\p,\utau_L]=\mu(c[-1;0]\otimes \one)-c[-1;0]\um_R,
$$
$$
[\p,\utau_R]=\mu(\one\otimes c[-1;0])-c[-1;0]\um_L,
$$
$$
[\p,\usigma]=\um_R(\one\otimes c[-1;0])+\um_L(c[-1;0]\otimes\one),
$$
$$
[\p,\ubeta]=-c[-1;0]\usigma +\utau_R(c[-1;0]\otimes\one)+\utau_L(\one\otimes c[-1;0]). 
$$
}

\smallskip 

The relation $[\p,\mu]=0$ follows directly from the discussion in Appendix~\ref{sec:shifts}. Since $\mu^2_\cA$ is a chain map, so is its shift. 

The relations for $\um_L$ and $\utau_R$ are obtained by restricting equation~\eqref{eq:mu12} to $\cA[-1]\otimes \cM[-2]$, where it becomes
\begin{align*}
& \mu^1_\cB(\tau^{1|1|0},m^{1|1|0}) + (\tau^{1|1|0},m^{1|1|0})(\mu^1_\cA\otimes \one) \cr 
& \qquad \qquad \qquad \qquad + (\tau^{1|1|0}+\tau^{2|0},m^{1|1|0})(\one\otimes \mu^1_\cB)=0 \cr
\Leftrightarrow \quad & ([\p,\tau^{1|1|0}]-c[-2;-1]m^{1|1|0}-\tau^{2|0}(\one\otimes c[-2;-1]),[\p,m^{1|1|0}])=0 \cr 
\Leftrightarrow \quad & [\p,\tau^{1|1|0}]-c[-2;-1]m^{1|1|0}-\tau^{2|0}(\one\otimes c[-2;-1])=0, \cr 
& [\p,m^{1|1|0}]=0 \cr
\Leftrightarrow \quad & ([\p,\tau^{1|1|0}]-c[-2;-1]m^{1|1|0}-\tau^{2|0}(\one\otimes c[-2;-1]))[1,1;1]=0, \cr
& [\p,m^{1|1|0}][1,1;1]=0 \cr
\Leftrightarrow \quad & -[\p,\utau_R]-c[-1;0]\um_L+\mu(\one\otimes c[-1;0])=0, \cr 
& -[\p,\um_L]=0.
\end{align*}

The relations for $\um_R$ and $\utau_L$ are obtained by restricting equation~\eqref{eq:mu12} to $\cM[-2]\otimes \cA[-1]$, where it becomes 
\begin{align*}
& \mu^1_\cB(\tau^{0|1|1},m^{0|1|1}) + (\tau^{0|1|1}+\tau^{2|0},m^{0|1|1})(\mu^1_\cB\otimes\one) \cr 
& \qquad \qquad \qquad \qquad + (\tau^{0|1|1},m^{0|1|1})(\one\otimes\mu^1_\cB)=0 \cr 
\Leftrightarrow \quad & ([\p,\tau^{0|1|1}]-c[-2;-1]m^{0|1|1}-\tau^{2|0}(c[-2;-1]\otimes \one),[\p,m^{0|1|1}])=0 \cr
\Leftrightarrow \quad &  [\p,\tau^{0|1|1}]-c[-2;-1]m^{0|1|1}-\tau^{2|0}(c[-2;-1]\otimes \one)=0, \cr 
& [\p,m^{0|1|1}] = 0 \cr
\Leftrightarrow \quad & ([\p,\tau^{0|1|1}]-c[-2;-1]m^{0|1|1}-\tau^{2|0}(c[-2;-1]\otimes \one))[1,1;1]=0, \cr 
& [\p,m^{0|1|1}][1,1;1]=0 \cr
\Leftrightarrow \quad & -[\p,\utau_L]-c[-1;0]\um_R+\mu(c[-1;0]\otimes\one)=0, \cr
& -[\p,\um_R]=0.
\end{align*}

The relations for $\usigma$ and $\ubeta$ are obtained by restricting equation~\eqref{eq:mu12} to $\cM[-2]\otimes\cM[-2]$, where it becomes 
\begin{align*}
& \mu^1_\cB(\tau^{0|2|0},m^{0|2|0})  +(\tau^{0|2|0}+\tau^{1|1|0},m^{0|2|0}+m^{1|1|0})(\mu^1_\cB\otimes \one) \cr
& \qquad \qquad \qquad \quad + (\tau^{0|2|0}+\tau^{0|1|1},m^{0|2|0}+m^{0|1|1})(\one\otimes\mu^1_\cB)=0 \cr
\Leftrightarrow \quad & [\p,\tau^{0|2|0}]-c[-2;-1]m^{0|2|0}- \tau^{1|1|0}(c[-2;-1]\otimes\one) \cr 
& \qquad \qquad \qquad \quad \qquad \qquad \qquad \quad - \tau^{0|1|1}(\one\otimes c[-2;-1]) = 0, \cr
& [\p,m^{0|2|0}]-m^{1|1|0}(c[-2;-1]\otimes \one) - m^{0|1|1}(\one\otimes c[-2;-1]) = 0 \cr
\Leftrightarrow \quad & ([\p,\tau^{0|2|0}]-c[-2;-1]m^{0|2|0}- \tau^{1|1|0}(c[-2;-1]\otimes\one) \cr 
& \qquad \qquad \qquad \quad \qquad \quad - \tau^{0|1|1}(\one\otimes c[-2;-1]))[1,1;1]=0 \cr
& ([\p,m^{0|2|0}]-m^{1|1|0}(c[-2;-1]\otimes \one) \cr 
& \qquad \qquad \qquad \quad \qquad \quad - m^{0|1|1}(\one\otimes c[-2;-1]))[1,1;1]=0\cr
\Leftrightarrow \quad & -[\p,\ubeta]-c[-1;0]\usigma +\utau_R(c[-1;0]\otimes \one) + \utau_L(\one\otimes c[-1;0])=0,\cr
& -[\p,\usigma]+\um_L(c[-1;0]\otimes\one)+\um_R(\one\otimes c[-1;0])=0.
\end{align*}

\smallskip 

\noindent {\it Step~2. We prove the relations for $\mu,m_L,m_R,\tau_L,\tau_R,\sigma,\beta$.} 

\smallskip 

Recall that we have 
$$
m_L=\um_L[0,1;1],\qquad m_R=\um_R[1,0;1],
$$
$$
\tau_R=\utau_R[0,1;0],\qquad \tau_L=\utau_L[1,0;0],
$$
$$
\sigma=\usigma[1,1;1],\qquad \beta=\ubeta[1,1;0].
$$

We already proved $[\p,\mu]=0$, i.e. $\mu$ is a chain map. That $m_L,m_R$ are also chain maps follows from the fact that they are shifts of the chain maps $\um_L,\um_R$. 

To derive the equation for $[\p,\tau_L]$ we use the equation for $[\p,\utau_L]$:  
\begin{align*}
& [\p,\utau_L]-\mu(c[-1;0]\otimes\one)+c[-1;0]\um_R=0 \cr
\Leftrightarrow \quad & ([\p,\utau_L]-\mu(c[-1;0]\otimes\one)+c[-1;0]\um_R)[1,0;0]=0 \cr
\Leftrightarrow \quad & [\p,\tau_L] - \mu(c\otimes \one) + cm_R=0. 
\end{align*}
In the last equivalence we use $c[-1;0]\omega_1=c$, where $\omega_1:\cM\to\cM[-1]$ is the shift. 

To derive the equation for $[\p,\tau_R]$ we use the equation for $[\p,\utau_R]$:
\begin{align*}
& [\p,\utau_R] - \mu(\one\otimes c[-1;0]) + c[-1;0]\um_L=0 \cr
\Leftrightarrow \quad & ([\p,\utau_R] - \mu(\one\otimes c[-1;0]) + c[-1;0]\um_L)[0,1;0]=0 \cr
\Leftrightarrow \quad & [\p,\tau_R]- \mu(\one\otimes c) + c m_L=0.
\end{align*}

To derive the equation for $[\p,\sigma]$ we use the equation for $[\p,\usigma]$: 
\begin{align*}
& [\p,\usigma] - \um_R(\one\otimes c[-1;0]) - \um_L(c[-1;0]\otimes\one) = 0 \cr
\Leftrightarrow \quad & ([\p,\usigma] - \um_R(\one\otimes c[-1;0]) - \um_L(c[-1;0]\otimes\one))[1,1;1]=0 \cr
\Leftrightarrow \quad & -[\p,\sigma] + m_R(\one\otimes c) -m_L(c\otimes\one)=0.
\end{align*}

Finally, to derive the equation for $[\p,\beta]$ we use the equation for $[\p,\ubeta]$: 
\begin{align*} 
& [\p,\ubeta]+c[-1;0]\usigma -\utau_R(c[-1;0]\otimes\one)-\utau_L(\one\otimes c[-1;0]) = 0 \cr
\Leftrightarrow \quad & ([\p,\ubeta]+c[-1;0]\usigma -\utau_R(c[-1;0]\otimes\one) \cr 
& \qquad \qquad \qquad \qquad \qquad \qquad -\utau_L(\one\otimes c[-1;0]))[1,1;0]=0 \cr
\Leftrightarrow \quad & [\p,\beta] + c\sigma -\tau_R(c\otimes \one) + \tau_L(\one\otimes c)=0.
\end{align*}
\end{proof}

In the next statement we denote an element of $Cone(c)=\cA\oplus \cM[-1]$ by $(a,\bar x)$, meaning that $\bar x\in\cM[-1]$ is the shift of an element $x\in\cM$. In particular $|\bar x|=|x|+1$.

\begin{proposition} \label{prop:product_on_cone}
Let $(\cM,c,\cA)$ be an $A_2$-triple with operations denoted $\mu,m_L,m_R,\tau_L,\tau_R,\sigma,\beta$ as above. The formula 
\begin{eqnarray} \label{eq:product_on_cone}
\lefteqn{m \big((a,\bar x),(a',\bar x')\big)} \\
 & = \big( \mu(a,a')+(-1)^{|a|}\tau_R(a,x') + \tau_L(x,a') - (-1)^{|\bar x|}\beta(x,x'), \nonumber \\
 &  \qquad \qquad  \qquad \qquad (-1)^{|a|}\overline{m_L(a,x')}+\overline{m_R(x,a')}-(-1)^{|\bar x|}\overline{\sigma(x,x')}\big) \nonumber
\end{eqnarray}
defines a degree $0$ bilinear product 
$$
m:Cone(c)\otimes Cone(c)\to Cone(c)
$$ 
which is a chain map. This product coincides with the one from~\eqref{eq:dga-from-Ainfty} if the $A_2$-triple is induced from an $A_\infty$-triple as in Lemma~\ref{lem:extension-triples-from-Ainfty}.
\end{proposition}

\begin{proof}
That the formula defines a chain map can be checked directly. 

Assume now that the $A_2$-triple is induced from an $A_\infty$-triple. The product induced on $Cone(c)$ by the $A_\infty$-structure is 
$$
m =(\mu+\utau_R+\utau_L+\ubeta,\um_L+\um_R+\usigma).
$$
We therefore merely need to express the maps $\utau_R,\utau_L,\ubeta, \um_L,\um_R,\usigma$ in terms of $\tau_R,\tau_L,\beta,m_L,m_R,\sigma$. 

We work out in full detail the case of $\usigma$. We claim that 
\begin{equation} \label{eq:shift-sigma}
\usigma=-\sigma[-1,-1;-1],
\end{equation}
$$
\usigma(\bar x,\bar x')=-(-1)^{|\bar x|}\overline{\sigma(x,x')}.
$$
Indeed, we start with $\sigma=\usigma[1,1;1]$ and then compute $\sigma[-1,-1;-1]=\usigma[1,1;1][-1,-1;-1]=-\usigma$. The explicit formula in terms of elements is a consequence of the definition of the shift by $[-1,-1;-1]$. 

Similarly, we compute:
\begin{itemize}
\item 
\begin{equation} \label{eq:shift-mL}
\um_L=m_L[0,-1;-1], 
\end{equation}
$$
\um_L(a,\bar x')=(-1)^{|a|}\overline{m_L(a,x')}
$$
since $m_L[0,-1;-1]=\um_L[0,1;1][0,-1;-1]=\um_L$. 
\item 
\begin{equation}\label{eq:shift-mR}
\um_R=m_R[-1,0;-1],
\end{equation}
$$
\um_R(\bar x,a')=\overline{m_R(x,a')}
$$
since $m_R[-1,0;-1]=\um_R[1,0;1][-1,0;-1]=\um_R$.
\item 
\begin{equation}\label{eq:shift-tauL}
\utau_L=\tau_L[-1,0;0],
\end{equation} 
$$
\utau_L(\bar x,a')=\tau_L(x,a')
$$
since $\tau_L[-1,0;0]=\utau_L[1,0;0][-1,0;0]=\utau_L$. 
\item 
\begin{equation}\label{eq:shift-tauR}
\utau_R=\tau_R[0,-1;0],
\end{equation}
$$
\utau_R(a,\bar x')=(-1)^{|a|}\tau_R(a,x')
$$
since $\tau_R[0,-1;0]=\utau_R[0,1;0][0,-1;0]=\utau_R$.
\item 
$$
\ubeta=-\beta[-1,-1;0],
$$
$$
\ubeta(\bar x,\bar x')=-(-1)^{|\bar x|}\beta(x,x')
$$
since $\beta[-1,-1;0]=\ubeta[1,1;0][-1,-1;0]=-\ubeta$. 
\end{itemize}

\end{proof}

\begin{remark} We chose to infer the equations for the maps $m_L$, $m_R$, $\tau_R$, $\tau_L$, $\sigma$, $\beta$ in Definition~\ref{defi:pre-subalgebra} from the $A_\infty$-equations. As such, they assemble canonically into the product structure on the cone induced from the $A_\infty$-structure. But even so, there is a small amount of choice involved: we could have defined $\sigma$ as $m^{0|2|0}[2,2;2]$, which would have changed its sign. We settled for our convention for the reasons mentioned in the preamble of the proof of Lemma~\ref{lem:extension-triples-from-Ainfty}. 

The existence of this potential change of sign can also be understood from the following perspective. Assume one wishes to determine equations for such a collection of maps so that they assemble into \emph{some} product structure on the cone. A close inspection of the formula defining the product $m$ shows that, once we require that it restricts to the product $\mu$ on $\cA$, the equations for $m_L$, $m_R$, $\tau_R$, $\tau_L$, $\sigma$, $\beta$ are uniquely determined up to obvious multiplications by $\pm 1$ by the requirement that the equation expressing the compatibility with the differential, i.e. $[\p_{Cone(c)},m]=0$, translates into \emph{functional} equations for the various maps involved. Our procedure to define the maps from the $A_\infty$-structure can be seen as one convenient way to fix the signs. 
\end{remark}

We call the product $m$ the \emph{canonical product on the cone defined by the $A_2$-structure}. Associativity up to homotopy for the product $m$ is not a priori guaranteed. For this, one would need to enhance the data of an $A_2$-triple precisely with the operations of arity $3$ involved in the definition of an $A_\infty$-triple, see~\eqref{eq:associator-from-Ainfty}. 
  
The $A_2$-triples used in this paper will always be arity $2$ restrictions of genuine $A_\infty$-triples canonically defined up to homotopy. While we will not construct nor make use of the full $A_\infty$-structure, it is important to acknowledge its existence. In particular, the homotopy transfer and homotopy invariance statements for $A_2$-triples discussed below are avatars of the homotopy transfer and homotopy invariance statements for $A_\infty$-structures. 
    
\begin{proposition} \label{prop:arity2homotopyinvariance}
Let $(\cM',c',\cA')$ be a triple which is a homotopy retract of the triple $(\cM,c,\cA)$ as in Definition~\ref{defi:homotopy_retract_triple}. 
Given the structure of an $A_2$-triple on $(\cM,c,\cA)$, there is an induced structure of an $A_2$-triple on $(\cM',c',\cA')$ such that the maps 
$$
\xymatrix{
Cone(c) \ar@<.5ex>[rr]^P && Cone(c') \ar@<.5ex>[ll]^I
}
$$
involved in the homotopy retract are compatible with the products on $Cone(c)$ and $Cone(c')$. \qed
\end{proposition}  
  
Very explicitly, and because we are in arity 2, the transfer of structure is obtained by summing over the different ways of labeling the inputs and output of the unique binary rooted tree with two leaves by $\cA$ and $\cM$, and inserting accordingly at the inputs the maps $i,\iota, \cH$, and at the output the maps $p,\pi,\cK$ (compare to~\cite{Markl-transfer}). For example, the map $\sigma':\cM'\otimes \cM'\to \cM'$ is given by 
$$
\sigma'=\pm \pi\sigma \iota\otimes \iota \pm \pi m_L \cH\otimes \iota \pm \pi m_R \iota\otimes \cH.
$$  
See Figure~\ref{fig:transfer-sigma}, in which we see 3 different labelings contributing to the transfer. In contrast, the map $\mu':\cA'\otimes\cA'\to\cA'$ is given by $\mu'=p\mu i\otimes i$ since there is only one labeling contributing to the transfer because of the upper triangular form of the maps $P$ and $I$. The formula for the transferred map $\beta'$ involves 7 terms. 

\begin{figure} [ht]
\centering
\input{transfer-sigma.pstex_t}
\caption{Trees for the transferred map $\sigma'$.}
\label{fig:transfer-sigma}
\end{figure} 

In the situation in which the maps $P,I$ defining a homotopy retract are homotopy inverses, with the homotopy $H':\cB'\to\cB'$ such that $[\p,H']=\one-PI$ being also in upper triangular form, we obtain in particular that the (non-associative) algebras $(Cone(c),m)$ and $(Cone(c'),m')$ are homotopy equivalent. In contrast to Proposition~\ref{prop:transfer-homotopy-equivalence} above, this fact is immediate and needs no additional assumption (because we do not ask for any higher compatibilities of the homotopies). Thus, homotopy invariance in the context of $A_2$-triples is automatic.

\subsection{Morphisms of $A_2$-triples}

Consider two $A_2$-triples $(\cM,c,\cA)$ and $(\cM',c',\cA')$ with
associated operations $(m_L,m_R,\tau_L,\tau_R,\sigma,\beta)$ and
$(m_L',m_R',\tau_L',\tau_R',\sigma',\beta')$, respectively.   

\begin{definition}\label{defi:pre-subalgebra-mor}
A {\em special morphism of $A_2$-triples}
$$
(g,f):(\cM,c,\cA)\to(\cM',c',\cA')
$$ 
consists of the following maps:

(1) degree $0$ chain maps $g:\cM\to\cM'$ and $f:\cA\to\cA'$ satisfying
$$
   fc=c'g;
$$
(2) degree $1$ bilinear maps 
$$
  \wh\mu:\cA\otimes\cA\to\cA',\quad
  \wh m_L:\cA\otimes \cM\to \cM',\quad
  \wh m_R:\cM\otimes\cA\to \cM'
$$
satisfying
\begin{gather*}
  [\p,\wh\mu] = \mu'(f\otimes f)-f\mu,\cr
  [\p,\wh m_L] = m_L'(f\otimes g)-gm_L,\qquad
  [\p,\wh m_R] = m_R'(g\otimes f)-gm_R;
\end{gather*}
(3) degree $2$ bilinear maps $\wh\sigma:\cM\otimes \cM\to \cM'$ satisfying
$$
  [\p,\wh\sigma] = \sigma'(g\otimes g)-g\sigma-\wh m_R(1\otimes
c)+\wh m_L(c\otimes 1)
$$
and
$$
  \wh\tau_L:\cM\otimes \cA\to \cA',\qquad \wh\tau_R:\cA\otimes \cM\to \cA'
$$
satisfying
\begin{gather*}
  [\p,\wh\tau_L] = \tau_L'(g\otimes f)-f\tau_L-\wh\mu(c\otimes
  1)+c'\wh m_R,\cr
  [\p,\wh\tau_R] = \tau_R'(f\otimes g)-f\tau_R-\wh\mu(1\otimes
  c)+c'\wh m_L;
\end{gather*}
(4) a degree $3$ bilinear map $\wh\beta:\cM\otimes\cM\to \cA'$ satisfying
$$
  [\p,\wh\beta] = \beta'(g\otimes g)-f\beta+c'\wh\sigma
  -\wh\tau_R(c\otimes 1)+\wh\tau_L(1\otimes c).
$$
\end{definition}

\begin{remark} 
We call a morphism of $A_2$-triples \emph{special} because we require the equality $fc=c'g$ to be satisfied strictly. More generally, one can relax this equality up to homotopy, in which case the rest of the identities need to be suitably corrected. Special morphisms of $A_2$-triples are enough for our applications to Rabinowitz Floer homology and we therefore limit ourselves to this case. 
\end{remark}

The next result explains the genesis of this definition. The $A_\infty$-morphisms from the statement could also be called \emph{special}. 

\begin{proposition} \label{prop:morphisms-fromAinfty-to-A2}
Let $\cF:(\cM,c,\cA)\to (\cM',c',\cA')$ be an\break $A_\infty$-morphism between $A_\infty$-triples such that $\cF^1$ has no $(\cM[-1];\cA')$-component, i.e. it has diagonal form
$$
\cF^1=\begin{pmatrix} F^1 & 0 \\ 0 & G^1 \end{pmatrix}.
$$
The arity $1$ and $2$ components of $\cF$ determine canonically a special morphism between the $A_2$-triples $(\cM,c,\cA)$ and $(\cM',c',\cA')$.  
\end{proposition}

\begin{proof} The proof goes by direct verification using explicit formulas for the various components, expressed in terms of shifts as in Appendix~\ref{sec:shifts}. 

Firstly, $\cF^1$ acts as $Cone(c)[-1]=\cA[-1]\oplus\cM[-2]\to Cone(c')[-1]=\cA'[-1]\oplus \cM'[-2]$, with components $F^1:\cA[-1]\to\cA'[-1]$ and $G^1:\cM[-2]\to\cM'[-2]$. We define  
$$
f=F^1[1;1],\qquad g=G^1[2;2].
$$

Secondly, $\cF^2$ acts as $Cone(c)[-1]^{\otimes 2}\to Cone(c')[-1]$ and we denote its components $\cF^2_{\cA[-1]\otimes\cA[-1];\cA'[-1]}$, $\cF^2_{\cM[-2]\otimes \cA[-1];\cM'[-2]}$ etc. We define 
$$
\wh \mu= \cF^2_{\cA[-1]\otimes\cA[-1];\cA'[-1]} [1,1;1],  
$$
$$
\wh m_L= - \cF^2_{\cA[-1]\otimes\cM[-2];\cM'[-2]} [1,1;1][0,1;1], 
$$
$$ 
\wh m_R = - \cF^2_{\cM[-2]\otimes\cA[-1];\cM'[-2]} [1,1;1][1,0;1],
$$
$$
\wh \sigma = - \cF^2_{\cM[-2]\otimes\cM[-2];\cM'[-2]} [1,1;1][1,1;1],
$$
$$
\wh \tau_L =  \cF^2_{\cM[-2]\otimes\cA[-1];\cA'[-1]} [1,1;1][1,0;0], 
$$
$$ 
\wh \tau_R = \cF^2_{\cA[-1]\otimes\cM[-2];\cA'[-1]} [1,1;1][0,1;0],
$$
$$
\wh \beta = \cF^2_{\cM[-2]\otimes\cM[-2];\cA'[-1]} [1,1;1][1,1;0].
$$

Recalling the definition of the $A_2$-triple structures $(\mu,m_L,m_R,\tau_L,\tau_R,\break \sigma,\beta)$, resp. $(\mu', m_L',m_R',\tau_L',\tau_R',\sigma',\beta')$ from the proof of Lemma~\ref{lem:extension-triples-from-Ainfty}, one then checks directly that the previous formulas indeed give rise to a special $A_2$-morphism. Since the calculations are very similar to those in the proof of Lemma~\ref{lem:extension-triples-from-Ainfty} we omit the details. 
\end{proof}

\begin{proposition}\label{prop:A2mor-to-cone}
A special morphism of $A_2$-triples\break $(g,f):(\cM,c,\cA)\to(\cM',c',\cA')$
induces a degree $0$ chain map
$$
   Cone(g,f):Cone(c)\to Cone(c')
$$
intertwining the products $m,m'$ defined in Proposition~\ref{prop:product_on_cone} up to chain homotopy.
The induced map on homology fits into the commuting diagram with exact sequences
$$
\xymatrix{
   \cdots H_*(\cM) \ar[r]^{c_*} \ar[d]^{g_*} & H_*(\cA) \ar[r] \ar[d]^{f_*} & H_*(Cone(c)) \ar[r] \ar[d]^{Cone(g,f)_*} & H_{*-1}(\cM) \ar[d]^{g_*} \cdots \\
   \cdots H_*(\cM') \ar[r]^{c'_*} & H_*(\cA') \ar[r] & H_*(Cone(c')) \ar[r] & H_{*-1}(\cM') \cdots
}
$$
If $f$ and $g$ induce isomorphisms on homology, then so does $Cone(g,f)$. 
\end{proposition}

\begin{proof} The proposition is proved by inverting the formulas given in the proof of Proposition~\ref{prop:morphisms-fromAinfty-to-A2}. Indeed, the arity $1$ and $2$ components of an $A_\infty$-morphism can be recovered from the components of an $A_2$-morphism, and the first two equations defining an $A_\infty$-morphism $\cF:Cone(c)\to Cone(c')$ are equivalent to the first assertion of the proposition: $\cF^1:Cone(c)[-1]\to Cone(c')[-1]$ is a chain map, and it intertwines the products $\mu^2_{Cone(c)}$ and $\mu^2_{Cone(c')}$ up to a homotopy given by $\cF^2$, i.e. 
$$
\cF^1 \mu^2_{Cone(c)} - \mu^2_{Cone(c')}(\cF^1\otimes \cF^1) = [\p,\cF^2]. 
$$
The second assertion follows by passing to homology in the commuting diagram of short exact sequences
$$
\xymatrix{
   0 \ar[r] & \cA \ar[r] \ar[d]^f & Cone(c) \ar[r] \ar[d]^{Cone(g,f)} & \cM[-1] \ar[r] \ar[d]^g & 0 \\
   0 \ar[r] & \cA' \ar[r] & Cone(c') \ar[r] & \cM'[-1] \ar[r] & 0
}
$$
The last assertion follows from the second one by the 5-lemma.
\end{proof}

\subsection{Examples}

\begin{example}[{\bf Ideals}] Let $(\cA,\p_\cA,\mu)$ be a dga and $\cM\subset \cA$ be a dg ideal. Let $c=incl:\cM\hookrightarrow \cA$ be the inclusion. The triple $(\cM,incl,\cA)$ has a canonical structure of an $A_2$-triple defined as follows: the operations $m_L:\cA\otimes \cM\to\cM$ and $m_R:\cM\otimes \cA\to\cM$ are given by multiplication in $\cA$ and endow $\cM$ with the structure of a strict $\cA$-bimodule, whereas the operations $\tau_R,\tau_L,\sigma,\beta$ are all zero. The corresponding product $m$ on $Cone(incl)=\cA\oplus \cM[-1]$ is given by
\footnote{In~\cite[Lemma~2.1]{Herzog-Takayama} the authors call it ``Nagata product". However, this terminology is potentially misleading. Nagata defined in~\cite[p.~2]{Nagata} a product on the direct sum between a module and its base ring in the context of his general procedure of ``idealization", i.e. of turning a module into an ideal, which signed the beginning of the theory of square zero extensions. In the dg setting, a square zero extension is a \emph{surjective} map of dga's whose kernel squares to zero. In contrast, our setup is concerned with \emph{injective} maps of dga's, i.e. with pairs consisting of an algebra and a subalgebra. In the current setup, if $\cM^2$ were zero then $\cA$ would be a square zero extension of $\cA/\cM$, inheriting a ``Nagata product". In contrast, the cone of the inclusion $\cM\hookrightarrow \cA$ always carries a dga structure.}
$$
m\big((a,\bar x),(a',\bar x')\big)=\big(\mu(a,a'),(-1)^{|a|}\overline{\mu(a,x')}+\overline{\mu(x,a')}\big). 
$$ 
The projection 
$$
proj:Cone(incl)\to \cA/\cM, \qquad (a,x)\mapsto [a]
$$
is clearly a ring map. Assuming that the short exact sequence of $R$-modules $0\to\cM\to \cA\to \cA/\cM\to 0$ is split, it is a general fact that $proj$ is a homotopy equivalence with an explicit homotopy inverse defined in terms of the splitting (see for example~\cite[Lemma~4.3]{CO}). 
\end{example} 
 
\begin{example}[{\bf Quotients}]  \label{example:quotients}
Let now $(\cM,c,\cA)$ be an $A_2$-triple with operations $\mu,m_L,m_R,\tau_L,\tau_R,\sigma,\beta$. Let $m$ be the corresponding product structure on $Cone(c)$. 

Assume $c:\cM\to \cA$ to be surjective, denote $K=\ker\, c$ and assume that the short exact sequence $0\to K \to \cM\stackrel c \to \cA\to 0$ is split. Writing the differential $\p_\cM$ in upper triangular form with respect to the splitting, denote $f:\cA\to K[-1]$ its $(\cA,K)$-component. We then have a homotopy equivalence  
$$
\begin{tikzcd}[cramped]
\arrow[out=200,in=160,looseness=15, "H"] Cone(c)=\cA\oplus K[-1]\oplus \cA[-1] \arrow[r, shift left, "\Sigma"] & \arrow[l,shift left,"T"] K[-1],
\end{tikzcd}
$$
where $T$ is the inclusion on the $K[-1]$-factor, $\Sigma(a,\bar k,\bar a)=f(a)+\bar k$, and $\Sigma T =\one$. The homotopy $H$ acts by $H(a,\bar k,\bar a)=(0,0,a)$. 
The shifted kernel $K[-1]$ inherits the product structure $\tilde \sigma =\Sigma m T\otimes T$. Explicitly 
$$
\tilde \sigma(\bar x_K,\bar x'_K)=-(-1)^{|\bar x_K|}f\beta(x_K,x'_K)-(-1)^{|\bar x_K|}\overline{pr_K\sigma(x_K,x'_K)},
$$
where $pr_K:\cM\to K$ is the projection determined by the splitting. 

In the presence of arity 3 data on the triple $(\cM,c,\cA)$, this product is associative up to homotopy and the maps $\Sigma$, $T$ interchange $\tilde \sigma$ and $m$ up to homotopy. 

Under the additional assumptions 
$$
\beta|_{K\otimes K}=0 \qquad \mbox{and} \qquad \sigma(K\otimes K)\subset K,
$$
we have 
$$
\tilde \sigma(\bar x_K,\bar x'_K)=-(-1)^{|\bar x_K|}\overline{\sigma(x_K,x'_K)}, 
$$ 
i.e. 
$$
\tilde\sigma = -\sigma[-1,-1;-1].  
$$
In case the $A_2$-triple is induced by an $A_\infty$-structure as before, we obtain $\tilde \sigma = \underline \sigma$, the product on $\cM[-1]$ induced by the $A_\infty$-operations. 

In any case, the map $T$ is an algebra map (and a homotopy equivalence), so that $(Cone(c),m)$ and $(K[-1],\tilde \sigma)$ are homotopy equivalent as algebras. 
\end{example}
 
\begin{example}[{\bf Duality $\cM=\cA^\vee$}] \label{example:duality}
Let $(\cA,\p)$ be a dg $R$-module and $\cA^\vee_*=\mathrm{Hom}_R(\cA_{-*},R)$ its graded dual. In~\S\ref{sec:A2+} we describe a structure on $\cA$, called an \emph{$A_2^+$-structure}, which gives rise to an $A_2$-triple $(\cA^\vee,c,\cA)$. It encodes all the information of the $A_2$-triple in terms of operations on $\cA$, which will be important in~\cite{CHO-MorseFloerGH} to relate Rabinowitz Floer homology of a unit cotangent bundle to Rabinowitz loop homology. We study in~\S\ref{sec:A2+} the functoriality of $A_2^+$-structures and prove an algebraic counterpart of our main Poincar\'e duality theorem for the cones of the associated $A_2$-triples. 
\end{example}

\section{Cones of Floer continuation maps}  \label{sec:cones_continuation}

Floer continuation maps naturally give rise to $A_2$-triples, and indeed to $A_\infty$-triples. 

\subsection{Closed symplectically aspherical manifolds} Although main\-ly interested in the noncompact case, we start with a discussion of the compact symplectically aspherical case. In this situation the continuation maps are homotopy equivalences and their cones are acyclic, but it is interesting to see how the $A_2$-triple arises. We work on a closed symplectically aspherical manifold $W$ of dimension $2n$ with trivial first Chern class and a fixed choice of trivialization of its canonical bundle. All our Floer chain complexes are graded by the Conley-Zehnder index, computed in this trivialization. Whenever we write Floer chain complexes, continuation maps, and more generally equations for pseudo-holomorphic maps defined on Riemann surfaces, we tacitly mean that there are choices of compatible almost complex structures involved. In order not to burden the notation we will not make further reference to these unless absolutely necessary. 

\begin{proposition} \label{prop:Hamiltonian-pre-subalgebra}
Let 
$$
\cM_*=FC_*(H)[n], \qquad \cA_*=FC_*(K)[n]
$$ 
be shifted Floer chain complexes determined by nondegenerate Hamiltonians $H,K$ on a closed symplectically aspherical manifold $W$ of dimension $2n$. Recall that $FC_*(H)[n]=FC_{*+n}(H)$. Let $\{H_s\, : \, s\in\R\}$ be a homotopy such that $H_s=K$ for $s\ll 0$ and $H_s=H$ for $s\gg 0$, and denote 
$$
c:\cM_*\to \cA_*
$$
the corresponding continuation map. 
Then $(\cM,c,\cA)$ carries the structure of an $A_2$-triple, canonically defined up to homotopy.  
\end{proposition}

\begin{proof} We need to define operations $\mu$, $m_L$, $m_R$, $\tau_L$, $\tau_R$, $\sigma$ and $\beta$. 

The operations $\mu:\cA\otimes \cA\to \cA$, $m_L:\cA\otimes \cM\to \cM$, $m_R:\cM\otimes \cA\to \cM$ are pair-of-pants products. They are defined by counts of index $0$ pairs-of-pants with 2 inputs (positive punctures) and 1 output (negative puncture). See Figure~\ref{fig:mu-mL-mR}, in which we depict pairs-of-pants with two inputs and one output schematically as binary rooted trees with two leaves. The two inputs are ordered, the first one is depicted on the left and the second one is depicted on the right. 

\begin{figure} [ht]
\centering
\input{mu-mL-mR.pstex_t}
\caption{Curves defining the maps $\mu$, $m_L$, $m_R$ in Floer theory.}
\label{fig:mu-mL-mR}
\end{figure} 

The operations $\sigma:\cM\otimes \cM\to \cM$, $\tau_L:\cM\otimes \cA\to \cA$, $\tau_R:\cA\otimes \cM\to \cA$ are defined by counts of index $-1$ pairs-of-pants with 2 positive punctures and 1 negative puncture in suitable $1$-dimensional families parametrized by the interval $[0,1]$. See Figure~\ref{fig:sigma-tauL-tauR}. 

\begin{figure} [ht]
\centering
\input{sigma-tauL-tauR.pstex_t}
\caption{Curves defining the maps $\sigma$, $\tau_L$, $\tau_R$ in Floer theory.}
\label{fig:sigma-tauL-tauR}
\end{figure} 

The operation $\beta:\cM\otimes\cM\to\cA$ is defined by the count of index $-2$ pairs-of-pants with 2 positive punctures and 1 negative puncture in a 2-dimensional family parametrized by the 2-simplex. See Figure~\ref{fig:beta}. 

\begin{figure} [ht]
\centering
\input{beta.pstex_t}
\caption{Curves defining the map $\beta$ in Floer theory.}
\label{fig:beta}
\end{figure} 

The proof of the relations defining an $A_2$-triple is straightforward. That the resulting structure is canonically defined up to homotopy is also a straightforward---though combinatorially involved---argument. 
\end{proof}

\begin{remark}[$A_\infty$-triple]
The above $A_2$-triple is the arity 2 part of an $A_\infty$-triple. However, we will not construct the latter here. 
\end{remark}

\begin{remark}[filtration on the cone]
Assume $H\le K$ and the homotopy is monotone. Then all the maps defining the $A_2$-triple can be constructed such that they decrease the action. The cone then has a canonical $\R$-filtration 
$$
Cone(c)^{\le a} = FC_*^{\le a}(K)[n]\oplus FC_{*-1}^{\le a}(H)[n],\qquad a\in\R
$$
and the product structure $m$ defined by the $A_2$-triple preserves this filtration, meaning that 
$$
m\big(Cone(c)^{\le a}\otimes Cone(c)^{\le b}\big)\subseteq Cone(c)^{\le a+b}.
$$
Denote $Cone(c)^{(a,b)}=Cone(c)^{<b}/Cone(c)^{\le a}$. We obtain induced partial products 
$$
m: Cone(c)^{(a,b)}\otimes Cone(c)^{(a',b')}\to Cone(c)^{(\max\{a+b',a'+b\},b+b')}.
$$
\end{remark}

\subsection{Completions of Liouville domains} Consider now the noncom\-pact case where the underlying manifold is the symplectic completion $\wh W$ of a Liouville domain $W$. One additional complication is added by the fact that solutions of the relevant Cauchy-Riemann equations with $0$-order Hamiltonian perturbation are required to obey a maximum principle. More precisely, given a map $u:\Sigma\to \wh W$ defined on a (punctured) Riemann surface $\Sigma$ and solving the perturbed Cauchy-Riemann equation  
$$
(du-X_H\otimes \beta)^{0,1}=0,
$$
where $\beta\in\Omega^1(\Sigma,\R)$ and $H=\{H_z\}$, $z\in\Sigma$ is a $\Sigma$-dependent family of Hamiltonians, one requires that, outside a compact set, we have $d_z(H\beta)\le 0$ for all $z\in\Sigma$. When restricting to admissible Hamiltonians, i.e. Hamiltonians which are linear in the region $\{r\ge 1\}\subset \wh W$, it is enough to impose this condition in that region. If this condition holds on the entire completion $\wh W$, then in addition the relevant maps decrease the action.

The resulting structure is the following. 
Given admissible Hamiltonians $H_1\le H_2$ and $H'_1\le H'_2$ together with monotone homotopies $\{H_s\}$ and $\{H'_s\}$ connecting $H_1$ and $H_2$, respectively $H'_1$ and $H'_2$, denote $c:FC_*(H_1)[n]\to FC_*(H_2)[n]$ and $c':FC_*(H'_1)[n]\to FC_*(H'_2)[n]$ the corresponding continuation maps. Denote $H_1\#H_2(t,x)=H_1(t,x)+H_2(t,(\varphi_{H_1}^t)^{-1}(x))$, where $\phi_H^t$ denotes the Hamiltonian flow of $H$. Assume further the conditions
\begin{equation}\label{eq:conditions-continuation}
H'_1\ge H_1\#H_2,\qquad H'_2\ge 2H_2=H_2\#H_2. 
\end{equation}
The equality $2H_2=H_2\# H_2$ is a condition which holds for time-independent Hamiltonians, hence for the Hamiltonians that we use in the sequel. By modifying the moduli spaces considered in the proof of Proposition~\ref{prop:Hamiltonian-pre-subalgebra} into moduli spaces with inputs $1$-periodic orbits of $H_1, H_2$ and outputs $1$-periodic orbits of $H'_1,H'_2$, we construct operations $\mu$, $m_L$, $m_R$, $\tau_L$, $\tau_R$, $\sigma$, $\beta$ which assemble into a product
$$
m : Cone(c)\otimes Cone(c)\to Cone(c'). 
$$
Moreover, this product respects the canonical filtrations on the cones and induces 
$$
m : Cone(c)^{(a,b)}\otimes Cone(c)^{(a',b')}\to Cone(c')^{(\max\{a+b',a'+b\},b+b')}.
$$
Thus, in the noncompact case there is strictly speaking no pre-subal\-ge\-bra extension structure for a fixed pair of Hamiltonians. However, there is a ``directed system" of such, constructed as above. 
 
\begin{remark}
That conditions~\eqref{eq:conditions-continuation} are indeed sufficient for the existence of continuation maps is a consequence of the following construction, originally due to Matthias Schwarz~\cite[Proposition~4.1 sqq.]{Schwarz}. Given Hamiltonians $H$ and $K$, it is possible to construct a perturbed Floer equation on a pair-of-pants with two positive punctures and one negative puncture such that: near the positive punctures it specializes to the Floer equation for $H$ and $K$, near the negative puncture it specializes to the Floer equation for the Hamiltonian $H\#K(t,x)=H(t,x)+K(t,(\varphi_H^t)^{-1}x)$, and the solutions of this Floer equation satisfy the maximum principle and do not increase the Hamiltonian action. 
\end{remark} 
 
\begin{example}[Rabinowitz Floer homology] \label{example:Klambda} Given $\lambda\in\R$ denote by $K_\lambda$ the Hamiltonian which is $0$ on $W$ and is linear of slope $\lambda$ on $\{r\ge 1\}$, with a convex smoothing if $\lambda>0$, respectively a concave smoothing if $\lambda<0$. See Figure~\ref{fig:Klambda}. 

\begin{figure} [ht]
\centering
\input{Klambda.pstex_t}
\caption{The family of Hamiltonians $\{K_\lambda\}$, $\lambda\in\R$.}
\label{fig:Klambda}
\end{figure}

Given parameters $\lambda_-\le \lambda_+$ we denote 
$$
c_{\lambda_-,\lambda_+}:FC_*(K_{\lambda_-})[n]\to FC_*(K_{\lambda_+})[n]
$$
the continuation map induced by a monotone homotopy. Given also parameters $\lambda'_-\le \lambda'_+$ such that $\lambda_-+\lambda_+\le \lambda'_-$ and  $2\lambda_+\le \lambda'_+$, and given action bounds $-\infty<a< b<\infty$, we obtain bilinear maps 
\begin{equation}\label{eq:mtilde-conec}
m : Cone(c_{\lambda_-,\lambda_+})^{(a,b)}\otimes Cone(c_{\lambda_-,\lambda_+})^{(a,b)}\to Cone(c_{\lambda'_-,\lambda'_+})^{(a+b,2b)}.
\end{equation}

Let now $(a,b)$ be fixed. Given $\lambda_{- -}\le \lambda_-\le \lambda_+$ we have a canonical continuation map $cont:Cone(c_{\lambda_{- -},\lambda_+})\to Cone(c_{\lambda_-,\lambda_+})$. This map is compatible in homology with the bilinear maps $m$ defined above, meaning that we have commutative diagrams for allowable values of the parameters 

{\scriptsize{
$$
\xymatrix{
H_*(Cone(c_{\lambda_-,\lambda_+})^{(a,b)})\otimes H_*(Cone(c_{\lambda_-,\lambda_+})^{(a,b)}) \ar[r]^-{m} & H_*(Cone(c_{\lambda'_-,\lambda'_+})^{(a+b,2b)}) \\
H_*(Cone(c_{\lambda_{- -},\lambda_+})^{(a,b)})\otimes H_*(Cone(c_{\lambda_{- -},\lambda_+})^{(a,b)}) \ar[r]^-{m} \ar[u]^{cont\otimes cont} & H_*(Cone(c_{\lambda'_{- -},\lambda'_+})^{(a+b,2b)}) \ar[u]_{cont}
}
$$
}}

Similarly, given $\lambda_-\le \lambda_+\le \lambda_{++}$ we have a canonical continuation map $cont:Cone(c_{\lambda_-,\lambda_+})\to Cone(C_{\lambda_-,\lambda_{++}})$ which is also compatible in homology with the bilinear products $m$, with a similar meaning. 

Note moreover that the canonical maps 
$$
\lim\limits_{\stackrel\longleftarrow{\lambda_-\to-\infty}} H_*(Cone(c_{\lambda_-,\lambda_+})^{(a,b)})\stackrel \simeq \longrightarrow H_*(Cone(c_{a,\lambda_+})^{(a,b)})
$$ 
are isomorphisms for all $a\le \lambda_+$, and similarly the maps 
$$
H_*(Cone(c_{\lambda_-,b})^{(a,b)})\stackrel\simeq\longrightarrow \lim\limits_{\stackrel\longrightarrow{\lambda_+\to\infty}}H_*(Cone(c_{\lambda_-,\lambda_+})^{(a,b)})
$$
are isomorphisms for all $\lambda_-\le b$. 

Define\footnote{Recall that throughout this paper $S\H_*=SH_{*+n}$ denotes the degree shifted version of symplectic homology.}
$$
S\H_*^{(a,b)}(\{K_\lambda\}) := \lim\limits_{\stackrel\longrightarrow{\lambda_+\to\infty}} \lim\limits_{\stackrel\longleftarrow{\lambda_-\to-\infty}} H_*(Cone(c_{\lambda_-,\lambda_+})^{(a,b)}). 
$$
We obtain in the first-inverse-then-direct limit bilinear maps $m$ acting as  
$$
S\H_*^{(a,b)}(\{K_\lambda\})\otimes S\H_*^{(a,b)}(\{K_\lambda\})\to S\H_*^{(a+b,2b)}(\{K_\lambda\}).
$$

These maps are compatible with the canonical action truncation maps in Floer theory. We therefore define 
$$
S\H_*(\{K_\lambda\}) := \lim\limits_{\stackrel\longrightarrow{b\to\infty}} \lim\limits_{\stackrel\longleftarrow{a\to-\infty}} S\H_*^{(a,b)}(\{K_\lambda\}),
$$
and obtain a product 
$$
S\H_*(\{K_\lambda\})\otimes S\H_*(\{K_\lambda\})\to S\H_*(\{K_\lambda\}).
$$
The associativity of this product can be proved directly by incorporating arity 3 operations in the above discussion. However, associativity also follows a posteriori from Theorem~\ref{thm:Rabinowitz-cone} below, according to which we have a natural isomorphism $S\H_*(\{K_\lambda\})\simeq S\H_*(\p W)$ compatible with the products. 
In view of this, we will refer to $S\H_*(\{K_\lambda\})$ as being a ring. 
\end{example}

\begin{example}[Symplectic homology of filled Liouville cobordisms]
The previous construction can be generalized in the spirit of~\cite{CO} in order to describe product structures on the symplectic homology groups of filled Liouville cobordisms.

Consider $W$ a Liouville cobordism with a Liouville filling $F$ of its negative boundary. Denote $W'=F\cup W$ and further $\wh{W'}$ its symplectic completion. The symplectic completion $\wh{F}$ naturally embeds into $\wh{W'}$ via the Liouville flow and we denote $\wh{F}_R=F\cup [1,R]\times \p F$ for $R\ge 1$. Given $\lambda_+\ge 0$ we consider the Hamiltonians $K_{\lambda_+}:\wh{W'}\to\R$ as in the previous example. Given $\lambda_-\le 0$ and $R\ge 1$ we denote $K_{\lambda_-,R}:\wh{W'}\to\R$ the (smoothing of the) Hamiltonian which coincides on $\wh{F}_R$ with $K_{\lambda_-}:\wh{F}\to\R$ from the previous example, and which is constant equal to $(R-1)\lambda_-$ on $\wh{W'}\, \setminus\, \wh{F}_R$. See Figure~\ref{fig:Ham-cobordisms}. 

\begin{figure}
\begin{center}
\includegraphics[width=.7\textwidth]{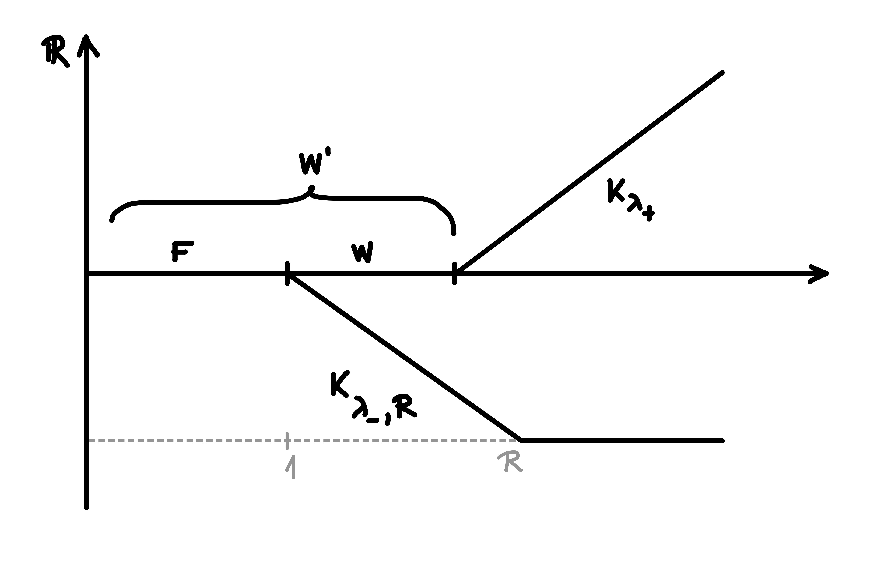}
\caption{Hamiltonians for cobordisms}
\label{fig:Ham-cobordisms}
\end{center}
\end{figure}

Given parameters $\lambda_-\le 0\le \lambda_+$ and $R\ge 1$ we denote 
$$
c_{\lambda_-,\lambda_+,R}:FC_*(K_{\lambda_-,R})[n]\to FC_*(K_{\lambda_+})[n]
$$
the continuation map induced by a monotone homotopy. Arguing as in the previous example one defines for a finite action window $(a,b)$ the groups 
$$
S\H_*^{(a,b)}(\{K_{\lambda_-,R}\}, \{K_{\lambda_+}\})= \lim\limits_{\stackrel\longrightarrow{\lambda_+\to\infty}} \lim\limits_{\stackrel\longleftarrow{\lambda_-\to-\infty}} \lim\limits_{\stackrel\longleftarrow{R\to\infty}} H_*(Cone(c_{\lambda_-,\lambda_+,R})^{(a,b)}),
$$
with product maps 
$$
S\H_*^{(a,b)}(\{K_{\lambda_-,R}\}, \{K_{\lambda_+}\})^{\otimes 2} \to S\H_*^{(a+b,2b)}(\{K_{\lambda_-,R}\}, \{K_{\lambda_+}\}).
$$
These maps are compatible with the canonical action truncation maps in Floer theory, and we define 
$$
S\H_*(\{K_{\lambda_-,R}\}, \{K_{\lambda_+}\}) := \lim\limits_{\stackrel\longrightarrow{b\to\infty}} \lim\limits_{\stackrel\longleftarrow{a\to-\infty}} S\H_*^{(a,b)}(\{K_{\lambda_-,R}\}, \{K_{\lambda_+}\}),
$$
and obtain a product 
$$
S\H_*(\{K_{\lambda_-,R}\}, \{K_{\lambda_+}\})\otimes S\H_*(\{K_{\lambda_-,R}\}, \{K_{\lambda_+}\})\to S\H_*(\{K_{\lambda_-,R}\}, \{K_{\lambda_+}\}).
$$
Theorem~\ref{thm:Rabinowitz-cone} can be generalized to an isomorphism of rings between $S\H_*(\{K_{\lambda_-,R}\}, \{K_{\lambda_+}\})$ and $S\H_*(W)=SH_{*+n}(W)$, the symplectic homology of the cobordism $W$ as defined in~\cite{CO}. As a matter of fact, the entire Eilenberg-Steenrod framework from~\cite{CO} could be rephrased and simplified in the language of cones of the present paper. We will not prove these statements here.
\end{example}

\begin{example}[The family of Hamiltonians $H_{\lambda,\mu}$, $\lambda,\mu\in\R$]  \label{example:Hlambdamu}
Consider again a Liouville domain $W$ with Liouville completion $\wh W$. Given real parameters $\lambda,\mu$ which do not belong to the action spectrum of $\p W$ we define a Hamiltonian $H_{\lambda,\mu}$ on $\wh W$ to be a smoothing of the Hamiltonian which is constant equal to $-\lambda/2$ on $\{r\le 1/2\}$, equal to $0$ at $r=1$, linear of slope $\lambda$ on $\{1/2\le r\le 1\}$, and linear of slope $\mu$ on $\{r\ge 1\}$. See Figure~\ref{fig:HL}. 
We denote $L_\lambda=H_{\lambda,\lambda}$. 

We have $H_{\lambda',\mu'}\ge H_{\lambda,\mu}$ for $\lambda'\le\lambda$ and $\mu'\ge \mu$. Also, we have 
\begin{equation}\label{eq:Hlambdamu-additive}
H_{\lambda+\lambda',\mu+\mu'}= H_{\lambda,\mu} + H_{\lambda',\mu'}
\end{equation} 
for all $\lambda,\lambda',\mu,\mu'$. See Figure~\ref{fig:HL}. 

\begin{figure} [ht]
\centering
\input{HL.pstex_t}
\caption{The family of Hamiltonians $\{H_{\lambda,\mu}\}$, $\lambda,\mu\in\R$.}
\label{fig:HL}
\end{figure}
  
As already seen before, the $1$-periodic orbits of $H_{\lambda,\mu}$ fall into two groups: orbits of type $F$ which are located in a neighborhood of the region $\{r\le 1/2\}$, and orbits of type $I$ which are located in a neighborhood of the region $\{r=1\}$. We will distinguish the following two cases: 

\underline{{\it Case (i): $\lambda<0<\mu$.}} In this situation the orbits of type $F$ form a subcomplex, cf.~\cite[\S2.3]{CO}, and we denote 
$$
i_{H_{\lambda,\mu}}:FC_*^F(H_{\lambda,\mu})[n]\to FC_*(H_{\lambda,\mu})[n]
$$
the inclusion of this subcomplex, with quotient complex $FC_*^I(H_{\lambda,\mu})[n]$. 

\underline{{\it Case (ii): $\lambda>0>\mu$.}} In this situation the orbits of type $I$ form a subcomplex, cf.~\cite[\S2.3]{CO}, and we denote 
\begin{equation}\label{eq:pHlambdamu}
p_{H_{\lambda,\mu}}:FC_*(H_{\lambda,\mu})[n]\to FC_*^F(H_{\lambda,\mu})[n]
\end{equation}
the projection onto the quotient complex consisting of orbits of type $F$, with kernel the subcomplex generated by orbits of type $I$. 

In both cases, we have canonical identifications $FC_*^F(H_{\lambda,\lambda})\cong FC_*(L_\lambda)$. 

\begin{remark}[Subcomplexes from action vs.~geometry] 
That orbits of type $F$ in case (i), or of type $I$ in case (ii), form subcomplexes follows from the geometric Lemmas 2.2 and 2.3 in~\cite{CO}. Alternatively, we could argue by Hamiltonian action: choosing the Hamiltonian $H_{\lambda,\mu}$ to be constant on $[0,\delta_\lambda]$ rather than $[0,1/2]$, with $\delta_\lambda>0$ sufficiently small, we could arrange that in case (i) orbits of type $F$ have smaller action that orbits of type $I$ (and similarly in case (ii)). However, such a shape of Hamiltonian would not satisfy equality~\eqref{eq:Hlambdamu-additive} which we need in the sequel, so we prefer the geometric approach. 
\end{remark}

We now consider separately the families 
$$
\cH_\vee= \{H_{\lambda,\mu}\, : \, \lambda<0<\mu\}
$$ 
and 
$$
\cH_\wedge=\{H_{\lambda,\mu}\, : \, \lambda>0>\mu\},
$$ 
and define from each of them a certain symplectic homology ring. 
  
\smallskip   
  
{\it \underline{Case (i). The family $\cH_\vee= \{H_{\lambda,\mu}\, : \, \lambda<0<\mu\}$.}} 

Given parameters $\lambda<0<\mu$ we denote 
$$
c_{(\lambda,\lambda),(\lambda,\mu)}:FC_*(L_\lambda)[n]\to FC_*(H_{\lambda,\mu})[n]
$$
the continuation map induced by a monotone homotopy. Note that we can choose the homotopy to be constant in the region $\{r\le 1\}$, and with this choice the continuation map $c_{(\lambda,\lambda),(\lambda,\mu)}$ is canonically identified with the inclusion $i_{H_{\lambda,\mu}}$ from above. 
We obtain bilinear maps 
\begin{multline} \label{eq:mtilde_V}
m = m_{\big((\lambda,\lambda),(\lambda,\mu)\big), \big((2\lambda,\lambda+\mu),(2\lambda,2\mu)\big)} :  \\
\mbox{\scriptsize{$Cone(c_{(\lambda,\lambda),(\lambda,\mu)})^{(a,b)}\otimes Cone(c_{(\lambda,\lambda),(\lambda,\mu)})^{(a,b)}
\to  Cone(c_{(2\lambda,\lambda+\mu),(2\lambda,2\mu)})^{(a+b,2b)}.$}} 
\end{multline}

We claim that, given $(a,b)$, the above maps canonically stabilize in homology for $\mu$ fixed and $\lambda$ negative enough. To prove the claim, consider $\lambda_1\le \lambda <0<\mu$. We then have homotopy commutative diagrams of continuation maps 
$$
\xymatrix{
FC_*(L_{\lambda_1})[n]\ar[r] \ar[d]& FC_*(H_{\lambda_1,\mu})[n] \ar@{=}[d] \\
FC_*(H_{\lambda_1,\lambda})[n] \ar[r]& FC_*(H_{\lambda_1,\mu})[n]\\
FC_*(L_\lambda)[n] \ar[u]\ar[r]& FC_*(H_{\lambda,\mu})[n]\ar[u]
}
$$
and 
$$
\xymatrix{
FC_*(H_{2\lambda_1,\lambda_1+\mu})[n]\ar[r] \ar[d]& FC_*(H_{2\lambda_1,2\mu})[n] \ar@{=}[d] \\
FC_*(H_{2\lambda_1,\lambda+\mu})[n] \ar[r]& FC_*(H_{2\lambda_1,2\mu})[n]\\
FC_*(H_{2\lambda,\lambda+\mu})[n] \ar[u]\ar[r]& FC_*(H_{2\lambda,2\mu})[n]\ar[u]
}
$$
The vertical maps induce maps between the cones of the horizontal maps, and we obtain homology commutative diagrams for the product structures involving the horizontal continuation maps. We denote symbolically the resulting diagram 
$$
\xymatrix{
m_{\big( (\lambda_1,\lambda_1),(\lambda_1,\mu)\big), \big((2\lambda_1,\lambda_1+\mu),(2\lambda_1,2\mu)\big)} 
\ar[d]\\
m_{\big( (\lambda_1,\lambda),(\lambda_1,\mu)\big), \big((2\lambda_1,\lambda+\mu),(2\lambda_1,2\mu)\big)} \\
\ar[u]
m_{\big( (\lambda,\lambda),(\lambda,\mu)\big), \big((2\lambda,\lambda+\mu),(2\lambda,2\mu)\big)}.
}
$$
For $(a,b)$ fixed, $\mu$ fixed and $\lambda_1\le \lambda\ll 0$, all the above maps are isomorphisms. This shows that the products stabilize. 

Similarly, we claim that, for $(a,b)$ fixed, diagram~\eqref{eq:mtilde_V} stabilizes in homology for $\lambda\ll 0$ and $b\le \mu\le |\lambda|$. The argument is similar, based on the homotopy commutative diagrams of continuation maps in which $b\le \mu\le\mu_1\le|\lambda|$: 
$$
\xymatrix{
FC_*(L_{\lambda})[n]\ar[r] \ar@{=}[d]& FC_*(H_{\lambda,\mu})[n] \ar[d] \\
FC_*(L_\lambda)[n] \ar[r]& FC_*(H_{\lambda,\mu_1})[n]
}
$$
and 
$$
\xymatrix{
FC_*(H_{2\lambda,\lambda+\mu})[n]\ar[r] \ar[d]& FC_*(H_{2\lambda,2\mu})[n] \ar[d] \\
FC_*(H_{2\lambda,\lambda+\mu_1})[n] \ar[r]& FC_*(H_{2\lambda,2\mu_1})[n]
}
$$
These induce homology commutative diagrams for the product structures involving the horizontal continuation maps, which we denote symbolically 
$$
\xymatrix{
m_{\big( (\lambda,\lambda),(\lambda,\mu)\big), \big((2\lambda,\lambda+\mu),(2\lambda,2\mu)\big)} 
\ar[d]\\
m_{\big( (\lambda,\lambda),(\lambda,\mu_1)\big), \big((2\lambda,\lambda+\mu_1),(2\lambda,2\mu_1)\big)}
}
$$
The outcome of this discussion is that the groups defined by 
$$
S\H_*^{(a,b)}(\cH_\vee) := \lim\limits_{\stackrel\longrightarrow{\mu\to\infty}}\lim\limits_{\stackrel\longleftarrow{\lambda\to-\infty}}H_*(Cone(c_{(\lambda,\lambda),(\lambda,\mu)})^{(a,b)})
$$
inherit a product 
$$
m:S\H_*^{(a,b)}(\cH_\vee)\otimes S\H_*^{(a,b)}(\cH_\vee)\to S\H_*^{(a+b,2b)}(\cH_\vee).  
$$
We define 
$$
S\H_*(\cH_\vee) := \lim\limits_{\stackrel\longrightarrow{b\to\infty}} \lim\limits_{\stackrel\longleftarrow{a\to-\infty}} S\H_*^{(a,b)}(\cH_\vee)
$$
and this inherits a bilinear product 
$$
m:S\H_*(\cH_\vee)\otimes S\H_*(\cH_\vee)\to S\H_*(\cH_\vee).
$$

\smallskip 

{\it \underline{Case (ii). The family $\cH_\wedge= \{H_{\lambda,\mu}\, : \, \lambda>0>\mu\}$.}} 

Given parameters $\lambda>0>\mu$ we denote 
$$
c_{(\lambda,\mu),(\lambda,\lambda)}:FC_*(H_{\lambda,\mu})[n]\to FC_*(L_\lambda)[n]
$$
the continuation map induced by a monotone homotopy. We can choose the homotopy to be constant in the region $\{r\le 1\}$, and with this choice the continuation map $c_{(\lambda,\mu),(\lambda,\lambda)}$ is canonically identified with the projection $p_{H_{\lambda,\mu}}$ from~\eqref{eq:pHlambdamu}. 
We obtain bilinear maps 

\begin{multline} \label{eq:mtilde_A}
m = m_{\big((\lambda,\mu),(\lambda,\lambda)\big), \big((2\lambda,\lambda+\mu),(2\lambda,2\lambda)\big)} :  \\
\mbox{\scriptsize{$Cone(c_{(\lambda,\mu),(\lambda,\lambda)})^{(a,b)}\otimes Cone(c_{(\lambda,\mu),(\lambda,\lambda)})^{(a,b)}
\to  Cone(c_{(2\lambda,\lambda+\mu),(2\lambda,2\lambda)})^{(a+b,2b)}.$}}
\end{multline}

Given $(a,b)$, the above maps canonically stabilize in homology for $\lambda$ fixed and $\mu$ negative enough. To prove this, consider $\lambda>0>\mu\ge \mu_1$. We then have homotopy commutative diagrams of continuation maps 
$$
\xymatrix{
FC_*(H_{\lambda,\mu})[n]\ar[r] & FC_*(H_{\lambda,\lambda})[n]  \\
FC_*(H_{\lambda,\mu_1})[n] \ar[r]\ar[u] & FC_*(H_{\lambda,\lambda})[n]\ar@{=}[u]
}
$$
and 
$$
\xymatrix{
FC_*(H_{2\lambda,\lambda+\mu})[n]\ar[r] & FC_*(H_{2\lambda,2\lambda})[n]  \\
FC_*(H_{2\lambda,\lambda+\mu_1})[n] \ar[r]\ar[u] & FC_*(H_{2\lambda,2\lambda})[n]\ar@{=}[u] 
}
$$
The vertical maps induce maps between the cones of the horizontal maps, and we obtain homology commutative diagrams for the product structures involving the horizontal continuation maps. We denote symbolically the resulting diagram 
$$
\xymatrix{
m_{\big((\lambda,\mu),(\lambda,\lambda)\big), \big((2\lambda,\lambda+\mu),(2\lambda,2\lambda)\big)}\\
m_{\big((\lambda,\mu_1),(\lambda,\lambda)\big), \big((2\lambda,\lambda+\mu_1),(2\lambda,2\lambda)\big)} \ar[u]. 
}
$$
For $(a,b)$ fixed, $\lambda$ fixed and $0\gg \mu\ge \mu_1$, these maps are isomorphisms. This shows that the products stabilize. 

Similarly, for $(a,b)$ fixed, diagram~\eqref{eq:mtilde_A} stabilizes in homology for $0\gg \mu$ and $b\le \lambda\le |\mu|$. The argument is similar, based on the homotopy commutative diagrams of continuation maps for parameters $|\mu|\ge \lambda_1\ge \lambda>0>\mu$:  
$$
\xymatrix{
FC_*(H_{\lambda,\mu})[n]\ar[r] \ar@{=}[d]& FC_*(H_{\lambda,\lambda})[n] \ar[d] \\
FC_*(H_{\lambda,\mu})[n] \ar[r]& FC_*(H_{\lambda,\lambda_1})[n]\\
FC_*(H_{\lambda_1,\mu})[n] \ar[u]\ar[r]& FC_*(H_{\lambda_1,\lambda_1})[n]\ar[u]
}
$$
and 
$$
\xymatrix{
FC_*(H_{2\lambda,\lambda+\mu})[n]\ar[r] \ar[d]& FC_*(H_{2\lambda,2\lambda})[n] \ar[d] \\
FC_*(H_{2\lambda,\lambda_1+\mu})[n] \ar[r]& FC_*(H_{2\lambda,2\lambda_1})[n]\\
FC_*(H_{2\lambda_1,\lambda_1+\mu})[n] \ar[u]\ar[r]& FC_*(H_{2\lambda_1,2\lambda_1})[n]\ar[u]
}
$$
The corresponding homology commutative diagram between products on cones is 
$$
\xymatrix{
m_{\big( (\lambda,\mu),(\lambda,\lambda)\big), \big((2\lambda,\lambda+\mu),(2\lambda,2\lambda)\big)} 
\ar[d]\\
m_{\big( (\lambda,\mu),(\lambda,\lambda_1)\big), \big((2\lambda,\lambda_1+\mu),(2\lambda,2\lambda_1)\big)} \\
\ar[u]
m_{\big( (\lambda_1,\mu),(\lambda_1,\lambda_1)\big), \big((2\lambda_1,\lambda_1+\mu),(2\lambda_1,2\lambda_1)\big)}.
}
$$
The outcome of the discussion is that the groups defined by 
$$
S\H_*^{(a,b)}(\cH_\wedge) := \lim\limits_{\stackrel\longrightarrow{\lambda\to\infty}}
\lim\limits_{\stackrel\longleftarrow{\mu\to-\infty}}
H_*(Cone(c_{(\lambda,\mu),(\lambda,\lambda)})^{(a,b)})
$$
inherit a product 
$$
m:S\H_*^{(a,b)}(\cH_\wedge)\otimes S\H_*^{(a,b)}(\cH_\wedge)\to S\H_*^{(a+b,2b)}(\cH_\wedge).  
$$
We define 
$$
S\H_*(\cH_\wedge) := \lim\limits_{\stackrel\longrightarrow{b\to\infty}} \lim\limits_{\stackrel\longleftarrow{a\to-\infty}} S\H_*^{(a,b)}(\cH_\wedge)
$$
and this inherits a product 
$$
m:S\H_*(\cH_\wedge)\otimes S\H_*(\cH_\wedge)\to S\H_*(\cH_\wedge).
$$

As in Example~\ref{example:Klambda}, the associativity of the products on $S\H_*(\cH_\vee)$ and $S\H_*(\cH_\wedge)$ can be proved directly by incorporating arity 3 operations in the discussion. However, associativity also follows a posteriori from Theorems~\ref{thm:duality-cone} and~\ref{thm:Rabinowitz-cone} below, according to which we have natural isomorphisms $S\H_*(\cH_\vee)\simeq S\H_*(\cH_\wedge)\simeq S\H_*(\p W)$ compatible with the products. 
In view of this, we will refer to $S\H_*(\cH_\vee)$ and $S\H_*(\cH_\wedge)$ as being rings. 
\end{example}

\section{Rabinowitz Floer homology and cohomology rings via cones} \label{sec:RFH-cones}

In this section we discuss Rabinowitz Floer homology and cohomology rings from the cone perspective. In \S\ref{sec:RFH-hom-ring} we give two cone descriptions of the Rabinowitz Floer homology ring, in \S\ref{sec:alternative_product} we use a cone description to define the continuation product on Rabinowitz Floer cohomology, and in \S\ref{sec:unit} we construct the unit for this product.

\subsection{The Rabinowitz Floer homology ring}\label{sec:RFH-hom-ring}

Let $W$ be a Liouville domain with symplectic completion $\wh W$. Let us recall the definition of the Rabinowitz Floer homology ring $S\H_*(\p W)$ from~\cite{CO} in terms of the family of Hamiltonians $H_{\lambda,\mu}$, $\lambda<0<\mu$ from Example~\ref{example:Hlambdamu}. For a fixed finite action interval $(a,b)$ we set  
$$
S\H_*^{(a,b)}(\p W)=\lim\limits_{\stackrel\longrightarrow{\mu\to\infty}}\lim\limits_{\stackrel\longleftarrow{\lambda\to-\infty}} FH_*^{(a,b)}(H_{\lambda,\mu})[n], 
$$
and further 
$$
S\H_*(\p W)=\lim\limits_{\stackrel\longrightarrow{b\to\infty}}\lim\limits_{\stackrel\longleftarrow{a\to-\infty}}S\H_*^{(a,b)}(\p W). 
$$
The product is defined from the pair-of-pants product
$$
FH_*^{(a,b)}(H_{\lambda,\mu})[n]\otimes FH_*^{(a,b)}(H_{\lambda,\mu})[n]\to FH_*^{(a+b,2b)}(H_{2\lambda,2\mu})[n], 
$$
which induces 
$$
S\H_*^{(a,b)}(\p W) \otimes S\H_*^{(a,b)}(\p W)\to S\H_*^{(a+b,2b)}(\p W)
$$
and further 
$$
S\H_*(\p W)\otimes S\H_*(\p W)\to S\H_*(\p W).
$$
The main result of~\cite{CHO-PD} is the following. 

\begin{theorem}[{Poincar\'e duality~\cite[Theorem~4.8]{CHO-PD}}]  \label{thm:coh-product}
Let $W$ be a Liouville domain of dimension $2n$. 

(i) Rabinowitz Floer cohomology $SH^*(\p W)$ carries a canonical unital degree $n-1$ product. 

(ii) There is a canonical Poincar\'e duality isomorphism of unital rings 
$$
PD: S\H_*(\p W)\simeq SH^{1-n-*}(\p W).
$$ 
\end{theorem}

From the perspective of~\cite{CHO-PD} the proofs of (i) and (ii) are inextricably related: the canonical isomorphism $PD$ induces a canonical unital product on Rabinowitz Floer cohomology. One goal of the present paper is to disconnect these two sides of Theorem~\ref{thm:coh-product} and give an independent proof of Poincar\'e duality in terms of cones. 

We begin by realizing Rabinowitz Floer homology in terms of cones.

\begin{theorem} \label{thm:Rabinowitz-cone}
We have canonical isomorphisms which respect the products 
$$
S\H_*(\{K_\lambda\})\simeq S\H_*(\cH_\vee) \simeq S\H_*(\p W).
$$
\end{theorem}

\begin{proof} 

{\it Step~1. We prove the first isomorphism $S\H_*(\{K_\lambda\})\simeq S\H_*(\cH_\vee)$. }

Denote for convenience 
$$
H_{0,\lambda}=K_\lambda.  
$$

Consider $\lambda<0<\mu$. Then we have a homotopy commutative diagram of continuation maps 
$$
\xymatrix{
FC_*(K_\lambda)[n]\ar[d]_{c} \ar@<.75ex>[rr]^-\pi_-\sim && 
FC_*(L_\lambda)[n] \ar[d]^{c_\vee}  \ar@<.75ex>[ll]^-\iota \\
FC_*(K_\mu)[n] \ar@<.75ex>[rr]^-p_-\sim&& FC_*(H_{\lambda,\mu})[n] \ar@<.75ex>[ll]^-i
}
$$
in which the horizontal maps are chain homotopy equivalences, and where the maps $\pi$ and $p$ preserve the filtration, but the maps $\iota$ and $i$ do not. The left vertical map $c=c_{(0,\lambda),(0,\mu)}$ is involved in the definition of $S\H_*(\{K_\lambda\})$. The right vertical map $c_\vee=c_{(\lambda,\lambda),(\lambda,\mu)}$ is involved in the definition of $S\H_*(\cH_\vee)$. See Figure~\ref{fig:cone-iso-KHvee-source}. 

\begin{figure} [ht]
\centering
\input{cone-iso-KHvee-source.pstex_t}
\caption{Isomorphism $S\H_*(\{K_\lambda\})\simeq S\H_*(\cH_\vee)$: continuation diagram at the source.}
\label{fig:cone-iso-KHvee-source}
\end{figure} 

Consider also the homotopy commutative diagram of continuation maps (Figure~\ref{fig:cone-iso-KHvee-target})
$$
\xymatrix{
FC_*(K_{\lambda+\mu})[n]\ar[d]_{c_\wedge} \ar@<.75ex>[rr]^-\pi_-\sim && 
FC_*(H_{2\lambda,\lambda+\mu})[n] \ar[d]^{c_\vee}  \ar@<.75ex>[ll]^-\iota \\
FC_*(K_{2\mu})[n] \ar@<.75ex>[rr]^-p_-\sim&& FC_*(H_{2\lambda,2\mu})[n] \ar@<.75ex>[ll]^-i
}
$$

\begin{figure} [ht]
\centering
\input{cone-iso-KHvee-target.pstex_t}
\caption{Isomorphism $S\H_*(\{K_\lambda\})\simeq S\H_*(\cH_\vee)$: continuation diagram at the target.}
\label{fig:cone-iso-KHvee-target}
\end{figure} 

By Proposition~\ref{prop:arity2homotopyinvariance} and the subsequent discussion on homotopy invariance for $A_2$-triples, we obtain an isomorphism between product structures 
\begin{equation}\label{eq:iso-product-KH}
m_{\big((0,\lambda),(0,\mu)\big), \big((0,\lambda+\mu),(0,2\mu)\big)} \simeq 
m_{\big((\lambda,\lambda),(\lambda,\mu)\big), \big((2\lambda,\lambda+\mu),(2\lambda,2\mu)\big)}.
\end{equation}

In order to conclude the proof, we fix an action interval $(a,b)$ with $a<0<b$ and consider the parameter values $\lambda=a$, $\mu=b$ in the previous setup. A stabilization argument as in Example~\ref{example:Hlambdamu} shows that we have isomorphisms 

{\scriptsize{
$$
\xymatrix
@C=20pt
{
H_*(Cone(c_{(a,a),(a,b)}))\otimes H_*(Cone(c_{(a,a),(a,b)})) \ar[rr]^-{m}\ar[d]^\simeq && H_*(Cone(c_{(2a,a+b),(2a,2b)})) \ar[d]^\simeq \\
S\H_*^{(a,b)}(\cH_\vee)\otimes S\H_*^{(a,b)}(\cH_\vee) \ar[rr]^-{m}&& S\H_*^{(a+b,2b)}(\cH_\vee)
}
$$
}}

and 

{\scriptsize{
$$
\xymatrix{
H_*(Cone(c_{(0,a),(0,b)}))\otimes H_*(Cone(c_{(0,a),(0,b)})) \ar[rr]^-{m}\ar[d]^\simeq && H_*(Cone(c_{(0,a+b),(0,2b)})) \ar[d]^\simeq \\
S\H_*^{(a,b)}(\{K_\lambda\})\otimes S\H_*^{(a,b)}(\{K_\lambda\}) \ar[rr]^-{m}&& S\H_*^{(a+b,2b)}(\{K_\lambda\})
}
$$
}}

The top lines in the above two diagrams are isomorphic by~\eqref{eq:iso-product-KH}, and we infer the isomorphism of the bottom lines. This isomorphism is compatible with action truncation maps and yields an isomorphism of rings 
$$
S\H_*(\cH_\vee)\simeq S\H_*(\{K_\lambda\}).
$$

\bigskip 

{\it Step~2. We prove the second isomorphism $S\H_*(\cH_\vee) \simeq S\H_*(\p W)$. }

Consider parameters $\lambda<0<\mu$. The group and product structure on $S\H_*(\cH_\vee)$ are built from the cone of the continuation map 
$$
c_\vee=c_{(\lambda,\lambda),(\lambda,\mu)}:FC_*(L_\lambda)[n]\to FC_*(H_{\lambda,\mu})[n]. 
$$
Recall that the 1-periodic orbits of $H_{\lambda,\mu}$ fall into two classes, $F$ consisting of orbits located in a neighborhood of $\{r\le 1/2\}$ and $I$ consisting of orbits located in a neighborhood of $\{r=1\}$, giving rise to a subcomplex $FC_*^F(H_{\lambda,\mu})$ and to a quotient complex $FC_*^I(H_{\lambda,\mu})$. We have a canonical identification $FC_*(L_\lambda)\equiv FC_*^F(H_{\lambda,\mu})$ and $c_\vee[-n]\equiv 
i_{H_{\lambda,\mu}}$, the inclusion of $FC_*^F(H_{\lambda,\mu})$ into $FC_*(H_{\lambda,\mu})$. 

The projection $\pi:Cone(i_{H_{\lambda,\mu}})\to FC_*^I(H_{\lambda,\mu})[n]$ is a homotopy equivalence (see for example~\cite[Lemma~4.3]{CO}). Moreover, the map $\pi$ does not increase the action, and also its homotopy inverse 
$$\left(\begin{array}{c} 0 \\\one \\-\p_{I,F}\end{array}\right)[-n]:FC_*^I(H_{\lambda,\mu})\to FC_*^F(H_{\lambda,\mu})\oplus FC_*^I(H_{\lambda,\mu})\oplus FC_{*-1}^F(H_{\lambda,\mu})$$ 
does not increase the action. As a consequence, the induced maps  
$$
\pi^{(a,b)}:Cone(i_{H_{\lambda,\mu}})^{(a,b)}\to FC_*^{I,(a,b)}(H_{\lambda,\mu})[n]
$$
are also homotopy equivalences for any action interval $(a,b)$. 

Let us now fix such a finite action interval $(a,b)$. For $\lambda\ll 0$ the action of the orbits in the group $F$ falls below the action window. Thus the only elements in $Cone(i_{H_{\lambda,\mu}})^{(a,b)}$ are of the form $(A,0)$, where $A\in FC_*(H_{\lambda,\mu})[n]$, and actually $A\in FC_*^I(H_{\lambda,\mu})[n]$. The product of two such elements in $Cone(i_{H_{2\lambda,2\mu}})$ is considered modulo action $\le a+b$, and as such is also represented for $\lambda\ll 0$ by elements in $FC_*(H_{2\lambda,2\mu})[n]$, and actually in $FC_*^I(H_{2\lambda,2\mu})[n]$. We then have 
\begin{align*}
\pi^{(a+b,2b)}m((A,0),(A',0)) & = \pi^{(a+b,2b)}(\mu(A,A'),0)\\
& =\mu(A,A') \, \mathrm{mod} \le a+b. 
\end{align*}
Thus $\pi$ interchanges in the relevant action window the product $m$ on $Cone(i_{H_{\lambda,\mu}})^{(a,b)}$ with the pair-of-pants product $\mu$ on $FH_*^{(a,b)}(H_{\lambda,\mu})[n]$. 

These identifications are compatible with the limits involved in the definitions of $S\H_*(\cH_\vee)$ and $S\H_*(\p W)$. The desired isomorphism of rings follows. 
\end{proof}

\subsection{The Rabinowitz Floer cohomology ring} \label{sec:alternative_product} 

Our original definition of the secondary product on $SH^*(\p W)$ was intimately tied to the proof of Poincar\'e duality. We give here an alternative definition which is independent of that proof. We prove the equivalence of the two definitions in Proposition~\ref{prop:products_coincide} and its Corollary~\ref{cor:products_coincide}.  

We consider the family of Hamiltonians $\cH_\wedge=\{H_{\mu,\lambda}\, : \, \mu>0>\lambda\}$. Define 
\begin{equation} \label{eq:SHhat}
\wh{SH}_*(\p W) := \lim\limits_{\stackrel\longrightarrow{b\to\infty}}\lim\limits_{\stackrel\longleftarrow{a\to-\infty}}\lim\limits_{\stackrel\longleftarrow{\mu\to\infty,\, \lambda\to-\infty}}FH_*^{(a,b)}(H_{\mu,\lambda}).
\end{equation}
(This is the same as $SH_*(\p W\times I,\p W\times \p I)$ from~\cite{CO}.)

We consider the continuation map $c_\wedge=c_{(\mu,\lambda),(\mu,\mu)}:FC_*(H_{\mu,\lambda})[n]\to FC_*(L_\mu)[n]$ and the associated map $\sigma$ from the definition of the multiplication $m:Cone(c_{(\mu,\lambda),(\mu,\mu)})^{\otimes 2}\to Cone(c_{(2\mu,\lambda+\mu),(2\mu,2\mu)})$. Thus 
$$
\sigma:FC_*(H_{\mu,\lambda})[n]\otimes FC_*(H_{\mu,\lambda})[n] \to  FC_{*+1}(H_{2\mu,\lambda+\mu})[n]. 
$$
(We could have taken the target of $\sigma$ to be also $FC_{*+1}(H_{2\mu,2\lambda})[n]$, but the above choice is in line with the previous discussion.) 

Recall that the map $\sigma$ satisfies the relation $[\p,\sigma]=m_R(\one\otimes c_\wedge)-m_L(c_\wedge\otimes\one)$. 
The relevant observation now is that, for a fixed finite action interval $(a,b)$, the filtered map 
$$
\sigma^{(a,b)}:FC_*^{(a,b)}(H_{\mu,\lambda})[n]\otimes FC_*^{(a,b)}(H_{\mu,\lambda})[n] \to FC_{*+1}^{(a+b,2b)}(H_{2\mu,\lambda+\mu})[n]
$$
is actually a chain map as soon as $\mu>2b$. Indeed, for $\mu>2b$ the 1-periodic orbits of the Hamiltonian $L_\mu$ have action larger than $b$, hence the continuation map $c_\wedge$ vanishes on $FC_*^{\le b}(H_{\mu,\lambda})[n]$. 
We obtain a degree $1$ product 
$$
\sigma^{(a,b)}:FH_*^{(a,b)}(H_{\mu,\lambda})[n]^{\otimes 2}\to FH_{*+1}^{(a+b,2b)}(H_{2\mu,\lambda+\mu})[n]
$$
for $\mu>2b$. This product stabilizes for $(a,b)$ fixed as $\lambda\to-\infty$ and $\mu\to\infty$, and it is compatible with the tautological maps given by enlarging the action window. As such, it induces a degree $1$ product 
$$
\sigma:\wh{SH}_*(\p W)[n]\otimes \wh{SH}_*(\p W)[n]\to \wh{SH}_{*+1}(\p W)[n].
$$
In view of the canonical isomorphism $\wh{SH}_*(\p W)\simeq SH^{-*}(\p W)$, we infer a degree $-1$ product 
$$
\sigma^\vee:SH^*(\p W)[-n] \otimes SH^*(\p W)[-n] \to SH^*(\p W)[-n].
$$

It is useful to recast $\sigma$ and $\sigma^\vee$ as degree $0$ products. Our convention is to use the shift $\usigma=-\sigma[-1,-1;-1]$, i.e. $\sigma=\usigma[1,1;1]$, which defines a degree $0$ product 
$$
\usigma:\wh{SH}_*(\p W)[n-1]\otimes \wh{SH}_*(\p W)[n-1]\to \wh{SH}_{*+1}(\p W)[n-1]. 
$$
Dually, we have a degree $0$ product in cohomology 
$$
\usigma^\vee:SH^*(\p W)[1-n] \otimes SH^*(\p W)[1-n] \to SH^*(\p W)[1-n].
$$
This is our alternative definition of the product on $SH^*(\p W)[1-n]$. 

\begin{definition}
We call the products $\usigma$ and $\usigma^\vee$ the \emph{continuation secondary products on $\wh{SH}_*(\p W)[n-1]$ and $SH^*(\p W)[1-n]$}. In order to emphasize the role played by continuation maps, we denote them by $\usigma_c$ and $\usigma^c$ in the Introduction and in~\S\ref{sec:pop-GH-iso}. Their unshifted versions are then denoted $\sigma_c$ and $\sigma^c$.  
\end{definition}

\subsection{The unit in cohomology} \label{sec:unit}

The ring $(\wh{SH}_*(\p W)[n-1],\usigma)$ is unital and we give in this section a description of the unit. As before, we consider the family of Hamiltonians $\cH_\wedge = \{H_{\mu,\lambda}\, : \, \mu>0>\lambda\}$. The starting point of the construction is to consider the family of cycles 
$$
U_{\mu,\lambda}\in FC_0(H_{\mu,\lambda})[n-1] 
$$
defined as follows. 

Consider a Morse perturbation of $H_{\mu,\lambda}$ in the region $\{r\le 1/2\}$ with a single minimum denoted $1_\mu$. (We can assume w.l.o.g. that the perturbation is independent of $\lambda$ and also independent of $\mu$ up to translating the values of the function. Also, for further use, we can assume w.l.o.g. that the actions of all the constant orbits are slightly larger than $\mu/2$.) The Conley-Zehnder degree of $1_\mu$ is equal to $n$, and we define 
$$
U_{\mu,\lambda}=\p(1_\mu), 
$$
where $\p$ is the Floer differential for the complex $FC_*(H_{\mu,\lambda})[n-1]$.
The reason for denoting the minimum $1_\mu$ is that it is a Floer cycle in $FC_*(L_\mu)[n]$ which defines the unit-up-to-continuation for $FH_*(L_\mu)[n]$. Note also that $U_{\mu,\lambda}\in FC_*^I(H_{\mu,\lambda})$, the subcomplex generated by orbits located in a neighborhood of $\{r=1\}$. 
Moreover, since $U_{\mu,\lambda}$ is the differential of an element of action (slightly larger than) $\mu/2$, we have 
$$
U_{\mu,\lambda}\in FC_0^{\le \frac\mu2}(H_{\mu,\lambda})[n-1]. 
$$
Given $a<\mu/2$ we denote  
$$
U_{\mu,\lambda}^{a} \in FC_0^{(a,\frac\mu2)}(H_{\mu,\lambda})[n-1]
$$
the truncation of the cycle $U_{\mu,\lambda}$ in action $> a$. 

The following Lemma is a variant of~\cite[Lemma~7.4]{CO}. 

\begin{lemma} \label{lem:alternative-whSH}
The group $\wh{SH}_*(\p W)$ defined by~\eqref{eq:SHhat} coincides with  
$$
\lim\limits_{\stackrel\longrightarrow{b\to\infty}}\lim\limits_{\stackrel\longleftarrow{a\to-\infty}}\lim\limits_{\stackrel\longleftarrow{\lambda\to-\infty}}FH_*^{(a,b)}(H_{2b,\lambda}).
$$
\qed\end{lemma}

\begin{lemma} \label{lem:U} 
The collection of classes $[U_{2b,\lambda}^a]\in FH_0^{(a,b)}(H_{2b,\lambda})[n-1]$, $a<0<b$, $\lambda<0$ defines a class 
$$
U= \lim\limits_{\stackrel\longrightarrow{b\to\infty}}\lim\limits_{\stackrel\longleftarrow{a\to-\infty}}\lim\limits_{\stackrel\longleftarrow{\lambda\to-\infty}} [U_{2b,\lambda}^a] \in \wh{SH}_0(\p W)[n-1]. 
$$
\end{lemma}

\begin{proof}
The key observation is the following. Let $\mu' \ge \mu > 0 > \lambda \ge \lambda'$ and consider the shifted continuation map $\uc=\uc_{(\mu',\lambda'),(\mu,\lambda)}:FC_*(H_{\mu',\lambda'})[n-1]\to FC_*(H_{\mu,\lambda})[n-1]$. Then 
\begin{equation}\label{eq:cU}
\uc_*[U_{\mu',\lambda'}]=[U_{\mu,\lambda}].
\end{equation} 
To prove this we choose the homotopy from $H_{\mu',\lambda'}$ to $H_{\mu,\lambda}$ to be constant up to translation in a small neighborhood of the region $\{r\le 1/2\}$ that contains the constant $1$-periodic orbits and the nonconstant $1$-periodic orbits of $H_{\mu',\lambda'}$ that correspond to Reeb orbits with period $\le \mu$. Then 
$$
\uc(1_{\mu'})=1_\mu + \beta
$$ 
with $\beta\in FC_*^I(H_{\mu,\lambda})[n-1]$, and we obtain 
$$
\uc(U_{\mu',\lambda'})= \uc\p (1_{\mu'}) = \p \uc(1_{\mu'}) = \p (1_\mu + \beta) = U_{\mu,\lambda} + \p \beta. 
$$
Equation~\eqref{eq:cU} has filtered variants. For a finite action window $(a,b)$ and a choice of parameters  $0>\lambda\ge \lambda'$, the continuation map $\uc=\uc_{(2b,\lambda'),(2b,\lambda)}$ satisfies $\uc_*[U_{2b,\lambda'}^a]=[U_{2b,\lambda}^a]$, so that we can define the limit 
$$
U^{(a,b)}=\lim\limits_{\stackrel\longleftarrow{\lambda\to-\infty}} [U_{2b,\lambda}^a] \in \lim\limits_{\stackrel\longleftarrow{\lambda\to-\infty}} FH_0^{(a,b)}(H_{2b,\lambda})[n-1]=\wh{SH}_*^{(a,b)}(\p W)[n-1]. 
$$
The rest of the proof is formal. The classes $[U^{(a,b)}]$ are compatible with the morphisms given by enlarging the action windows, hence the class $U$ is well-defined. 
\end{proof}

\begin{proposition} \label{prop:U}
The class $U$ from Lemma~\ref{lem:U} is the unit of the ring $(\wh{SH}_*(\p W)[n-1],\usigma)$. 
\end{proposition}

\begin{proof}
Recall the fundamental relation $[\p,\sigma]=m_R(\one\otimes c_\wedge) - m_L(c_\wedge \otimes \one)$, which translates into 
\footnote{Note the sign change.} 
$[\p,\usigma]=\um_R(\one\otimes \uc_\wedge) + \um_L(\uc_\wedge\otimes \one)$
in the notation of Lemma~\ref{lem:extension-triples-from-Ainfty}, where $\uc_\wedge=c_\wedge[-1;0]$.    
Let us evaluate both sides at $1_{2b}$ at the first entry and denote 
$$
\uzeta= \usigma(1_{2b}\otimes \one) : FC_*(H_{2b,\lambda})[n-1]\to FC_*(H_{4b,\lambda+2b})[n-1].
$$
This is a linear map of degree $1$, the shifted degree of $1_{2b}$. The relation for $[\p,\usigma]$ becomes 
\begin{equation}\label{eq:sigmaU}
[\p,\uzeta] + \usigma(U_{2b,\lambda}\otimes\one) = \um_R(1_{2b}\otimes \uc_\wedge) + \um_L(\uc_\wedge(1_{2b})\otimes\one). 
\end{equation}
This is a relation between degree $0$ maps defined on $FC_*(H_{2b,\lambda})[n-1]$ and taking values in $FC_*(H_{4b,\lambda+2b})[n-1]$. 

The filtered version of relation~\eqref{eq:sigmaU} is 
$$
[\p,\uzeta] + \usigma(U_{2b,\lambda}^a\otimes\one) = \um_R(1_{2b}\otimes \uc_\wedge) + \um_L(\uc_\wedge(1_{2b})\otimes\one)
$$
and holds at the level of filtered maps acting as 
$$
FC_*^{(a,b)}(H_{2b,\lambda})[n-1]\to FC_*^{(a+b,2b)}(H_{4b,\lambda+2b})[n-1]. 
$$
The term $\um_R(1_{2b}\otimes \uc_\wedge)$ vanishes because the map $\uc_\wedge$ acts as $\uc_\wedge:FC_*^{(a,b)}(H_{2b,\lambda})[n-1]\to FC_*^{(a,b)}(L_{2b})[n-1]$ and the latter complex is zero because all the orbits of $L_{2b}$ have action larger than $b$. On the other hand $\uc_\wedge(1_{2b})=1_{2b}$ and the second term on the right hand side is therefore equal to $\um_L(1_{2b}\otimes\one)$. This is precisely the continuation map $FC_*^{(a,b)}(H_{2b,\lambda})[n-1]\to FC_*^{(a+b,2b)}(H_{4b,\lambda+2b})[n-1]$ induced by a homotopy which is non-increasing on $\{r\ge 1\}$ and non-decreasing with gap equal to $b$ on $\{r\le 1\}$. The outcome of the discussion is that $\usigma(U_{2b,\lambda}^a\otimes\one)$ induces in homology the continuation map 
$$
\uc_*:FH_*^{(a,b)}(H_{2b,\lambda})[n-1]\to FH_*^{(a+b,2b)}(H_{4b,\lambda+2b})[n-1].
$$

As a consequence of Lemma~\ref{lem:alternative-whSH}, the limit 
$$
\lim\limits_{\stackrel\longrightarrow{b\to\infty}}\lim\limits_{\stackrel\longleftarrow{a\to-\infty}}\lim\limits_{\stackrel\longleftarrow{\lambda\to-\infty}} \Big(\uc_*:FH_*^{(a,b)}(H_{2b,\lambda})[n-1]\to FH_*^{(a+b,2b)}(H_{4b,\lambda+2b})[n-1]\Big) 
$$
is equal to $\mathrm{Id}_{\wh{SH}_*(\p W)[n-1]}$. However, the previous discussion shows that the above limit is also equal to $\usigma(U\otimes\one)$. This shows that $U$ is the unit for the ring $(\wh{SH}_*(\p W)[n-1],\usigma)$. 
\end{proof}

\section{Poincar\'e duality} \label{sec:PD}

In this section we prove the Cone Duality Theorem~\ref{thm:duality-cone}, derive from it the cone version of Poincar\'e duality (Theorem~\ref{thm:PD-cones}), and show that it coincides with the Poincar\'e Duality Theorem~\ref{thm:coh-product}.

\subsection{Cone duality theorem}


\begin{theorem}[Cone duality] \label{thm:duality-cone}
There is a canonical isomorphism which respects the products 
$$
S\H_*(\cH_\vee)\simeq S\H_*(\cH_\wedge). 
$$
\end{theorem}

\begin{proof}
The proof is essentially the same as that of Step~1 in Theorem~\ref{thm:Rabinowitz-cone}. 

Consider $\lambda<0<\mu$. Then we have a homotopy commutative diagram of continuation maps 
$$
\xymatrix{
FC_*(H_{\mu,\lambda})[n]\ar[d]_{c_\wedge} \ar@<.75ex>[rr]^-\pi_-\sim && 
FC_*(L_\lambda)[n] \ar[d]^{c_\vee}  \ar@<.75ex>[ll]^-\iota \\
FC_*(L_\mu)[n] \ar@<.75ex>[rr]^-p_-\sim&& FC_*(H_{\lambda,\mu})[n] \ar@<.75ex>[ll]^-i
}
$$
in which the horizontal maps are chain homotopy equivalences. (The maps $\pi$ and $p$ preserve the filtration, but the maps $\iota$ and $i$ do not. However, in the proof below we will use the total complexes for suitable choices of the parameters, which will palliate to this ailment.) The left vertical map $c_\wedge=c_{(\mu,\lambda),(\mu,\mu)}$ is involved in the definition of $S\H_*(\cH_\wedge)$. The right vertical map $c_\vee=c_{(\lambda,\lambda),(\lambda,\mu)}$ is involved in the definition of $S\H_*(\cH_\vee)$. See Figure~\ref{fig:cone-duality-source}. 

\begin{figure} [ht]
\centering
\input{cone-duality-source.pstex_t}
\caption{Duality theorem via cones: continuation diagram at the source.}
\label{fig:cone-duality-source}
\end{figure} 

Consider also the homotopy commutative diagram of continuation maps (Figure~\ref{fig:cone-duality-target})
$$
\xymatrix{
FC_*(H_{2\mu,\lambda+\mu})[n]\ar[d]_{c_\wedge} \ar@<.75ex>[rr]^-\pi_-\sim && 
FC_*(H_{2\lambda,\lambda+\mu})[n] \ar[d]^{c_\vee}  \ar@<.75ex>[ll]^-\iota \\
FC_*(L_{2\mu})[n] \ar@<.75ex>[rr]^-p_-\sim&& FC_*(H_{2\lambda,2\mu})[n] \ar@<.75ex>[ll]^-i
}
$$

\begin{figure} [ht]
\centering
\input{cone-duality-target.pstex_t}
\caption{Duality theorem via cones: continuation diagram at the target.}
\label{fig:cone-duality-target}
\end{figure} 

By Proposition~\ref{prop:arity2homotopyinvariance} and the subsequent discussion on homotopy invariance for $A_2$-triples, we obtain an isomorphism between product structures 
\begin{equation}\label{eq:iso-product-duality}
m_{\big((\mu,\lambda),(\mu,\mu)\big), \big((2\mu,\lambda+\mu),(2\mu,2\mu)\big)} \simeq 
m_{\big((\lambda,\lambda),(\lambda,\mu)\big), \big((2\lambda,\lambda+\mu),(2\lambda,2\mu)\big)}.
\end{equation}

In order to conclude the proof, we fix an action interval $(a,b)$ with $a<0<b$ and consider the parameter values $\lambda=a$, $\mu=b$ in the previous setup. A stabilization argument as in Example~\ref{example:Hlambdamu} shows that we have isomorphisms 

{\scriptsize{
$$
\xymatrix{
H_*(Cone(c_{(a,a),(a,b)})\otimes H_*(Cone(c_{(a,a),(a,b)}) \ar[rr]^-{m}\ar[d]^\simeq && H_*(Cone(c_{(2a,a+b),(2a,2b)}) \ar[d]^\simeq \\
S\H_*^{(a,b)}(\cH_\vee)\otimes S\H_*^{(a,b)}(\cH_\vee) \ar[rr]^-{m}&& S\H_*^{(a+b,2b)}(\cH_\vee)
}
$$
}}

and 

{\scriptsize{
$$
\xymatrix{
H_*(Cone(c_{(b,a),(b,b)})\otimes H_*(Cone(c_{(b,a),(b,b)}) \ar[rr]^-{m}\ar[d]^\simeq && H_*(Cone(c_{(2b,a+b),(2b,2b)}) \ar[d]^\simeq \\
S\H_*^{(a,b)}(\cH_\wedge)\otimes S\H_*^{(a,b)}(\cH_\wedge) \ar[rr]^-{m}&& S\H_*^{(a+b,2b)}(\cH_\wedge)
}
$$
}}

The top lines in the above two diagrams are isomorphic by~\eqref{eq:iso-product-duality}, and we infer the isomorphism of the bottom lines. This isomorphism is compatible with action truncation maps and yields an isomorphism of rings 
$$
S\H_*(\cH_\vee)\simeq S\H_*(\cH_\wedge).
$$
\end{proof}

\subsection{Poincar\'e duality theorem}

Now we can state the cone version of Poincar\'e duality. Recall that $\wh{SH}_*(\p W)[n-1]$ is canonically isomorphic to $SH^{1-n-*}(\p W)$. 

\begin{theorem}[Poincar\'e duality redux] \label{thm:PD-cones}
We have a canonical isomorphism of rings 
$$
(S\H_*(\p W),\mu) \simeq (\wh{SH}_*(\p W)[n-1],\usigma). 
$$
\end{theorem}

\begin{remark}The unitality of the ring $(\wh{SH}_*(\p W)[n-1],\usigma)$ also follows from the above isomorphism. However, this point of view is roundabout and the direct description of the unit given in~\S\ref{sec:unit} is important for applications.  

Another proof of the existence of the unit comes---in view of the isomorphism $S\H_*(\cH_\wedge)\simeq \wh{SH}_*(\p W)[n-1]$ below---from a discussion of unitality for products on cones. 
More specifically, given an $A_2$-triple $(\cM,c,\cA)$ and assuming that the algebra $\cA$ is unital, one can write down conditions 
under which the ring $(Cone(c),m)$ is unital with unit equal to $(1,0)$, $1\in\cA$. 
More generally, this is related to the notion of homological unitality for $A_\infty$-algebras. 
\end{remark}

Theorem~\ref{thm:PD-cones} is an immediate consequence of Theorems~\ref{thm:duality-cone} and~\ref{thm:Rabinowitz-cone} together with

\begin{theorem}\label{thm:PD-cones2}
We have a canonical isomorphism of rings 
$$
S\H_*(\cH_\wedge)\simeq \wh{SH}_*(\p W)[n-1]. 
$$
\end{theorem}

\begin{proof}
The proof follows exactly the same lines as those of Theorem~\ref{thm:Rabinowitz-cone}. Given a Hamiltonian $H_{\mu,\lambda}$ as in the definition of $\wh{SH}_*(\p W)$, its $1$-periodic orbits are of two types: type $F$ located in a neighborhood of the region $\{r\le 1/2\}$ or type $I$ located in a neighborhood of the region $\{r=1\}$. Accordingly, the free module $FC_*(H_{\mu,\lambda})$ splits as a direct sum $FC_*^I(H_{\mu,\lambda})\oplus FC_*^F(H_{\mu,\lambda})$, and $FC_*^I(H_{\mu,\lambda})$ is a subcomplex, while $FC_*^F(H_{\mu,\lambda})$ is a quotient complex. There is a canonical identification $FC_*^F(H_{\mu,\lambda})\equiv FC_*(L_\mu)$.

Denote $c_\wedge=c_{(\mu,\lambda),(\mu,\mu)}:FC_*(H_{\mu,\lambda})[n]\to FC_*(L_\mu)[n]$ the continuation map. We choose the homotopy from $H_{\mu,\lambda}$ to be constant in the region $\{r\le 3/4\}$, so that $c_\wedge$ coincides with the projection $
p_{H_{\mu,\lambda}}:FC_*(H_{\mu,\lambda})[n]\to FC_*^F(H_{\mu,\lambda})[n]$. 
It is then a general fact that the inclusion $\iota:FC_*^I(H_{\mu,\lambda})[n-1] = \ker p_{H_{\mu,\lambda}}[-1]\hookrightarrow Cone(p_{H_{\mu,\lambda}})$ is a chain homotopy equivalence which preserves the action filtration, and so does its explicit homotopy inverse (see~\cite[Lemma~4.3]{CO} and Example~\ref{example:quotients}). 
We therefore obtain chain homotopy equivalences 
$$
\iota^{(a,b)}:FC_*^{I,(a,b)}(H_{\mu,\lambda})[n-1] \stackrel{\sim}\longrightarrow Cone(p_{H_{\mu,\lambda}})^{(a,b)} = Cone(c_\wedge)^{(a,b)} 
$$
for all action intervals $(a,b)$. 

Let us now fix such a finite action interval $(a,b)$. For $\mu\gg 0$ the action of the orbits in the group $F$ rises above the action window. Thus the only elements in $Cone(p_{H_{\mu,\lambda}})^{(a,b)}$ are of the form $(0,\bar x)$, where $\bar x\in FC_*(H_{\mu,\lambda})[n-1]$, and actually $\bar x\in FC_*^I(H_{\mu,\lambda})[n-1]$. The product of two such elements in $Cone(p_{H_{2\mu,\lambda+\mu}})$ is therefore also represented for $\mu\gg 0$ by elements in $FC_*(H_{2\mu,\lambda+\mu})[n-1]$, and actually in $FC_*^I(H_{2\mu,\lambda+\mu})[n-1]$. We thus have 
$$
\iota^{(a+b,2b)}\usigma(\bar x,\bar x')=(0,\usigma(\bar x,\bar x'))=m((0,\bar x),(0,\bar x')).
$$
Thus $\iota$ interchanges in the relevant action window the product $\usigma$ on $FH_*^{(a,b)}(H_{\lambda,\mu})[n-1]$ with the product $m$ on $Cone(p_{H_{\mu,\lambda}})^{(a,b)}$. 

These identifications and products are compatible with the limits involved in the definitions of $S\H_*(\cH_\wedge)$ and $\wh{SH}_*(\p W)[n-1]$, so the desired isomorphism of rings follows. 
This concludes the proof of Theorem~\ref{thm:PD-cones2}, and therefore of Theorem~\ref{thm:PD-cones}.
\end{proof}

The products $\mu$ and $\usigma$ preserve the action filtration at chain level. As a consequence, the homology groups truncated in negative values of the action $S\H_*^{<0}(\p W)$ and $\wh{SH}_*^{<0}(\p W)[n-1]$ inherit products still denoted $\mu$ and $\usigma$.  
We refer to~\cite{CO} for the formal definitions of $S\H_*^{<0}(\p W)$ and $\wh{SH}_*^{<0}(\p W)$. The following statement is a direct consequence of the fact that the isomorphism from Theorem~\ref{thm:PD-cones} preserves the action filtration. 

\begin{corollary} \label{cor:PD-negative-action}
We have a canonical isomorphism of rings 
$$
(S\H_*^{<0}(\p W),\mu) \simeq (\wh{SH}_*^{<0}(\p W)[n-1],\usigma). 
$$
\qed
\end{corollary}

\subsection{The two Poincar\'e duality theorems are the same}

Recall that $\wh{SH}_*(\p W)=SH_*(\p W\times I,\p W\times \p I)\simeq SH^{-*}(\p W)$. We proved in the Poincar\'e Duality Theorem~\ref{thm:coh-product}(i) that $SH^*(\p W)$ carries a product of degree $n-1$, or alternatively that $\wh{SH}_*(\p W)[n-1]$ carries a product of degree $0$. Part (ii) of the Poincar\'e Duality Theorem~\ref{thm:coh-product} can then be rephrased as an isomorphism of rings $S\H_*(\p W)\simeq \wh{SH}_*(\p W)[n-1]$. On the other hand, we constructed in~\S\ref{sec:alternative_product} another degree $0$ product on $\wh{SH}_*(\p W)[n-1]$ and the Poincar\'e Duality Theorem Redux~\ref{thm:PD-cones} also provides an isomorphism of rings $S\H_*(\p W)\simeq \wh{SH}_*(\p W)[n-1]$. 

\begin{proposition} \label{prop:products_coincide}
The isomorphisms $S\H_*(\p W)\simeq \wh{SH}_*(\p W)[n-1]$ from Theorems~\ref{thm:coh-product} and~\ref{thm:PD-cones} coincide.  
\end{proposition}

\begin{corollary} \label{cor:products_coincide}
The two products on $\wh{SH}_*(\p W)[n-1]$, defined in Theorem~\ref{thm:coh-product} and in~\S\ref{sec:alternative_product},  coincide. \qed\end{corollary}

\begin{proof}[Proof of Proposition~\ref{prop:products_coincide}] Consider parameters $\mu>0>\lambda$ and denote by $\lambda^-$ a real number slightly smaller but very close to $\lambda$. Denote $H_{\lambda^-,(\mu,\lambda)}$ 
a Hamiltonian obtained from $L_\lambda$ by replacing the linear part of slope $\lambda$ on the interval $[1/2,1]$ by a ``dent" of slopes $(\lambda^-,\mu)$, i.e. a continuous function which is linear of slope $\lambda^-$ on $[1/2,r_0]$ and linear of slope $\mu$ on $[r_0,1]$ for a suitable value $r_0=r_0(\lambda^-,\mu)$. See Figure~\ref{fig:Hlambdaminusmulambda}. 

\begin{figure} [ht]
\centering
\input{Hlambdaminusmulambda.pstex_t}
\caption{The Hamiltonian $H_{\lambda^-,(\mu,\lambda)}$.}
\label{fig:Hlambdaminusmulambda}
\end{figure} 

We use the graphical notation 
$$
H_{\lambda^-,(\mu,\lambda)}=\Ham{$\lambda^-$}{$\mu$}{$\lambda$}, \qquad H_{\lambda^-,\mu}=\Hamlambdamu{$\lambda^-$}{$\mu$},\qquad L_\lambda=\HamLlambda{$\lambda$} \, . 
$$
We also denote the corresponding Floer complexes $FC_*(\Ham{$\lambda^-$}{$\mu$}{$\lambda$})$ etc. The module $FC_*(\Ham{$\lambda^-$}{$\mu$}{$\lambda$})[n]$ splits as a direct sum $FC_*(\Hamwedge{$\mu$}{$\lambda$})[n]\oplus FC_*(\HamLlambda{$\lambda^-$})[n]\oplus FC_*(\Hamvee{$\lambda^-$}{$\mu$})[n]$ and the differential is upper triangular with respect to this splitting. Here the factors denote respectively the orbits appearing in the concave part, in the neighborhood of $\{r\le 1/2\}$, and in the convex part. We denote the diagonal terms of the differential $\p_{\hamwedge}$, $\p_{\hamLlambda}$, $\p_{\hamvee}$ and the mixed terms $\p_{\hamvee,\hamLlambda}$ etc. 
We use the same subscripts for the components of maps acting between complexes that are split in this way. 

We prove the statement of the Proposition in an arbitrary finite action window $(a,b)$. The statement in the limit $a\to-\infty$, $b\to \infty$ follows by arguments similar to the ones encountered before. Also as before, it is enough to discuss the case of a single set of parameters $\mu\gg 0\gg \lambda$. We first claim that the isomorphism from the Duality Theorem Redux~\ref{thm:PD-cones} is described at finite energy as the composition of the chain homotopy equivalences 
\begin{align*}
FC_*(\Hamwedge{$\mu$}{$\lambda$})[n-1] & \simeq Cone\Big(proj: FC_*(\Ham{$\lambda^-$}{$\mu$}{$\lambda$})[n] \to FC_*(\Hamlambdamu{$\lambda^-$}{$\mu$})[n]\Big) \cr
& \simeq Cone\Big(incl: FC_*(\HamLlambda{$\lambda$})[n]\to FC_*(\Hamlambdamu{$\lambda$}{$\mu$})[n]\Big) \cr 
& \simeq FC_*(\Hamvee{$\lambda$}{$\mu$})[n].
\end{align*}
Indeed, although the shapes of Hamiltonians used in that proof were slightly different, their slopes at infinity were the same as the ones of the Hamiltonians used above, so that the claim follows by homotopy invariance of the cone construction. On the other hand, the isomorphism from the Poincar\'e Duality Theorem~\ref{thm:coh-product} was induced by the $(\Hamvee{}{},\Hamwedge{}{})$-component of the differential of the Floer complex $FC_*(\Ham{}{}{})[n]$. We are thus left to show that the above composition of chain homotopy equivalences induces in the action window $(a,b)$ the same map in homology, denoted 
$$
\Phi^{(a,b)}:FC_*^{(a,b)}(\Hamvee{$\lambda$}{$\mu$})[n]\to FC_{*-1}^{(a,b)}(\Hamwedge{$\mu$}{$\lambda$})[n].
$$ 

We write 
\begin{align*}
Cone\Big(incl: & FC_*(\HamLlambda{$\lambda$}\,)[n] \to FC_*(\Hamlambdamu{$\lambda$}{$\mu$})[n]\Big) \cr
 & = FC_*(\Hamlambdamu{}{})[n]\oplus FC_*(\HamLlambda{}\,)[n-1] \cr
 & = FC_*(\HamLlambda{}\,)[n]\oplus FC_*(\Hamvee{}{})[n] \oplus FC_{*-1}(\HamLlambda{}\,)[n-1],
\end{align*}
\begin{align*}
Cone\Big(proj: & FC_*(\Ham{$\lambda^-$}{$\mu$}{$\lambda$})[n] \to FC_*(\Hamlambdamu{$\lambda^-$}{$\mu$})[n]\Big) \cr
 & = FC_*(\Hamlambdamu{}{})[n]\oplus FC_*(\Ham{}{}{})[n-1] \cr
 & = FC_*(\Hamlambdamu{}{})[n]\oplus FC_*(\Hamwedge{}{})[n-1] \oplus FC_*(\Hamlambdamu{}{})[n-1].
\end{align*}
The above composition is explicitly expressed in matrix form as follows (for the middle map we only write a $2\times 2$ matrix for readability):
$$
\Phi= \left(\begin{array}{ccc} \p_{\hamlambdamu,\hamwedge} & \one & 0 \end{array}\right) 
\left(\begin{array}{cc}\one_{\hamlambdamu}& \cH_{\hamLlambda,\hamlambdamu}\\0 & c_{\hamLlambda,\ham}\end{array}\right)
\left(\begin{array}{c}0 \\\one \\-\p_{\hamvee,\hamLlambda}\end{array}\right). 
$$
Here $\cH_{\hamLlambda,\hamlambdamu}$ denotes a homotopy 
between two possible continuation maps from $\HamLlambda{}$ to $\Hamlambdamu{}{}$ as in the discussion following Definition~\ref{defi:homotopy_retract_triple}, and $c_{\hamLlambda,\ham}$ is a continuation map induced by a small homotopy. The map $\cH_{\hamLlambda,\hamlambdamu}$ decreases the action, and the map $c_{\hamLlambda,\ham}$ distorts the action by an arbitrarily small amount. 

Given an element $A\in FC_*(\Hamvee{$\lambda$}{$\mu$})$ its image under $\Phi$ is 
$$
\Phi(A) = \p_{\hamvee,\hamwedge}A - \p_{\hamlambdamu,\hamwedge}\cH_{\hamLlambda,\hamlambdamu}\p_{\hamvee,\hamLlambda} A - c_{\hamLlambda,\ham}\p_{\hamvee,\hamLlambda} A.  
$$
The point now is that we work in a finite action window $(a,b)$. For a choice of the parameter $\lambda$ such that $\lambda/2< a$, all the generators of the complex $FC_*(\HamLlambda{$\lambda$}\,)$ have action $<a$. Since $c_{\hamLlambda,\ham}$ distorts the action by an arbitrarily small amount and all the other maps involved in the expression of $\Phi(A)$ decrease the action, it follows that, given $A\in FC_*^{(a,b)}(\Hamvee{$\mu$}{$\lambda$})$, the truncation of $\Phi(A)$ in action $(a,b)$ is  
$$
\Phi^{(a,b)}(A) = \p_{\hamvee,\hamwedge}^{(a,b)}A. 
$$
This proves that the isomorphisms from Theorems~\ref{thm:PD-cones} and~\ref{thm:coh-product} coincide in the finite action range $(a,b)$. As already indicated, the statement in the limit $a\to-\infty$, $b\to\infty$ follows by standard arguments which were already seen before.   
\end{proof}

\section{The pair-of-pants product via varying weights}\label{sec:pop-GH-iso} 

We restrict in this section to the homology in negative action range 
$$
\wh{SH}_*^{<0}(\p W)\cong SH_*^{<0}(W,\p W)\cong SH^{-*}_{>0}(W).
$$

We introduce in~\S\ref{sec:varying-weights}
the {\em varying weights secondary 
  product} $\usigma_{w}$ on $SH_*^{<0}(W,\p W)[n-1]$. We show in~\S\ref{sec:PD-varying-weights-iso} that it coincides with the Poincar\'e duality product $\sigma_{PD}$ and with the continuation product $\usigma_c$. The product $\usigma_w$ is not used elsewhere in the paper but, unlike the Poincar\'e duality product $\sigma_{PD}$ and the continuation product $\usigma_c$, it did appear previously in the literature. Its construction goes back to Seidel and was further explored
in~\cite{Ekholm-Oancea}, see also~\cite{AS-product-structures}. The purpose of this section is to clarify its relationship to the constructions of the present paper. 

\subsection{Definition of the varying weights secondary product}
\label{sec:varying-weights}

Let $\Sigma$ be the genus zero Riemann surface with three punctures, two of them labeled as positive $z^0_+$, $z^1_+$ and the third one labeled as negative $z_-$, endowed with cylindrical ends $[0,\infty)\times S^1$ at the positive punctures and $(-\infty, 0]\times S^1$ at the negative puncture. Denote $(s,t)$, $t\in S^1$ the induced cylindrical coordinates at each of the punctures. Consider a smooth family of $1$-forms $\beta_\epsilon\in \Omega^1(\Sigma)$, $\epsilon\in (0,1)$ satisfying the following conditions: 
\begin{itemize}
\item {\sc (nonnegative)} $d\beta_\epsilon\ge 0$; 
\item {\sc (weights)} $\beta_\epsilon=dt$ near each of the punctures;
\item {\sc (interpolation)} we have $\beta_\epsilon=\epsilon dt$ on $[0,R(\epsilon)]\times S^1$ in the cylindrical end near $z^0_+$, and 
$\beta_\epsilon=(1-\epsilon) dt$ on $[0,R(1-\epsilon)]\times S^1$ in the cylindrical end near $z^1_+$, for some smooth function $R:(0,1)\to \R_{>0}$. In other words, the family $\{\beta_\epsilon\}$ interpolates between a $1$-form which varies a lot near $z^0_+$ and very little near $z^1_+$, and a $1$-form which varies a lot near $z^1_+$ and very little near $z^0_+$;
\item {\sc (neck stretching)} we have $R(\epsilon)\to +\infty$ as $\epsilon\to 0$. 
\end{itemize} 
We can assume without loss of generality that for $\epsilon$ close to
$0$ we have $\beta_\epsilon=f_\epsilon(s)dt$ in the cylindrical end at
the positive puncture $z^0_+$, with $f'_\epsilon\ge 0$, $f_\epsilon=1$
near $+\infty$, and $f_\epsilon=\epsilon$ on $[0,R(\epsilon)]$, and
similarly for $\epsilon$ close to $1$ on the end at $z^1_+$.  

\begin{figure} [ht]
\centering
\input{beta_epsilon.pstex_t}
\caption{Interpolating family of $1$-forms with varying weights.}
\label{fig:beta_epsilon}
\end{figure} 

Let $H:\wh W\to\R$ be a concave smoothing localized near $\p W$ of a Hamiltonian which is zero on $W$ and linear of \emph{negative} slope on $[1,\infty)\times \p W$. The Hamiltonian $H$ further includes a small time-dependent perturbation localized near $\p W$, so that all $1$-periodic orbits are nondegenerate. Assume the absolute value of the slope is not equal to the period of a closed Reeb orbit. Denote $\cP(H)$ the set of $1$-periodic orbits of $H$. The elements of $\cP(H)$ are contained in a compact set close to $W$. 
\begin{remark}
The Hamiltonian $H$ above has the standard shape used in the definition of symplectic homology $SH_*(W,\p W)$. However, the construction can accommodate more general Hamiltonians using methods from~\cite[Lemmas~2.2 and~2.3]{CO}.
\end{remark}
Let $J=(J_\epsilon^z)$, $z\in \Sigma$, $\epsilon\in(0,1)$ be a generic family of compatible almost complex structures, independent of $\epsilon$ and $s$ near the punctures, cylindrical and independent of $\epsilon$ and $z$ in the symplectization $[1,\infty)\times \p W$.
For $x^0,x^1,y\in\cP(H)$ denote 
\begin{align*}
   \MM^1(x^0,x^1;y) := \bigl\{&(\epsilon,u)\;\bigl|\; 
   \epsilon\in(0,1),\ u:\Sigma\to \wh W, \cr
   & (du-X_H\otimes \beta_\epsilon)^{0,1}=0, \cr
   & \lim_{\stackrel{s\to+\infty}{z=(s,t)\to z^i_+}} u(z)=x^i(t),\, i=0,1,\cr
   & \lim_{\stackrel{s\to-\infty}{z=(s,t)\to z_-}}u(z)=y(t)\bigr\}.
\end{align*}
In the symplectization $[1,\infty)\times \p W$ we have $H\le 0$ and therefore $d(H\beta)\le 0$, so that elements of the above moduli space are contained in a compact set. The dimension of the moduli space is 
$$
\dim\, \MM^1(x^0,x^1;y) = \CZ(x^0)+\CZ(x^1)-\CZ(y)-n+1. 
$$
When it has dimension zero the moduli space $\MM^1_{\dim=0}(x^0,x^1;y)$ is compact. When it has dimension $1$ the moduli space $\MM^1_{\dim=1}(x^0,x^1;y)$ admits a natural compactification into a manifold with boundary 
\begin{align*}
   \p\MM^1_{\dim=1}(x^0,x^1;y)
   &= \coprod_{\CZ(x')=\CZ(x^0)-1}\MM(x^0;x')\times\MM^1_{\dim=0}(x',x^1;y) \cr
   &\amalg  \coprod_{\CZ(x')=\CZ(x^1)-1}\MM(x^1;x')\times\MM^1_{\dim=0}(x^0,x';y) \cr
   &\amalg\coprod_{\CZ(y')=\CZ(y)+1}\MM^1_{\dim=0}(x^0,x^1;y')\times\MM(y';y) \cr
   &\amalg\MM^1_{\epsilon=1}(x^0,x^1;y) \amalg\MM^1_{\epsilon=0}(x^0,x^1;y).
\end{align*}
Here $\MM^1_{\epsilon=1}(x^0,x^1;y)$ and $\MM^1_{\epsilon=0}(x^0,x^1;y)$ denote the fibers of the first projection $\MM^1_{\dim=1}(x^0,x^1;y)\to (0,1)$, $(\epsilon,u)\mapsto \epsilon$ near $1$, respectively near $0$. (By a standard gluing argument the projection is a trivial fibration with finite fiber near the endpoints of the interval $(0,1)$.) 

Consider the degree $1$ operation 
$$
\sigma_{w}:FC_*(H)[n]\otimes FC_*(H)[n]\to FC_*(H)[n]
$$
defined on generators by 
$$
\sigma_{w}(x^0\otimes x^1)=\sum_{\CZ(y)=\CZ(x^0)+\CZ(x^1)-n+1} \# \MM^1_{\dim=0}(x^0,x^1;y) y,
$$
where $\# \MM^1_{\dim=0}(x^0,x^1;y)$ denotes the count of elements in the $0$-dimen\-sional moduli space $\MM^1_{\dim=0}(x^0,x^1;y)$ with signs determined by a choice of coherent orientations. Consider also the degree $0$ operations 
$$
\sigma_{w}^i:FC_*(H)[n]\otimes FC_*(H)[n]\to FC_*(H)[n], \qquad i=0,1
$$
defined on generators by 
$$
\sigma_{w}^i(x^0\otimes x^1)=\sum_{\CZ(y)=\CZ(x^0)+\CZ(x^1)-n} \# \MM^1_{\epsilon=i}(x^0,x^1;y) y,
$$
where $\# \MM^1_{\epsilon=i}(x^0,x^1;y)$ denotes the count of elements in the $0$-dimen\-sional moduli space $\MM^1_{\epsilon=i}(x^0,x^1;y)$ with signs determined by a choice of coherent orientations. 

The formula for $\p\MM^1_{\dim=1}(x^0,x^1;y)$ translates into the algebraic relation  
\begin{equation} \label{eq:product_with_weights}
\p ^F\sigma_{w} + \sigma_{w}(\p ^F\otimes\mathrm{id}+ \mathrm{id}\otimes \p ^F)=\sigma_{w}^1-\sigma_{w}^0.
\end{equation}
We now claim that  
\begin{equation} \label{eq:product_with_weights_vanishing_at_boundary}
\sigma_{w}^0| FC_*^{<0}(H)[n]\otimes FC_*(H)[n] = 0,\quad \sigma_{w}^1|FC_*(H)[n]\otimes FC_*^{<0}(H)[n] = 0. 
\end{equation}
To prove the claim for $\sigma_{w}^0$, note that this map can be
expressed as a composition $\mu\circ (c\otimes\mathrm{id})$, where
$\mu:FC_*(\epsilon H)[n]\otimes FC_*(H)[n]\to FC_*(H)[n]$ is a
pair-of-pants product, and $c:FC_*(H)[n]\to FC_*(\epsilon H)[n]$ is a
continuation map. The action decreases along continuation maps, hence
$c(FC_*^{<0}(H)[n])\subset FC_*^{<0}(\epsilon H)[n]$. At the same time this
last group vanishes because $\epsilon H$ has no nontrivial
$1$-periodic orbits of negative action for $\epsilon$ small
enough. The argument for $\sigma_{w}^1$ is similar.  

It follows that $\sigma_{w}$ restricts to a degree $1$ chain map
\begin{equation}\label{eq:sigmaF}
   \sigma_{w}:FC_*^{<0}(H)[n]\otimes FC_*^{<0}(H)[n]\to FC_*^{<0}(H)[n].
\end{equation}
(This map lands in $FC_*^{<0}(H)[n]$ for action reasons.) Passing to the limit we obtain a degree $1$ product 
$$
\sigma_w:SH_*^{<0}(W,\p W)[n]\otimes SH_*^{<0}(W,\p W)[n]\to SH_*^{<0}(W,\p W)[n]. 
$$ 
Finally we apply a shift as in~\eqref{eq:shift-sigma}, namely $\usigma_w=-\sigma_w[-1,-1;-1]$, in order to obtain a degree $0$ product 
$$
\usigma_w:SH_*^{<0}(W,\p W)[n-1]^{\otimes 2}\to SH_*^{<0}(W,\p W)[n-1].
$$ 
Explicitly $\usigma_w(\bar x,\bar x')=-(-1)^{|\bar x|}\overline{\sigma(x,x')}$, where $x,x'\in SH_*^{<0}(W,\p W)[n]$ and $\bar x, \bar x'\in SH_*^{<0}(W,\p W)[n-1]$ denote their shifted images. We call $\usigma_w$ the {\bf varying weights degree $-n+1$  secondary product on $SH_*^{<0}(W,\p W)$}.  

Equivalently, in view of the canonical isomorphism $SH_*^{<0}(W,\p W)\simeq SH^{-*}_{>0}(W)$ from~\cite{CO}, the above construction defines a {\bf degree $n-1$ secondary product on $SH^*_{>0}(W)$}, denoted $\usigma^w$. At the level of moduli spaces this is described by exchanging the roles of the positive and negative punctures, and reversing the sign of the Hamiltonian. Thus one considers curves with 2 negative punctures with varying weights treated as cohomological inputs and 1 positive puncture treated as a cohomological output. In this framework, the relevant Floer equation involves Hamiltonians on $\wh W$ with positive slope on $[1,\infty) \times \p W$. 

In yet another equivalent formulation, the above construction defines a {\bf degree $-n+1$ secondary coproduct on $SH_*^{>0}(W)$}, denoted $\underline{\lambda}_w$. The moduli spaces are the same as for the secondary product on $SH^*_{>0}(W)$ (2 negative punctures with varying weights and 1 positive puncture), except that the positive puncture is treated as a homological input and the negative punctures are treated as homological outputs.

\subsection{Secondary products: varying weights and continuation maps}\label{sec:PD-varying-weights-iso}

In view of Corollary~\ref{cor:PD-negative-action} and Proposition~\ref{prop:products_coincide}, the negative action homology group $SH_*^{<0}(W,\p W)[n-1]\cong \wh{SH}_*^{<0}(\p W)[n-1]$ carries two other products of degree $0$ which coincide~: 
\begin{itemize}
\item the {\bf Poincar\'e duality product} $\sigma_{PD}$, induced from the primary product on Rabinowitz Floer homology via the Poincar\'e duality isomorphism from Theorem~\ref{thm:coh-product} restricted in negative action. 
\item the {\bf continuation product} $\usigma_c=\usigma$ constructed in~\S\ref{sec:alternative_product} restricted in negative action. 
\end{itemize}

\begin{proposition}\label{prop:cont-weight}
The continuation product $\usigma_c$ and the varying weights product $\usigma_w$ coincide on $SH_*^{<0}(W,\p W)[n-1]$. 
\end{proposition}

\begin{proof} Going back to the definition of the unshifted varying weights product $\sigma_w$, we recall that the vanishing of the left boundary term in~\eqref{eq:product_with_weights_vanishing_at_boundary}, i.e.
\begin{equation*} 
\sigma_{w}^0| FC_*^{<0}(H)[n]\otimes FC_*(H)[n] = 0,
\end{equation*}
was ensured by the fact that $\sigma_{w}^0$ could be
expressed as a composition $\mu\circ (c\otimes\mathrm{id})$, where
$\mu:FC_*(\epsilon H)[n]\otimes FC_*(H)[n]\to FC_*(H)[n]$ is a
pair-of-pants product, and $c:FC_*(H)[n]\to FC_*(\epsilon H)[n]$ is a
continuation map. Since the action decreases along continuation maps, we have
$c(FC_*^{<0}(H)[n])\subset FC_*^{<0}(\epsilon H)[n]$, and the
last group vanishes because $\epsilon H$ has no nontrivial
$1$-periodic orbits of negative action for $\epsilon$ small
enough. Similarly, the boundary term $\sigma_{w}^1$ can be expressed as $\mu\circ (\mathrm{id}\otimes c)$. 

In the case of the unshifted continuation product $\sigma$, the boundary terms are expressed as $\mu\circ (c'\otimes\mathrm{id})$, respectively $\mu\circ (\mathrm{id}\otimes c')$, where $c':FC_*(H)[n]\to FC_*(K_\nu)[n]$ is the continuation map towards a Hamiltonian $K_\nu$ which vanishes on $W$ and has positive slope $\nu$ on $[1,\infty)\times \p W$. The continuation map $c'$ factors as $FC_*(H)[n]\stackrel{c}\longrightarrow FC_*(\epsilon H)[n]\to FC_*(K_\nu)[n]$ and, when restricted to negative action, vanishes for all positive values of $\nu$. As such, the $1$-parameter family of Floer problems with source a genus $0$ Riemann surface with two positive punctures and one negative puncture which defines $\sigma$ can be chosen as follows: on a first interval we interpolate near the first positive puncture from the continuation map $c'$ to the continuation map $c$. On a second interval we follow the $1$-parameter family which defines $\sigma_w$. And on a third interval we interpolate near the second positive puncture from the continuation map $c$ to the continuation map $c'$. When restricted to negative action, the first and third parametrizing intervals bring no contribution, so that $\sigma_c=\sigma_w$ for this choice of defining data. 

Finally, the continuation product $\usigma_c$ and the varying weights product $\usigma_w$ are defined by the same shift $-[-1,-1;-1]$ from $\sigma_c$ and respectively $\sigma_w$, so that they coincide as well. 
\end{proof}

\section{$A_2^+$-structures} \label{sec:A2+}

The goal of this section is to define the notion of an $A_2^+$-structure on a chain complex $\cA$ and show how it induces an $A_2$-triple $(\cA^\vee,c,\cA)$. 
The $R$-module $\cA$ need not be free or of finite rank. Of particular interest for the applications to string topology in~\cite{CHO-MorseFloerGH} is the case where $\cA$ is free over $R$, but possibly of infinite rank. We also discuss morphisms of $A_2^+$-algebras. 

\subsection{$A_2^+$-algebras} 

Let $(\cA,\p)$ be a dg $R$-module. We denote by $\tau:\cA\otimes \cA \to \cA\otimes \cA$ the twist $a\otimes b\mapsto (-1)^{|a|\cdot |b|}b\otimes a$. 

\begin{definition} \label{defi:A2+structure} 
An \emph{$A_2^+$-structure} on $(\cA,\p)$ consists of the following maps: 
\begin{itemize}
\item \emph{the copairing} $c_0:R\to \cA\otimes \cA$, of degree $0$;
\item \emph{the secondary copairing} $Q_0:R\to \cA\otimes\cA$, of degree $1$;
\item \emph{the product} $\mu:\cA\otimes \cA\to \cA$, of degree $0$;
\item \emph{the secondary coproduct} $\lambda:\cA\to \cA\otimes \cA$, of degree $1$;
\item \emph{the cubic vector} $B:R\to \cA\otimes\cA\otimes \cA$, of degree $2$.    
\end{itemize}
These maps are subject to the following conditions: 
\begin{enumerate}
\item $c_0$ is a cycle.  
\item $c_0$ is symmetric up to a homotopy given by $Q_0$, i.e. 
$$
\tau c_0 - c_0 = [\p,Q_0].
$$
\item $\mu$ is a chain map.
\item $\lambda$ satisfies the relation 
$$
[\p, \lambda]=(\mu\otimes 1)(1\otimes c_0) - (1\otimes \mu)(\tau c_0 \otimes 1).
$$
\item $B$ satisfies the relation 
$$
\p B = (\lambda_{c_0,\tau c_0}\otimes 1 ) \tau c_0 + (312)(\lambda_{\tau c_0,c_0} \otimes 1) \tau c_0 + (231) (\lambda_{\tau c_0,\tau c_0}\otimes 1)c_0.
$$
\end{enumerate}
Here for the last relation we denote $\lambda=\lambda_{c_0,c_0}$ and define 
\begin{align*}
\lambda_{\tau c_0,\tau c_0} & = \lambda_{c_0,c_0} + (\mu\otimes 1)(1\otimes Q_0) - (1\otimes \mu)(\tau Q_0\otimes 1),\\
\lambda_{c_0,\tau c_0} & = \lambda_{c_0,c_0} + (\mu\otimes 1)(1\otimes Q_0),\\
\lambda_{\tau c_0,c_0} & = \lambda_{c_0,c_0} - (1\otimes \mu)(\tau Q_0\otimes 1).
\end{align*}
By $(312)$ we denote the permutation on $\cA^{\otimes 3}$ given by the product of transpositions $\tau_{23}\tau_{12}$,
by $(231)$ the permutation given by the product of transpositions $\tau_{12}\tau_{23}=(\tau_{23}\tau_{12})^2$, 
and by $(123)$ the identity.\footnote{Note that this does not correspond to the cycle notation of permutations.}
The relation satisfied by $B$ can be rewritten 
$$
\p B=\sum_{\sigma\, \mathrm{cyclic}} \sigma \big( (\lambda_{\boldsymbol{a}_\sigma,\boldsymbol{b}_\sigma}\otimes 1)\boldsymbol{c}_\sigma\big),
$$
where the meaning of ``cyclic" is that $\sigma\in\{(123), (231), (312)\}$ and we denote $(\boldsymbol{a}_\sigma,\boldsymbol{b}_\sigma,\boldsymbol{c}_\sigma)=\sigma (c_0,\tau c_0,\tau c_0)$. 
\end{definition}

{\bf Conventions.} We depict the operation $\mu$ as a trivalent tree with two inputs and one output, where the inputs are read in \emph{clockwise} order with respect to the output. We depict the operation $\lambda$ as a trivalent tree with one input and two outputs, where the outputs are read in \emph{counterclockwise} order with respect to the input. We depict the operation $B$ with outputs ordered cyclically \emph{counterclockwise}. This is consistent with the operadic and co-operadic conventions, in which inputs or outputs are \emph{both} read horizontally from left to right. See Figure~\ref{fig:mu-lambda-B}. We depict the inputs and outputs as lying on a circle. 

\begin{figure}
\begin{center}
\includegraphics[width=\textwidth]{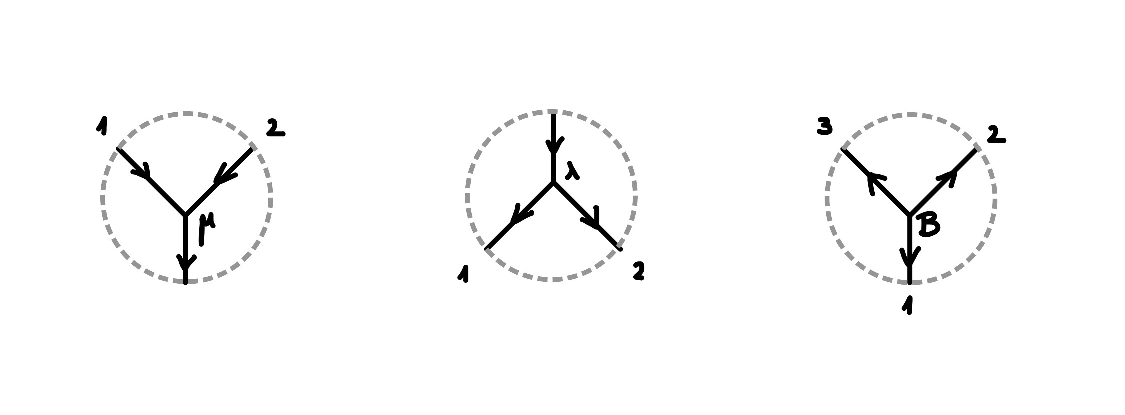}
\caption{Conventions for inputs and outputs of $\mu$, $\lambda$, $B$}
\label{fig:mu-lambda-B}
\end{center}
\end{figure}

The relations involving $c_0$, $Q_0$, $\lambda$, $B$ are depicted in Fig.~\ref{fig:cQ}, \ref{fig:lambda-trees-new} and~\ref{fig:B-new-Q}. 

In these figures and in the sequel we use the following 

{\bf Pictorial convention:} copairings are represented as $\leftrightarrow$, and we always feed their first component as input for some other operations.

The notation $\lambda=\lambda_{c_0,c_0}$ is motivated by Figure~\ref{fig:lambda-trees-new}, in which the copairing $c_0$ appears at both ends of the parametrizing interval. The relations satisfied by the other elements $\lambda_{\boldsymbol{a},\boldsymbol{b}}$, $\boldsymbol{a},\boldsymbol{b}=c_0$ or $\tau c_0$ are depicted in Fig.~\ref{fig:lambda-trees-new-ab}.

\begin{figure}
\begin{center}
\includegraphics[width=.7\textwidth]{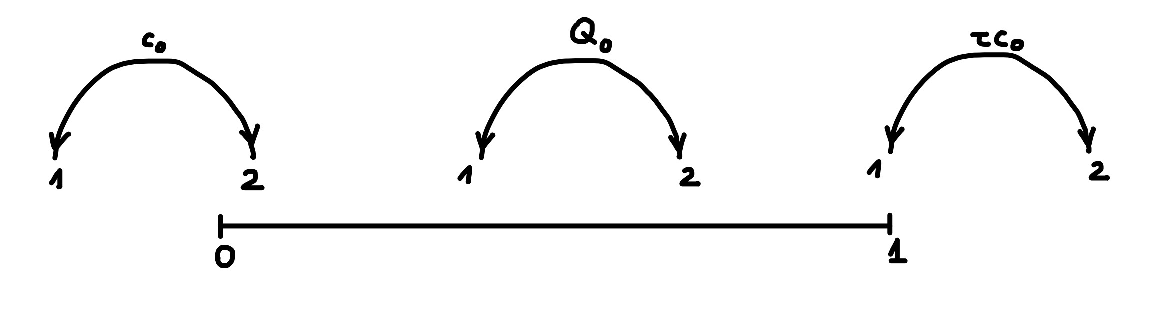}
\caption{The homotopy $Q_0$ between $c_0$ and $\tau c_0$.}
\label{fig:cQ}
\end{center}
\end{figure}


\begin{figure}
\begin{center}
\includegraphics[width=.7\textwidth]{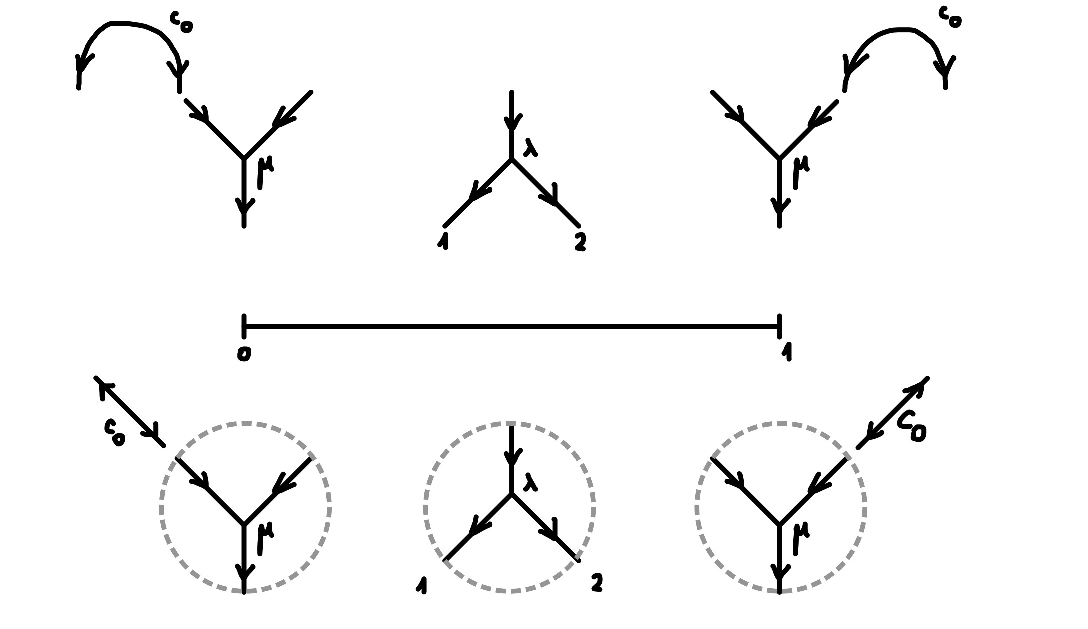}
\caption{Two equivalent descriptions of $\lambda$}
\label{fig:lambda-trees-new}
\end{center}
\end{figure}

\begin{figure}
\begin{center}
\includegraphics[width=\textwidth]{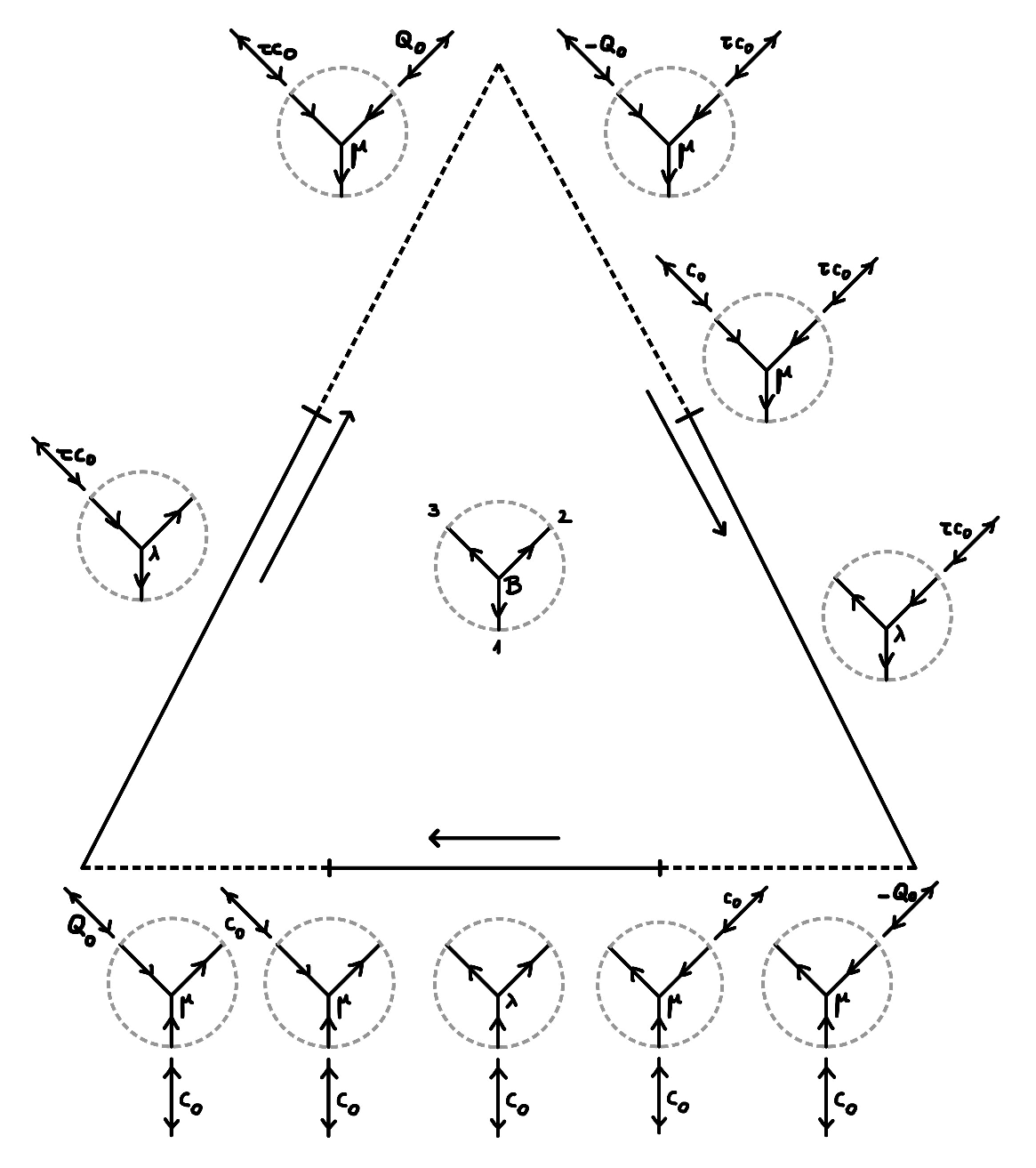}
\caption{The cubic vector $B$}
\label{fig:B-new-Q}
\end{center}
\end{figure}

\begin{figure}
\begin{center}
\includegraphics[width=\textwidth]{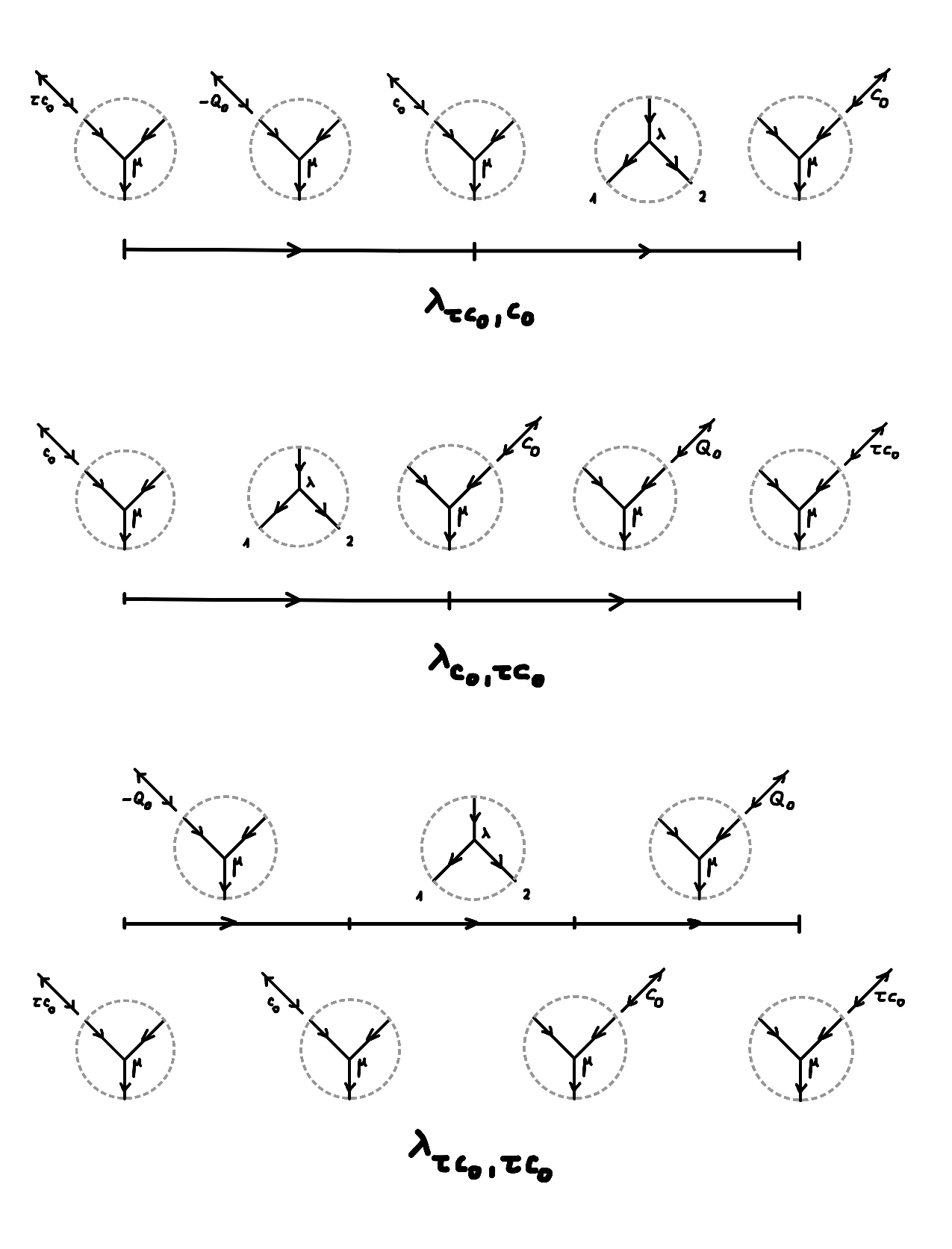}
\caption{The operations $\lambda_{\boldsymbol{a},\boldsymbol{b}}$, $\boldsymbol{a},\boldsymbol{b}=c_0$ or $\tau c_0$}
\label{fig:lambda-trees-new-ab}
\end{center}
\end{figure}

\begin{remark} The copairing $c_0$ gives rise to the map 
$$
c:=(\ev\otimes 1)(1\otimes c_0) :\cA^\vee\to\cA.
$$
That $c_0$ is a cycle implies that $c$ is a chain map. The fact that $c$ originates in $c_0\in \cA\otimes \cA$ implies that $c^\vee$ takes values in
$\cA\cong \iota(\cA)\subset \cA^{\vee\vee}$, where $\iota:\cA\to\cA^{\vee\vee}$ is the canonical map from~\S\ref{sec:consequences}. Moreover, under this identification $c^\vee$ can be expressed in terms of $\tau c_0$ as 
$$
c^\vee=(\ev\otimes 1)(1\otimes\tau c_0). 
$$
The condition $\tau c_0-c_0=[\p,Q_0]$ is therefore equivalent to 
$$
c^\vee-c=[\p,Q],
$$
with 
$$
Q=(\ev\otimes 1)(1\otimes Q_0).
$$
If $\iota$ is an isomorphism, then $c_0$ can be recovered from $c$ by the formula $c_0=(1\otimes c)\ev^\vee$.
\end{remark}

It is useful to spell out the unital variant of the notion of $A_2^+$-structure. 

\begin{definition}
Let $(\cA,\p)$ be a dg $R$-module. A \emph{unital $A_2^+$-structure} on $(\cA,\p)$ consists of the following maps: 
\begin{itemize}
\item \emph{the copairing} $c_0:R\to \cA\otimes \cA$, of degree $0$;
\item \emph{the product} $\mu:\cA\otimes \cA\to \cA$, of degree $0$;
\item \emph{the unit} $\eta:R\to \cA$, of degree $0$;
\item \emph{the secondary coproduct} $\lambda:\cA\to \cA\otimes \cA$, of degree $1$;
\item \emph{the cubic vector} $B:R\to \cA\otimes\cA\otimes \cA$, of degree $2$.
\item \emph{the left- and right unit homotopies} $\mathrm{Left}:\cA\to \cA$ and $\mathrm{Right}:\cA\to\cA$, of degree $1$.
\end{itemize}
These maps are subject to the following conditions: 
\begin{enumerate}
\item $c_0$ and $\eta$ are cycles.  
\item $\mu$ is a chain map.
\item $\eta$ is a two-sided homotopy unit for $\mu$, i.e. 
$$
\mu(\eta\otimes 1)=1+ [\p,\mathrm{Left}],\qquad \mu(1\otimes\eta)=1+ [\p,\mathrm{Right}].
$$
\item $\lambda$ satisfies the relation 
$$
[\p, \lambda]=(\mu\otimes 1)(1\otimes c_0) - (1\otimes \mu)(\tau c_0 \otimes 1).
$$
\item $B$ satisfies the relation (in the notation of Definition~\ref{defi:A2+structure})
$$
\p B=\sum_{\sigma\, \mathrm{cyclic}} \sigma \big( (\lambda_{\boldsymbol{a}_\sigma,\boldsymbol{b}_\sigma}\otimes 1)\boldsymbol{c}_\sigma\big).
$$
\end{enumerate}
\end{definition}

\begin{remark}
A unital $A_2^+$-structure is also an $A_2^+$-structure in the sense of Definition~\ref{defi:A2+structure}. To see this, apply relation~(4) to $\eta$ in order to obtain 
$$
\tau c_0-c_0= [\p,Q_0],
$$ 
where we define the \emph{secondary copairing $Q_0$} as 
$$
Q_0=-\lambda\eta + (\mathrm{Left}\otimes 1)c_0 - (1\otimes\mathrm{Right})\tau c_0:R\to \cA\otimes \cA.
$$

A further simplification arises if there exists a dg submodule $\cA_0\subset \cA$ which is stable under $\mu$, such that $\eta\in \cA_0$, $c_0\in\cA_0\otimes \cA_0$, and $\eta$ is a strict two-sided unit for $\mu$ on $\cA_0$ (more specifically, if the homotopies $\mathrm{Left}$ and $\mathrm{Right}$ vanish on $\cA_0$). In that case it is enough to set 
$$
Q_0=-\lambda\eta. 
$$
\end{remark} 

\begin{remark}
The terminology ``unital $A_2^+$-structure" is motivated as follows. We prove in Proposition~\ref{prop:TQFT+} that an $A_2^+$-structure on $\cA$ determines canonically an $A_2$-triple $(\cA^\vee,c,\cA)$, and further a product structure on $Cone(c)$. For a unital $A_2^+$-structure the product on the cone will be unital.  
\end{remark}

\subsection{From $A_2^+$-algebras to $A_2$-triples} 

\begin{proposition}\label{prop:TQFT+}
An $A_2^+$-structure on $(\cA,\p)$ canonically gives rise to an $A_2$-triple $(\cA^\vee,c,\cA)$. 
\end{proposition}

\begin{proof}
The $A_2$-triple operations have two inputs and one output. According to our conventions, the inputs are to be read \emph{clockwise} with respect to the output. Some of these operations are obtained by dualizing $\lambda$ and $B$, and it is therefore important to reorder the outputs of $\lambda$ and $B$. We therefore define 
$$
\ol{\lambda}=-\tau\lambda, \qquad \ol{B}=-(1\otimes \tau)B.
$$
Thus $\ol{\lambda}$ and $\ol{B}$ satisfy the relations depicted in Figures~\ref{fig:lambdabar-trees-new} and~\ref{fig:Bbar-new-Q}. 
More generally, we denote 
$$
\ol{\lambda}_{\boldsymbol{a},\boldsymbol{b}}=-\tau\lambda_{\boldsymbol{b},\boldsymbol{a}}, \qquad \boldsymbol{a},\boldsymbol{b}=c_0 \,\, \mbox{or }\tau c_0.
$$
In particular $\ol{\lambda}=\ol{\lambda}_{c_0,c_0}$. The relations satisfied by the other elements $\ol{\lambda}_{\boldsymbol{a},\boldsymbol{b}}$, $\boldsymbol{a},\boldsymbol{b}=c_0$ or $\tau c_0$ are depicted in Fig.~\ref{fig:lambdabar-trees-new-ab}.

\begin{figure}
\begin{center}
\includegraphics[width=.7\textwidth]{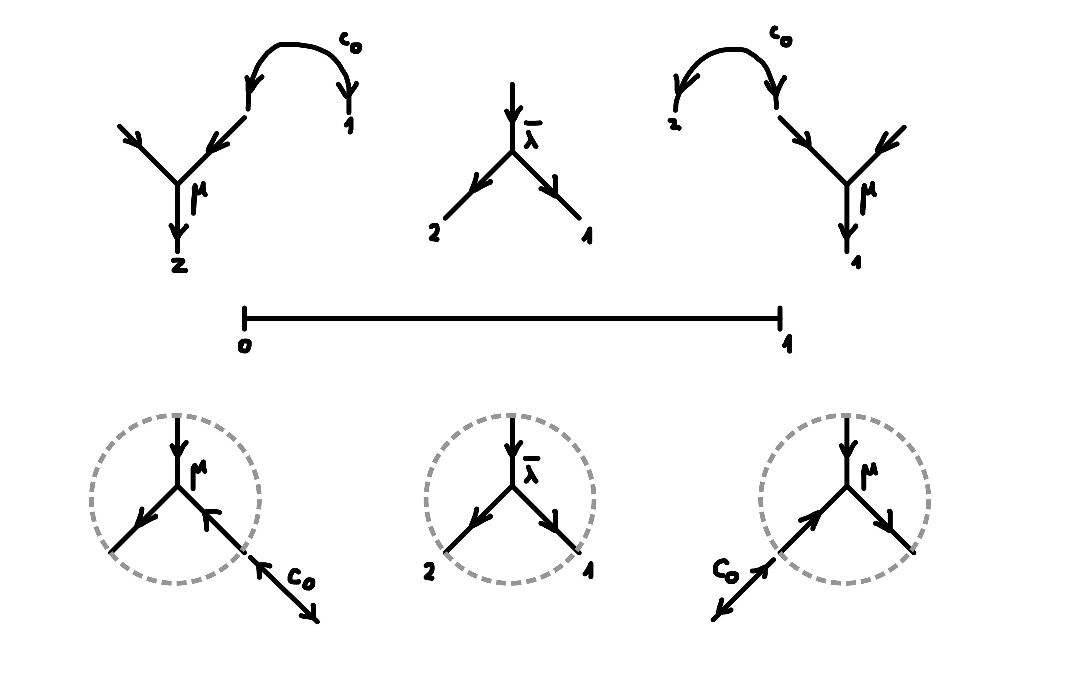}
\caption{The map $\ol{\lambda}$}
\label{fig:lambdabar-trees-new}
\end{center}
\end{figure}

\begin{figure}
\begin{center}
\includegraphics[width=\textwidth]{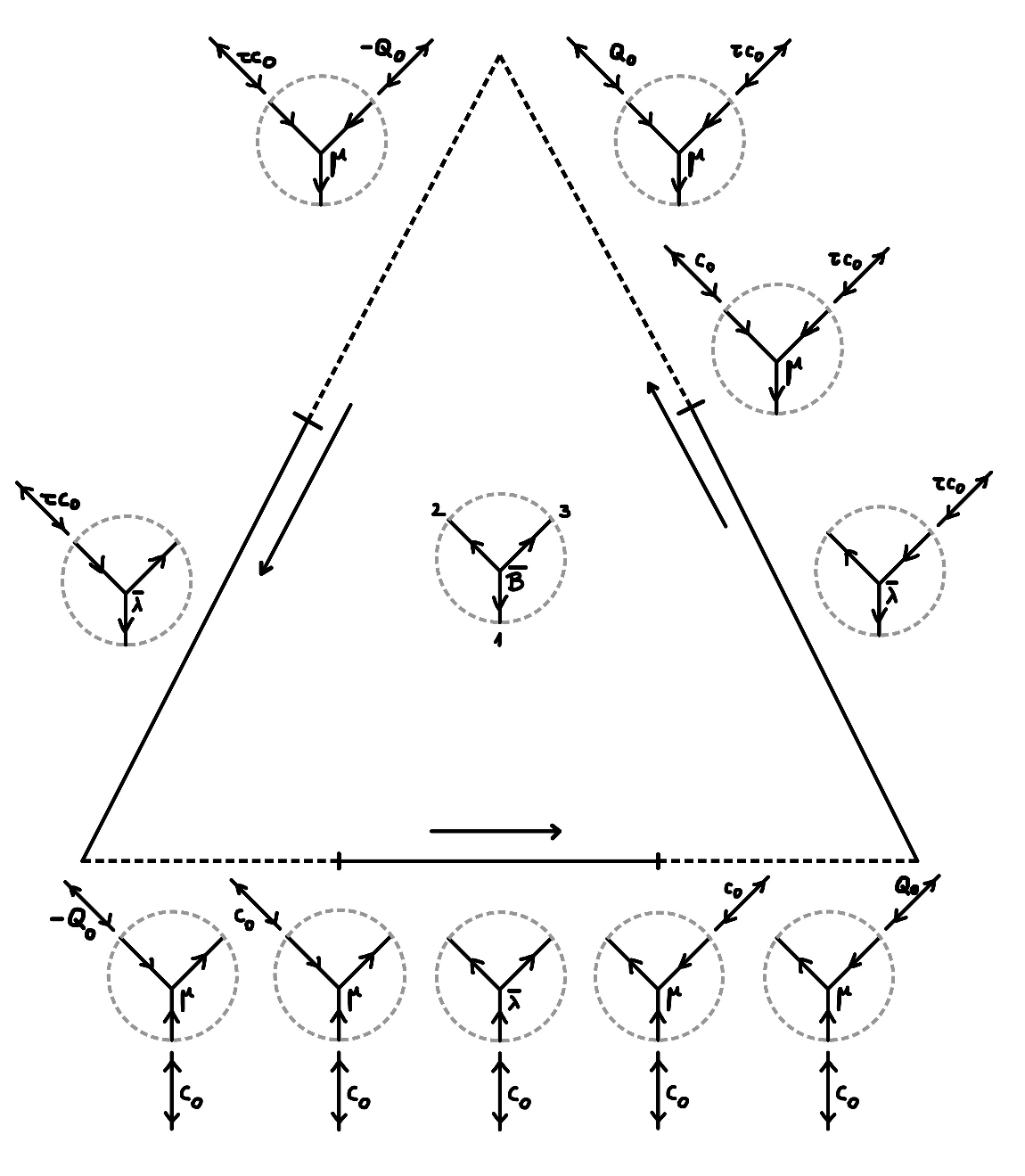}
\caption{The map $\ol{B}$}
\label{fig:Bbar-new-Q}
\end{center}
\end{figure}

\begin{figure}
\begin{center}
\includegraphics[width=\textwidth]{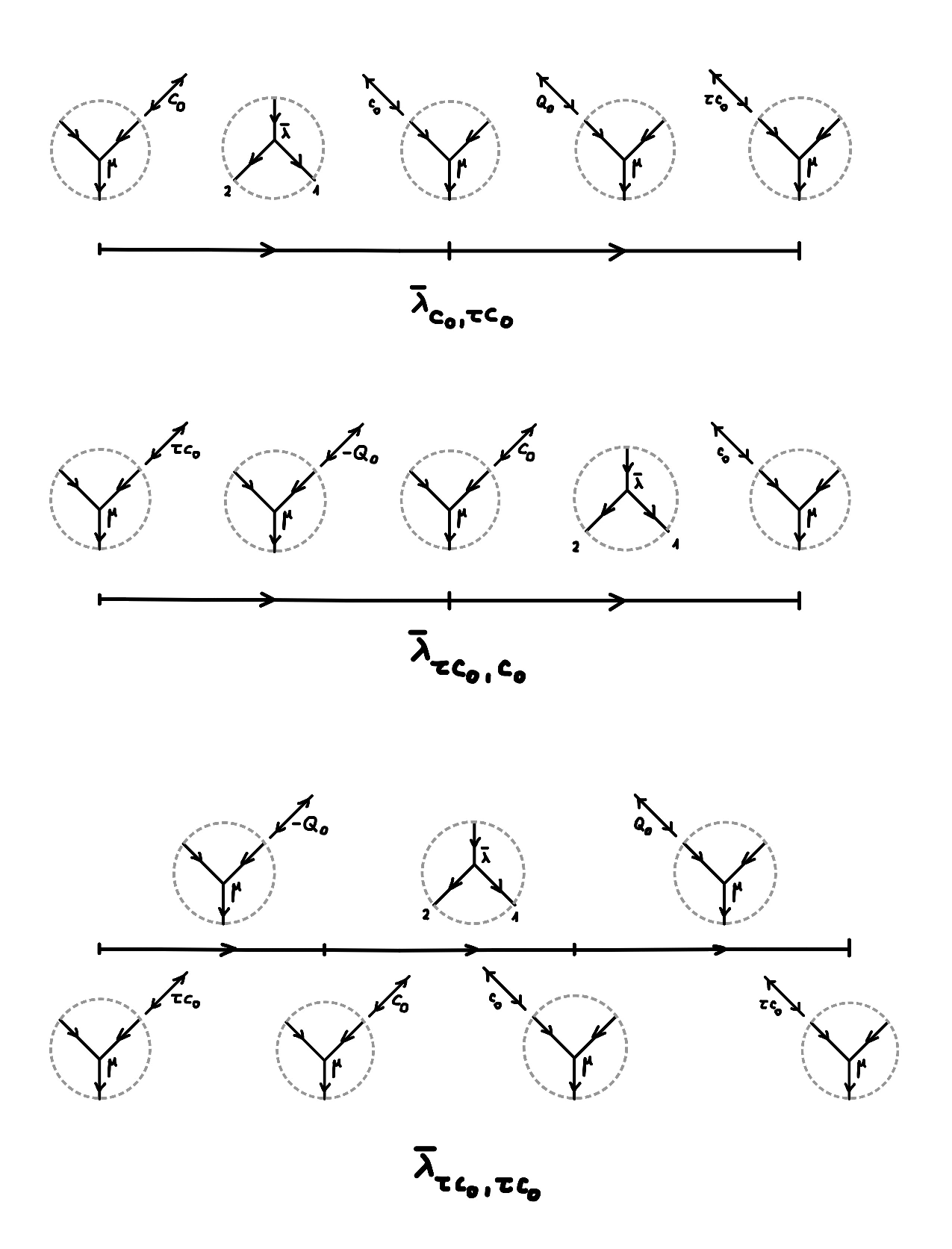}
\caption{The operations $\ol{\lambda}_{\boldsymbol{a},\boldsymbol{b}}$, $\boldsymbol{a},\boldsymbol{b}=c_0$ or $\tau c_0$}
\label{fig:lambdabar-trees-new-ab}
\end{center}
\end{figure}


The maps $m_L,m_R,\tau_L,\tau_R,\sigma,\beta$ defining the $A_2$-triple are given by explicit formulas in terms of $\mu$, $\ol{\lambda}$, $\ol{B}$ as follows. 

{\bf 1. The degree $0$ chain maps $m_L$ and $m_R$} are defined in terms of $\mu$ by Figure~\ref{fig:mLmR-circle}, in which we depict the inputs and outputs on a circle as above.
See also Figure~\ref{fig:mLmR}. 
According to~\S\ref{sec:trees}, the formulas which express the content of Figure~\ref{fig:mLmR-circle} read  
$$
  \ev (m_L\otimes 1)=\ev(1\otimes\mu)\tau_{23}\tau_{12},\qquad
  \ev(m_R\otimes 1)=\ev(1\otimes\mu),
$$
which in terms of elements means 
$$
\langle b,m_L(a,f)\rangle =\langle \mu(b,a),f\rangle,\qquad
\langle m_R(f,a),b\rangle =\langle f,\mu(a,b)\rangle.
$$
  
\begin{figure}
\begin{center}
\includegraphics[width=.7\textwidth]{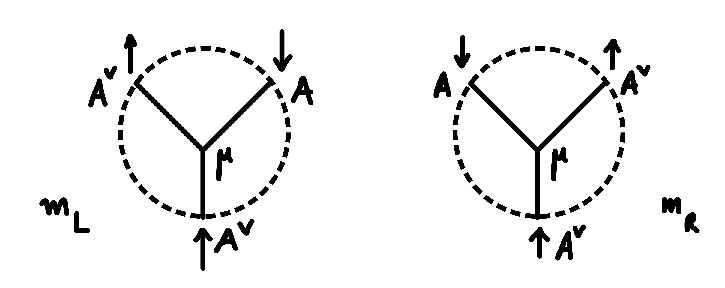}
\caption{The maps $m_L$ and $m_R$}
\label{fig:mLmR-circle}
\end{center}
\end{figure}

{\bf 2. The degree $1$ map $\sigma:\cA^\vee\otimes \cA^\vee\to \cA^\vee$} is determined by $\ol{\lambda}_{\tau c_0,\tau c_0}$ and is defined by Figure~\ref{fig:sigma-trees-new-Q}, to be compared with Figure~\ref{fig:lambdabar-trees-new-ab}. The relation $[\p,\sigma]=m_R(1\otimes c)-m_L(c\otimes 1)$ is readily seen on the figure. 
\begin{figure}
\begin{center}
\includegraphics[width=\textwidth]{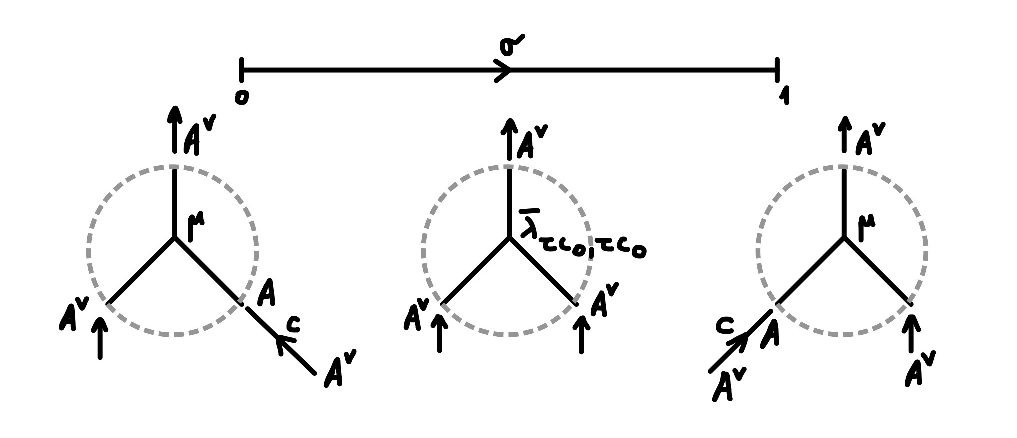}
\caption{The map $\sigma$}
\label{fig:sigma-trees-new-Q}
\end{center}
\end{figure}
The formula which expresses the content of Figure~\ref{fig:sigma-trees-new-Q} reads  
$$
  \ev(\sigma\otimes 1)= (\ev\otimes\ev)(1\otimes \tau\otimes 1)(1\otimes 1\otimes\ol{\lambda}_{\tau c_0,\tau c_0})
$$ 
for maps $\cA^\vee\otimes \cA^\vee\otimes\cA\to R$, or in terms of elements
\begin{align*}
  \langle \sigma(f,g),a\rangle
  &= (-1)^{|f|+|g|}\langle f\otimes g,\ol{\lambda}_{\tau c_0,\tau c_0}(a)\rangle \cr
  & = (-1)^{(|f|+1)(|g|+1)}\langle g\otimes f,\lambda_{\tau c_0,\tau c_0}(a)\rangle. 
\end{align*}
Equivalently, $\sigma$ equals the composition $\cA^\vee\otimes\cA^\vee\to (\cA\otimes \cA)^\vee\stackrel{\ol{\lambda}_{\tau c_0,\tau c_0}^{\vee}}\longrightarrow \cA^\vee$, see~\S\ref{sec:consequences}.3 for the first canonical map. Thus ``$\sigma$ is the dual of $\ol{\lambda}_{\tau c_0,\tau c_0}$". 

{\bf 3. The degree $1$ map $\tau_L:\cA^\vee\otimes\cA\to \cA$} is defined by Figure~\ref{fig:tauL-trees-new-Q}. The relation $[\p, \tau_L]=\mu(c\otimes 1) - cm_R$ is again readily visible on the figure. 
\begin{figure}
\begin{center}
\includegraphics[width=\textwidth]{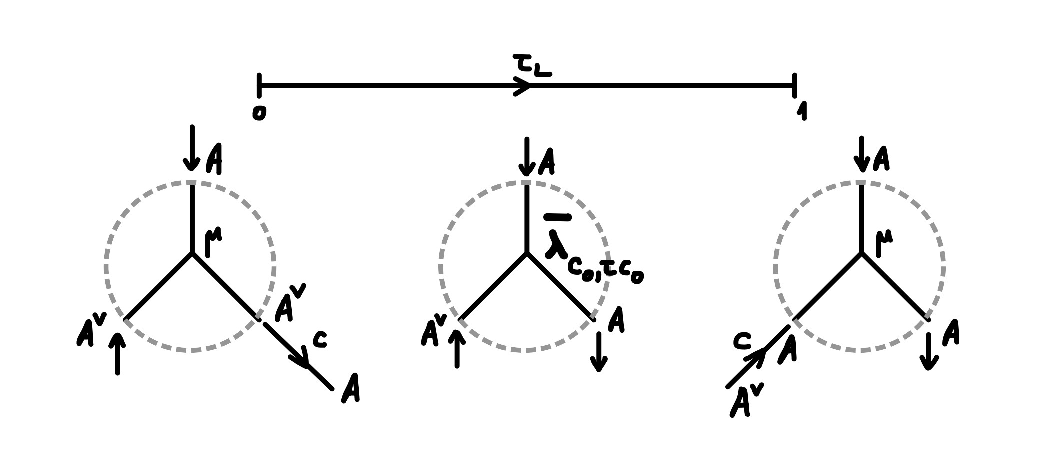}
\caption{The map $\tau_L$}
\label{fig:tauL-trees-new-Q}
\end{center}
\end{figure}
The explicit formula is
$$
  \tau_L=(\ev\otimes 1)(1\otimes\tau)(1\otimes \ol{\lambda}_{c_0,\tau c_0}),
$$
which means in terms of elements 
$$
  \tau_L(f,a)
  = (-1)^{|f|}\langle f\otimes 1,\tau \ol{\lambda}_{c_0,\tau c_0}(a)\rangle 
 = (-1)^{|f|+1}\langle f\otimes 1,\lambda_{\tau c_0,c_0}(a)\rangle. 
$$


{\bf 4. The degree $1$ map $\tau_R:\cA\otimes\cA^\vee\to\cA$} is defined by Figure~\ref{fig:tauR-trees-new-Q} (note the minus sign!) and satisfies the relation $[\p,\tau_R]=\mu(1\otimes c)-cm_L$. 
\begin{figure}
\begin{center}
\includegraphics[width=\textwidth]{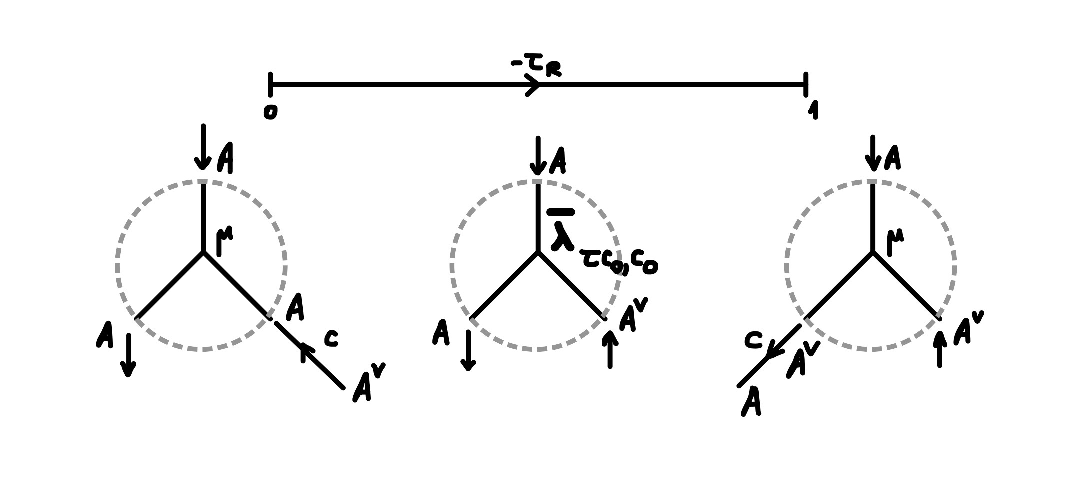}
\caption{The map $-\tau_R$}
\label{fig:tauR-trees-new-Q}
\end{center}
\end{figure}
The explicit formula is
$$
  \tau_R=-(1\otimes\ev)(\tau\otimes 1)(\ol{\lambda}_{\tau c_0,c_0}\otimes 1),
$$
which means in terms of elements 
$$
  \tau_R(a,f)
  = -\langle \tau\ol{\lambda}_{\tau c_0,c_0}(a),1\otimes f\rangle
    = \langle \lambda_{c_0,\tau c_0}(a),1\otimes f\rangle. 
$$




{\bf 5. The degree $2$ map $\beta:\cA^\vee\otimes\cA^\vee\to\cA$} is defined by Figure~\ref{fig:beta-trees-new-Q} (note the minus sign!) and satisfies the relation $[\p,\beta]= \tau_R(c\otimes 1) - c\sigma-\tau_L(1\otimes c)$. 
\begin{figure}
\begin{center}
\includegraphics[width=\textwidth]{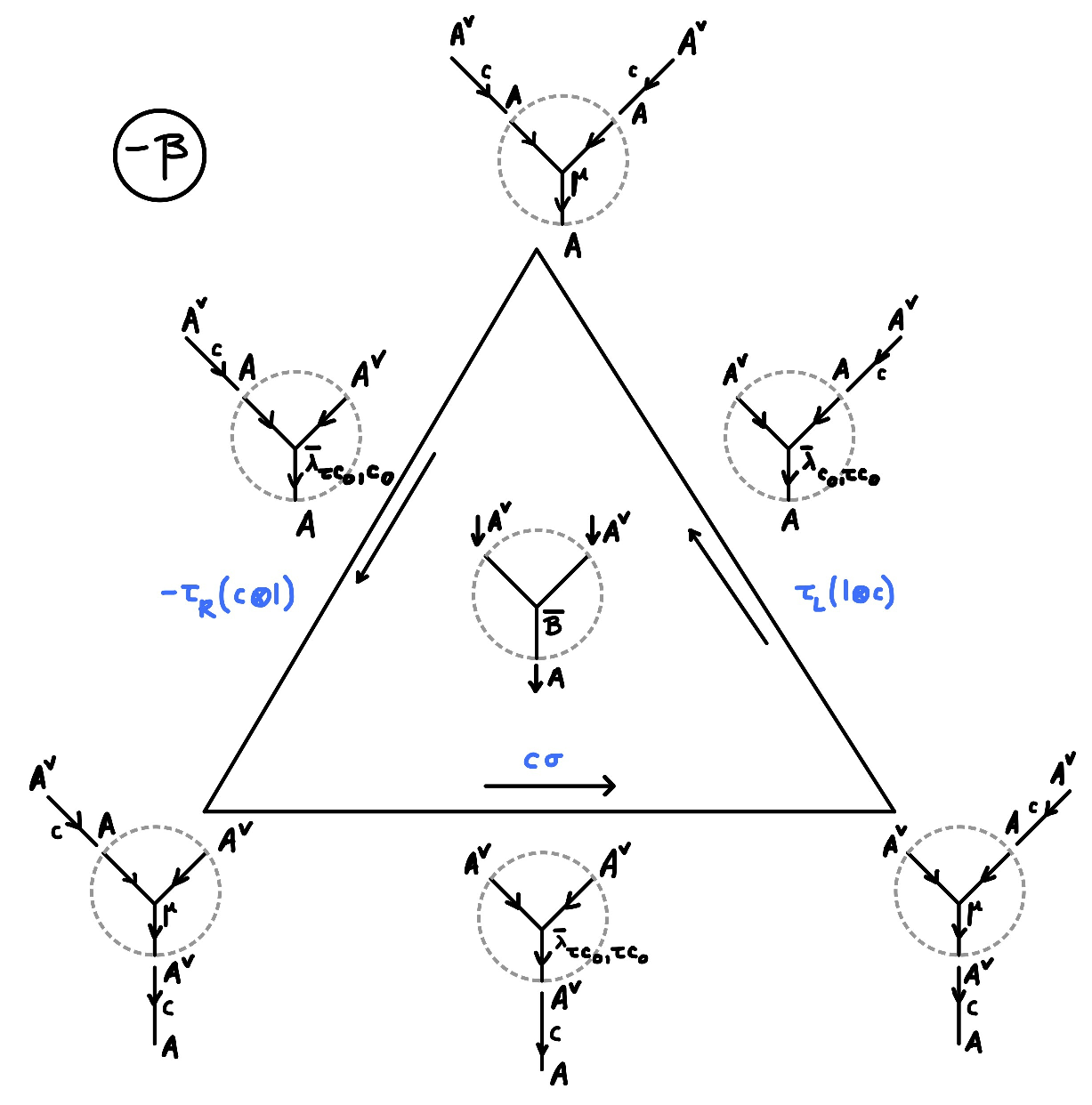}
\caption{The map $-\beta$}
\label{fig:beta-trees-new-Q}
\end{center}
\end{figure}
The explicit formula is
$$
\beta=-(1\otimes \ev\circ \tau \otimes\ev\circ\tau)(1\otimes1\otimes \tau\otimes 1)(\ol{B}\otimes 1\otimes 1),
$$
which in terms of elements means
$$
  \beta(f,g)=-\langle 1\otimes f\otimes g,\ol{B}\rangle. 
$$

%

This concludes the proof of Proposition~\ref{prop:TQFT+}.
\end{proof}

\begin{remark} \label{rmk:TQFT++}
We chose the terminology {\em $A_2^+$-structure} because such a structure comprises a product $\mu$, and additional operations needed to induce a product on the mapping cone. Note that the operations are parametrized by simplices of dimension $0$ (the product $\mu$ with $1$ output), of dimension $1$ (the secondary coproduct $\lambda$ with $2$ outputs), and of dimension $2$ (the cubic vector $B$ with $3$ outputs). The construction of the full $A_\infty$ structure on the cone makes use of the full enrichment by simplices of arbitrary dimension. More generally, one can consider (noncompact) TQFT-type structures with operations parametrized by topological types of $2$-dimensional surfaces with boundary and enriched in simplices: whenever a given surface has $k$ outputs, we attach to it a $(k-1)$-dimensional simplex. This perspective is related to the assocoipahedra of Poirier and Tradler~\cite{Poirier-Tradler},  and is currently studied by Mazuir~\cite{Mazuir-in-progress}. The paper~\cite{Ekholm-Oancea} is also relevant for this line of thought. 
\end{remark}

\subsection{Morphisms of $A_2^+$-algebras and correspondences}\label{sec:correspondencesA2+} 

Consider an $A_2^+$-algebra 
$$
\cA=(\cA,c_0,Q_0,\mu,\lambda,B).
$$ 
Recall that $c_0$ gives rise to the map  
$c=(\ev\otimes 1)(1\otimes c_0):\cA^\vee\to\cA$. 

\begin{definition}
We call $\cA$ {\em special} if
$$
  B=0   
$$
and
$$
(\lambda_{c_0,\tau c_0} \otimes 1) \tau c_0 = (\lambda_{\tau c_0,c_0}\otimes 1)\tau c_0=(\lambda_{\tau c_0,\tau c_0}\otimes 1) c_0 =0.
$$
\end{definition}

For a special $A_2^+$-algebra we will drop $B=0$ from the notation.
Consider now two special $A_2^+$-algebras $(\cA,c_0,Q_0,\mu,\lambda)$ and $(\cA',c_0',Q'_0,\mu',\lambda')$. 

\begin{definition}\label{def:A2+mor}
A \emph{special morphism of special $A_2^+$-algebras} 
$$
(\Psi,\Gamma,\Theta):\cA\to\cA'
$$
consists of the following maps (see Figures~\ref{fig:c0-Q0-Psi},~\ref{fig:Gamma} and~\ref{fig:Theta}):  

(i) a degree $0$ chain map $\Psi:\cA\to\cA'$ satisfying
$$
  c_0'=(\Psi\otimes\Psi)c_0 \qquad \mbox{and} \qquad Q'_0=(\Psi\otimes\Psi)Q_0;
$$
(ii) a degree $1$ bilinear map $\Gamma:\cA\otimes\cA\to\cA'$ such that
$$
  [\p,\Gamma] = \mu'(\Psi\otimes\Psi)-\Psi\mu;
$$
(iii) a degree $2$ bilinear map $\Theta:\cA\to\cA'\otimes\cA'$
  satisfying 
$$
\Theta c=0
$$ 
and 
$$
  [\p,\Theta] = \lambda'\Psi-(\Psi\otimes\Psi)\lambda 
  - (\Gamma\otimes\Psi)(1\otimes c_0) + (\Psi\otimes\Gamma)(\tau c_0\otimes 1).
$$
\end{definition}

\begin{figure}
\begin{center}
\includegraphics[width=.8\textwidth]{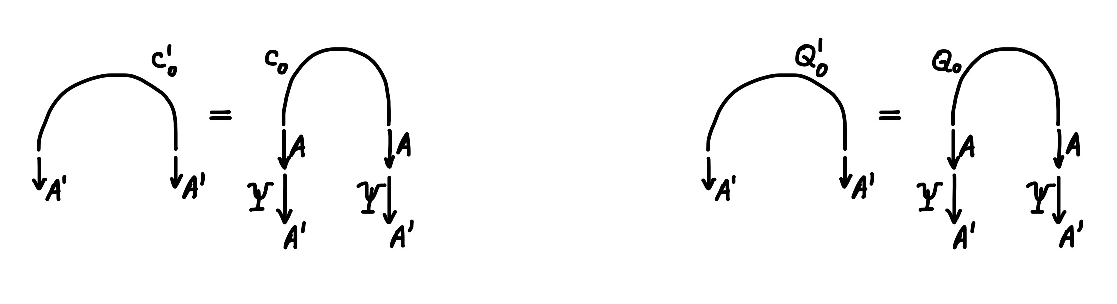}
\caption{$\Psi$ intertwines $c_0$ and $c_0'$, and also $Q_0$ and $Q'_0$}
\label{fig:c0-Q0-Psi}
\end{center}
\end{figure}

\begin{figure}
\begin{center}
\includegraphics[width=.7\textwidth]{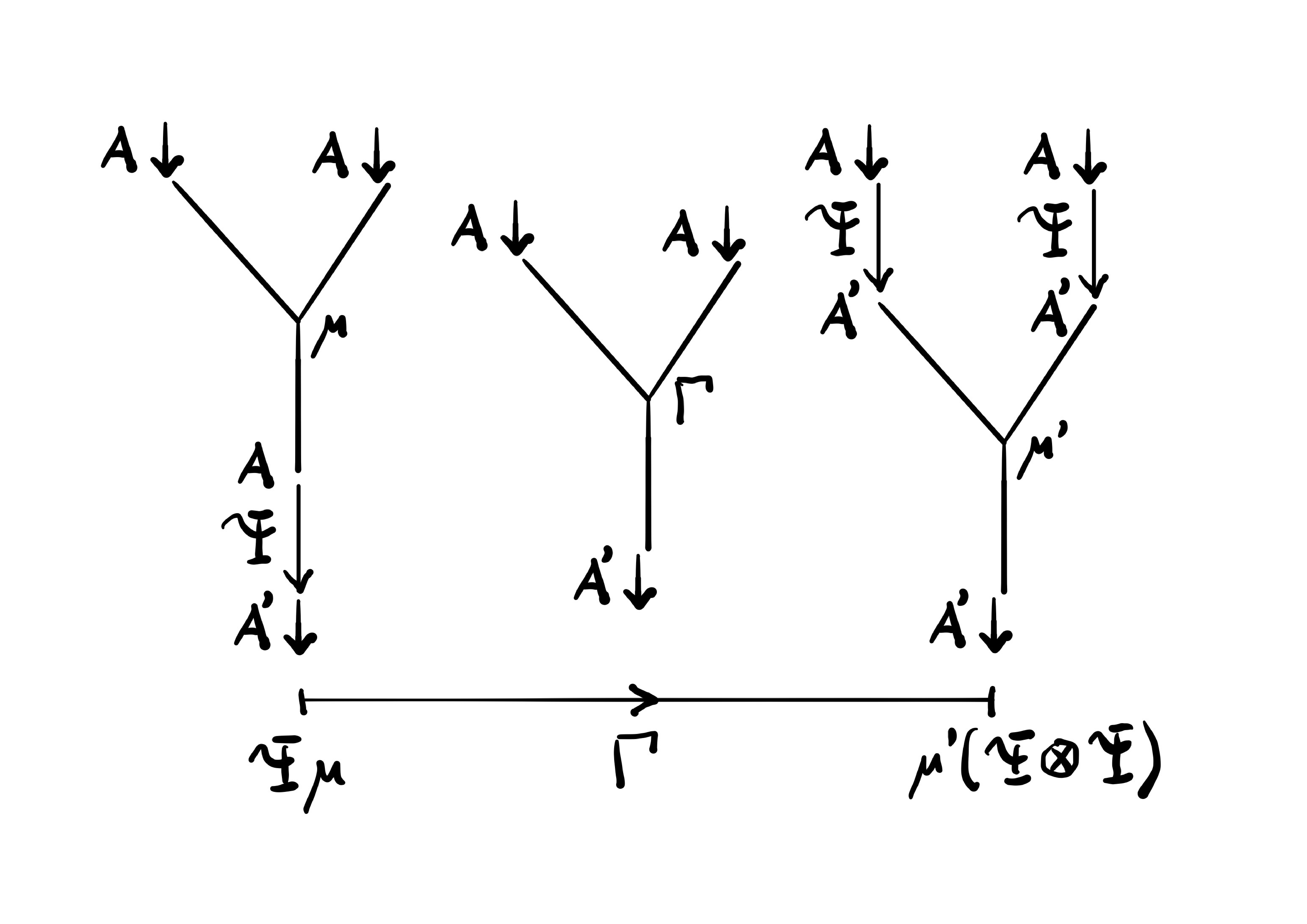}
\caption{The map $\Gamma$}
\label{fig:Gamma}
\end{center}
\end{figure}

\begin{figure}
\begin{center}
\includegraphics[width=\textwidth]{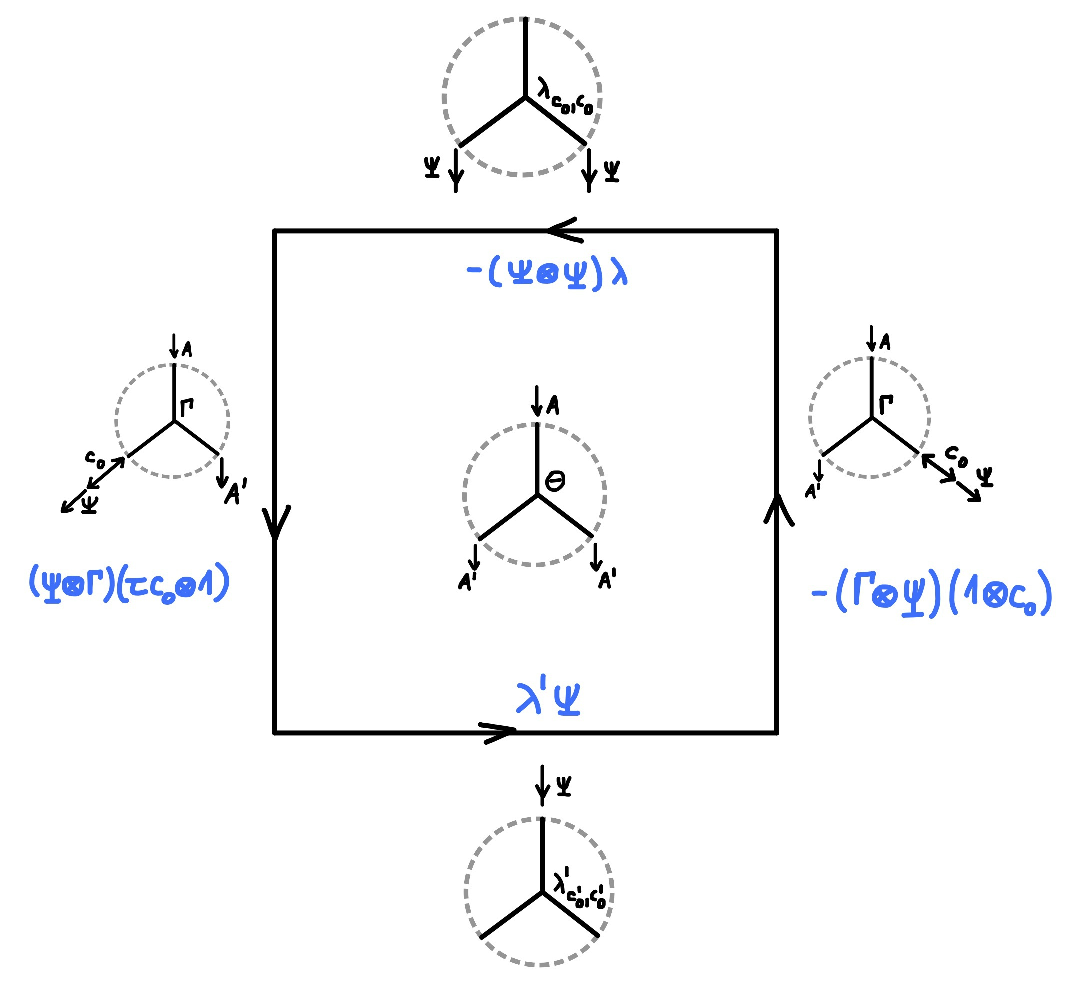}
\caption{The map $\Theta$}
\label{fig:Theta}
\end{center}
\end{figure}


Note that the first part of property (i) of a special morphism implies
\begin{equation}\label{eq:c-Psi}
   \Psi c\Psi^\vee = c'. 
\end{equation}

\begin{remark}
For a general morphism of $A_2^+$-algebras, the two relations forming condition (i) only need to hold up to chain homotopy and the other conditions have to be adjusted accordingly. Also, the condition $\Theta c=0$ would not be needed at all if the setup were upgraded to include arity $3$ operations. We restrict our discussion to special $A_2$-algebras and special morphisms because these suffice for our applications in Rabinowitz Floer homology. 
\end{remark}

Recall that by Proposition~\ref{prop:TQFT+} the $A_2^+$-algebra
$\cA$ canonically gives rise to an $A_2$-triple $(\cA^\vee,c,\cA)$,
and similarly for $\cA'$. Note that a chain map $\Psi:\cA\to\cA'$
induces a chain map $\Psi^\vee:\cA'^\vee\to\cA^\vee$. 

\begin{proposition}\label{prop:A2+toA2-mor}
A special morphism of special $A_2^+$-algebras $(\Psi,\Gamma,\Theta):\cA\to\cA'$ 
canonically induces an $A_2$-structure on the triple $(\cA'^\vee,c\Psi^\vee,\cA)$, 
as well as special morphisms of $A_2$-triples
\begin{equation}\label{eq:A2+toA2-mor}
\xymatrix{
  & (\cA'^\vee,c\Psi^\vee,\cA) \ar[dl]_{(\Psi^\vee,1)} \ar[dr]^{(1,\Psi)} \\
  (\cA^\vee,c,\cA) & & (\cA'^\vee,c',\cA').  
}
\end{equation}
\end{proposition}

Combining this proposition with Proposition~\ref{prop:A2mor-to-cone} yields the following immediate corollary.

\begin{corollary}\label{cor:A2+toA2-mor}
Let $(\Psi,\Gamma,\Theta):\cA\to\cA'$ be a special morphism of special $A_2^+$-algebras. Then there is a correspondence of chain maps
$$
\xymatrix{
  & Cone(c\Psi^\vee) \ar[dl]_{Cone(\Psi^\vee,1)} \ar[dr]^{Cone(1,\Psi)} \\
  Cone(c) & & Cone(c')\,.  
}
$$
Their induced maps on homology are ring maps and fit into the commuting diagram with exact sequences
$$
\xymatrix{
   \cdots H^{-*}(\cA) \ar[r]^-{c_*} & H_*(\cA) \ar[r] & H_*(Cone(c)) \ar[r] & H^{1-*}(\cA) \cdots \\
   \cdots H^{-*}(\cA') \ar[r]^-{c_*\Psi^*} \ar[u]_{\Psi^*} \ar@{=}[d] & H_*(\cA) \ar[r] \ar@{=}[u] \ar[d]^{\Psi_*} & H_*(Cone(c\Psi^\vee)) \ar[r] \ar[u] \ar[d] & H^{1-*}(\cA') \cdots \ar[u]_{\Psi^*} \ar@{=}[d] \\
   \cdots H^{-*}(\cA') \ar[r]^-{c'_*} & H_*(\cA') \ar[r] & H_*(Cone(c')) \ar[r] & H^{1-*}(\cA') \cdots 
}
$$
If $\Psi$ induces an isomorphism on homology, then so do the maps between the cones and the commuting diagram simplifies to
$$
\xymatrix{
   \cdots H^{-*}(\cA) \ar[r]^-{c_*} & H_*(\cA) \ar[r] \ar[d]^{\Psi_*}_\cong & H_*(Cone(c)) \ar[r] \ar[d]_\cong & H^{1-*}(\cA) \cdots \\
   \cdots H^{-*}(\cA') \ar[r]^-{c'_*} \ar[u]_{\Psi^*}^\cong & H_*(\cA') \ar[r] & H_*(Cone(c')) \ar[r] & H^{1-*}(\cA') \cdots \ar[u]_{\Psi^*}^\cong
}
$$
where the second and third vertical maps are ring isomorphisms. \qed
\end{corollary}

\begin{proof}[Proof of Proposition~\ref{prop:A2+toA2-mor}]
The $A_2$-structure $(\wt\mu,\wt m_L,\wt
m_R,\wt\tau_L,\wt\tau_R,\wt\sigma,\wt\beta)$ on the 
triple $(\cA'^\vee,c\Psi^\vee,\cA)$ is defined as follows. 
We set
$\wt\mu:=\mu$. 
The maps $\wt m_L$ and $\wt m_R$ are defined in
Figure~\ref{fig:mLmR-tilde}. The maps $\wt\tau_L$, $\wt\tau_R$ and $\wt\sigma$ are defined in Figures~\ref{fig:sigma-tilde-new-Q},~\ref{fig:tauL-tilde-new-Q}
and~\ref{fig:tauR-tilde-new-Q}, in which we denote 
$$
\ol{\lambda}_{\boldsymbol{a},\boldsymbol{b}}=-\tau\lambda_{\boldsymbol{b},\boldsymbol{a}},\qquad \boldsymbol{a},\boldsymbol{b}=c_0 \,\, \mbox{or}\,\, \tau c_0,
$$
$$
\ol{\lambda}'_{\boldsymbol{a},\boldsymbol{b}}=-\tau\lambda'_{\boldsymbol{b},\boldsymbol{a}},\qquad \boldsymbol{a},\boldsymbol{b}=c'_0 \,\, \mbox{or}\,\, \tau c'_0.
$$
Note in particular that $\wt m_L=m'_L(\Psi\otimes 1)$, $\wt m_R=m'_R(1\otimes \Psi)$ and $\tilde \sigma=\sigma'$. 

The map $\wt \beta$ is defined in Figure~\ref{fig:beta-tilde-new-Q}. It uses the auxiliary map 
$$
\ol{\Theta}_{\tau c_0,\tau c_0} = -\tau \Theta_{\tau c_0,\tau c_0},
$$
with $\Theta_{\tau c_0,\tau c_0}$ defined by Figure~\ref{fig:Thetabar-new-Q}.


\begin{figure}
\begin{center}
\includegraphics[width=.7\textwidth]{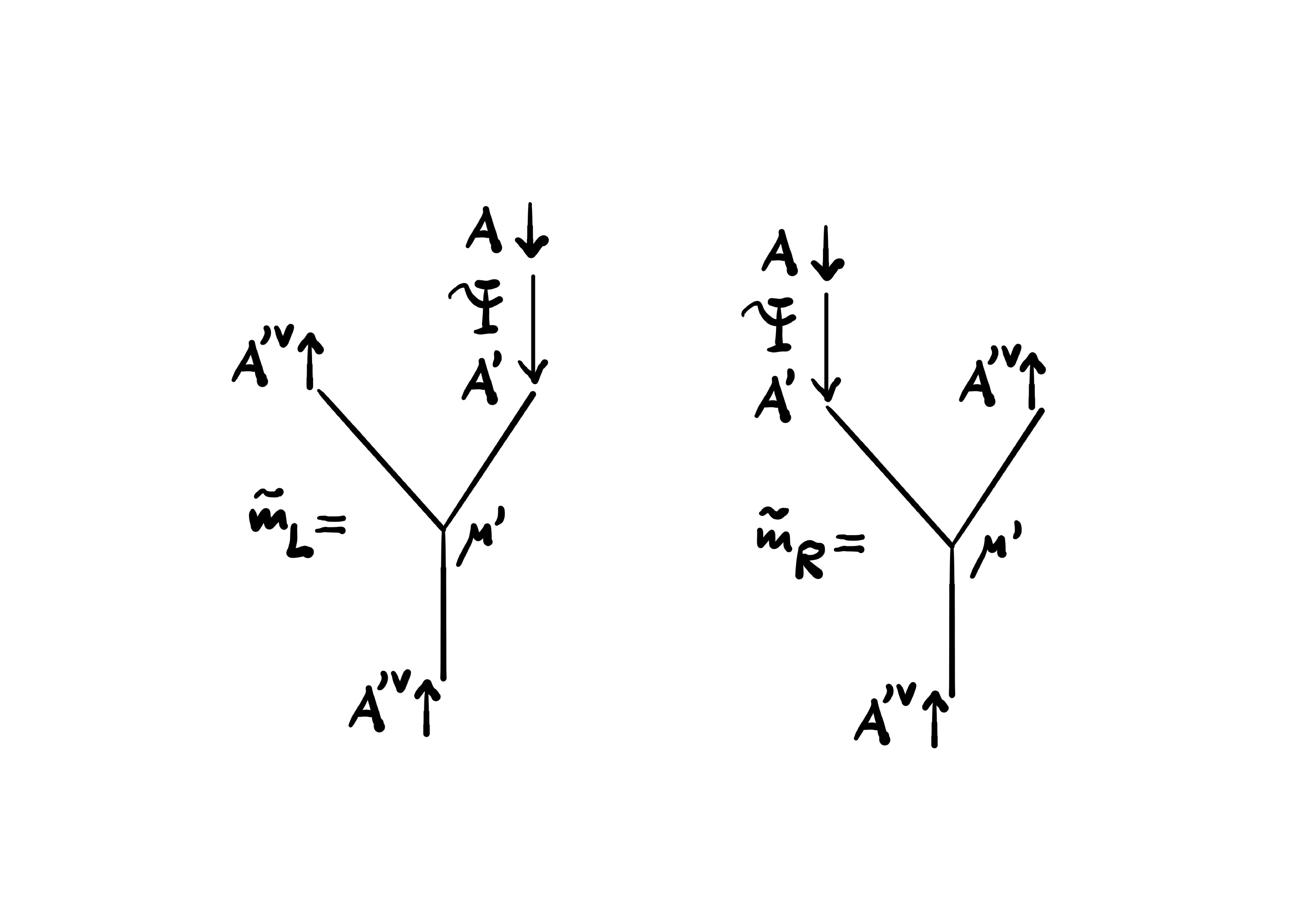}
\caption{Definition of $\wt m_L$ and $\wt m_R$}
\label{fig:mLmR-tilde}
\end{center}
\end{figure}

\begin{figure}
\begin{center}
\includegraphics[width=\textwidth]{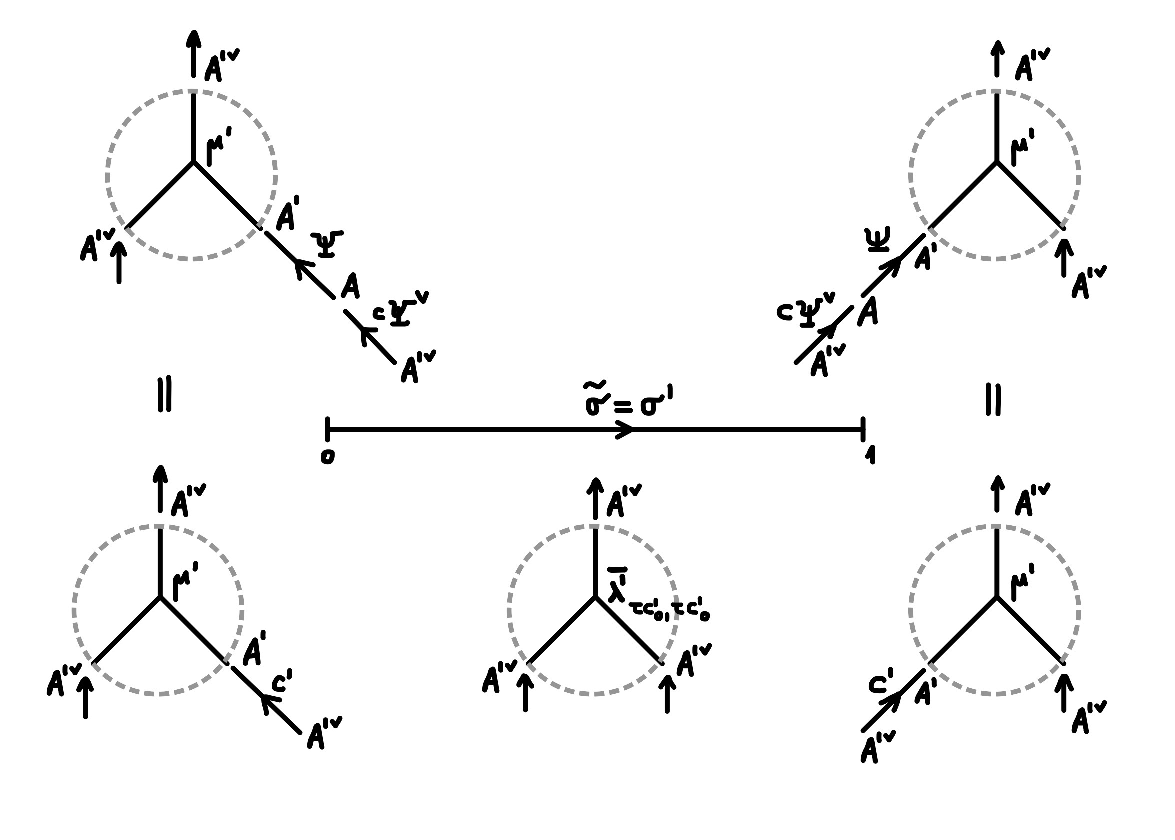}
\caption{Definition of $\wt\sigma$}
\label{fig:sigma-tilde-new-Q}
\end{center}
\end{figure}

\begin{figure}
\begin{center}
\includegraphics[width=\textwidth]{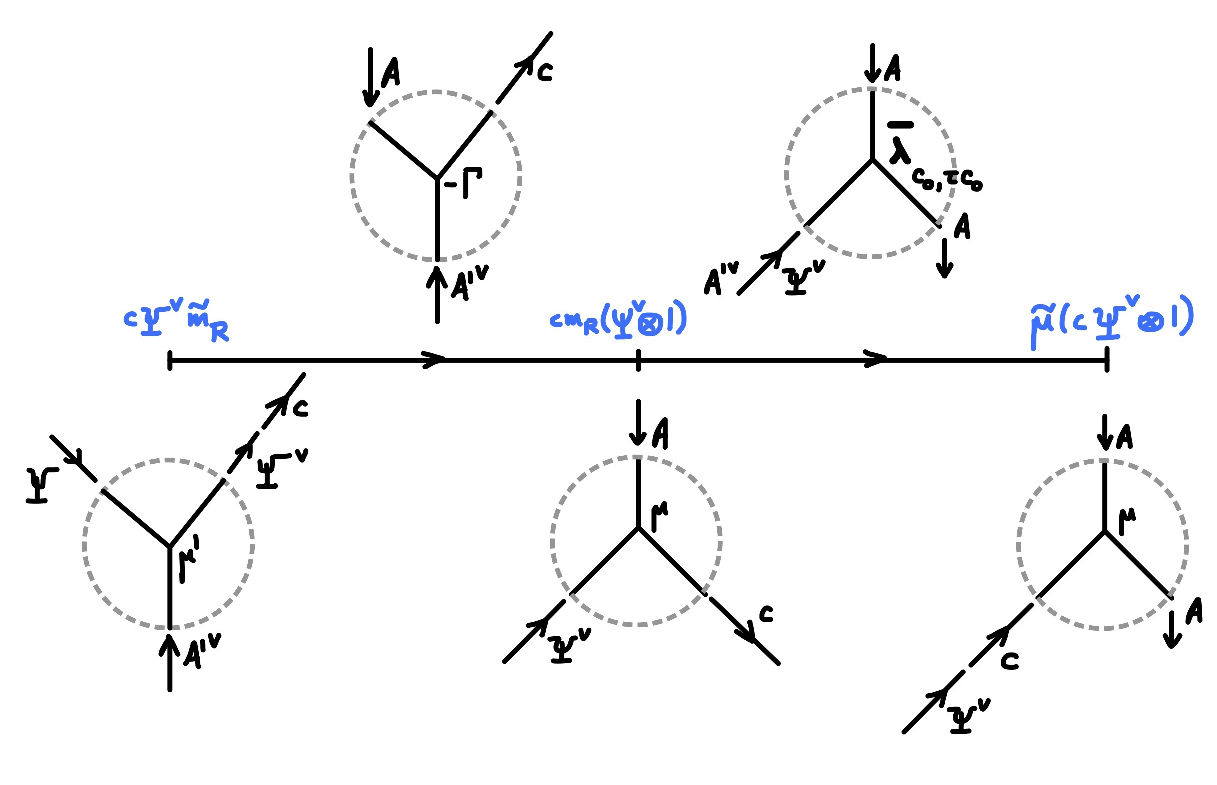}
\caption{Definition of $\wt\tau_L$}
\label{fig:tauL-tilde-new-Q}
\end{center}
\end{figure}

\begin{figure}
\begin{center}
\includegraphics[width=\textwidth]{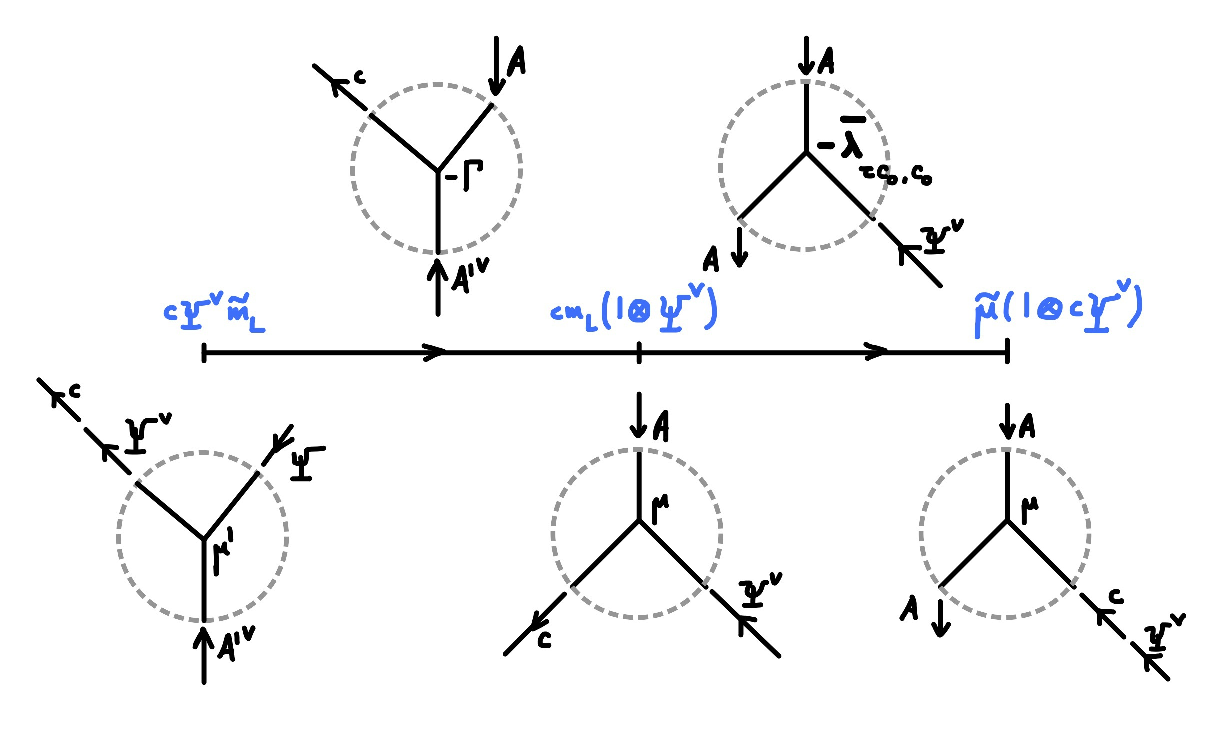}
\caption{Definition of $\wt\tau_R$}
\label{fig:tauR-tilde-new-Q}
\end{center}
\end{figure}

\begin{figure}
\begin{center}
\includegraphics[width=\textwidth]{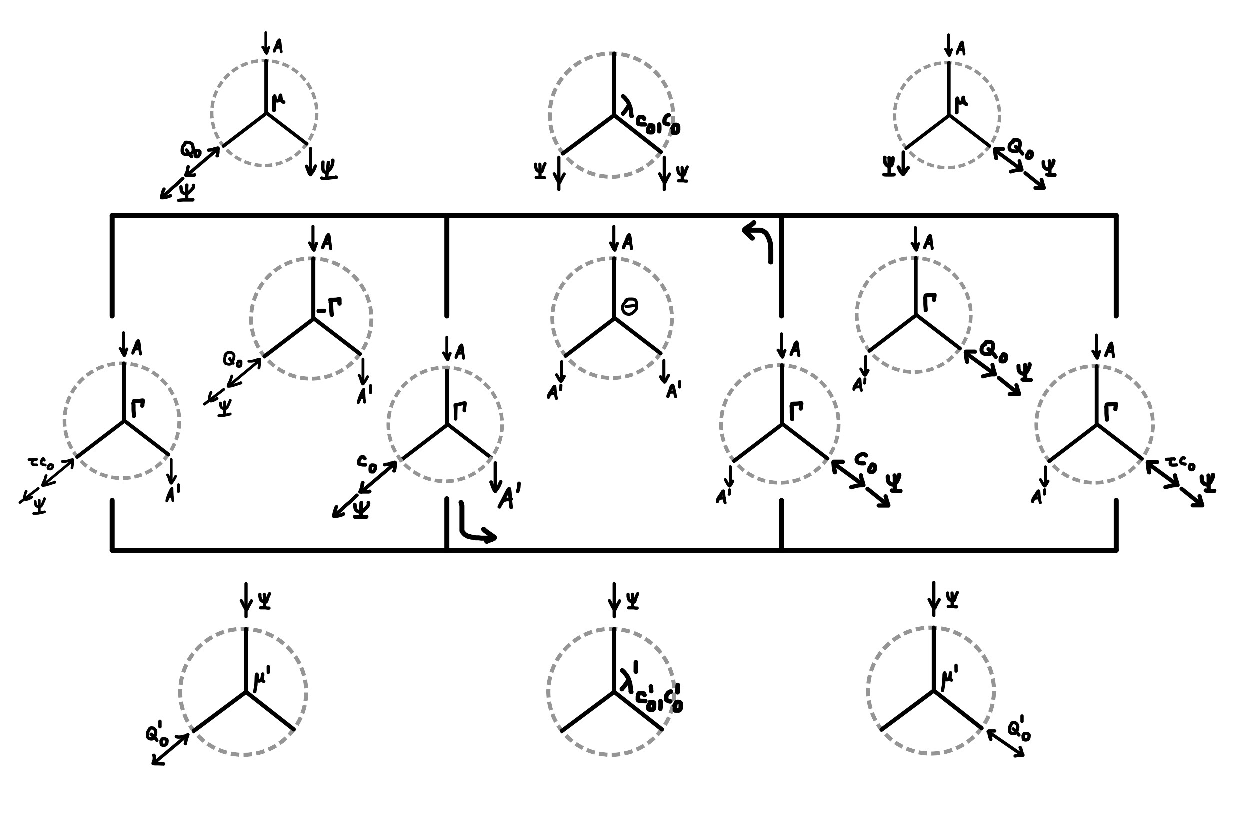}
\caption{Definition of $\Theta_{\tau c_0,\tau c_0}$}
\label{fig:Thetabar-new-Q}
\end{center}
\end{figure}

\begin{figure}
\begin{center}
\includegraphics[width=\textwidth]{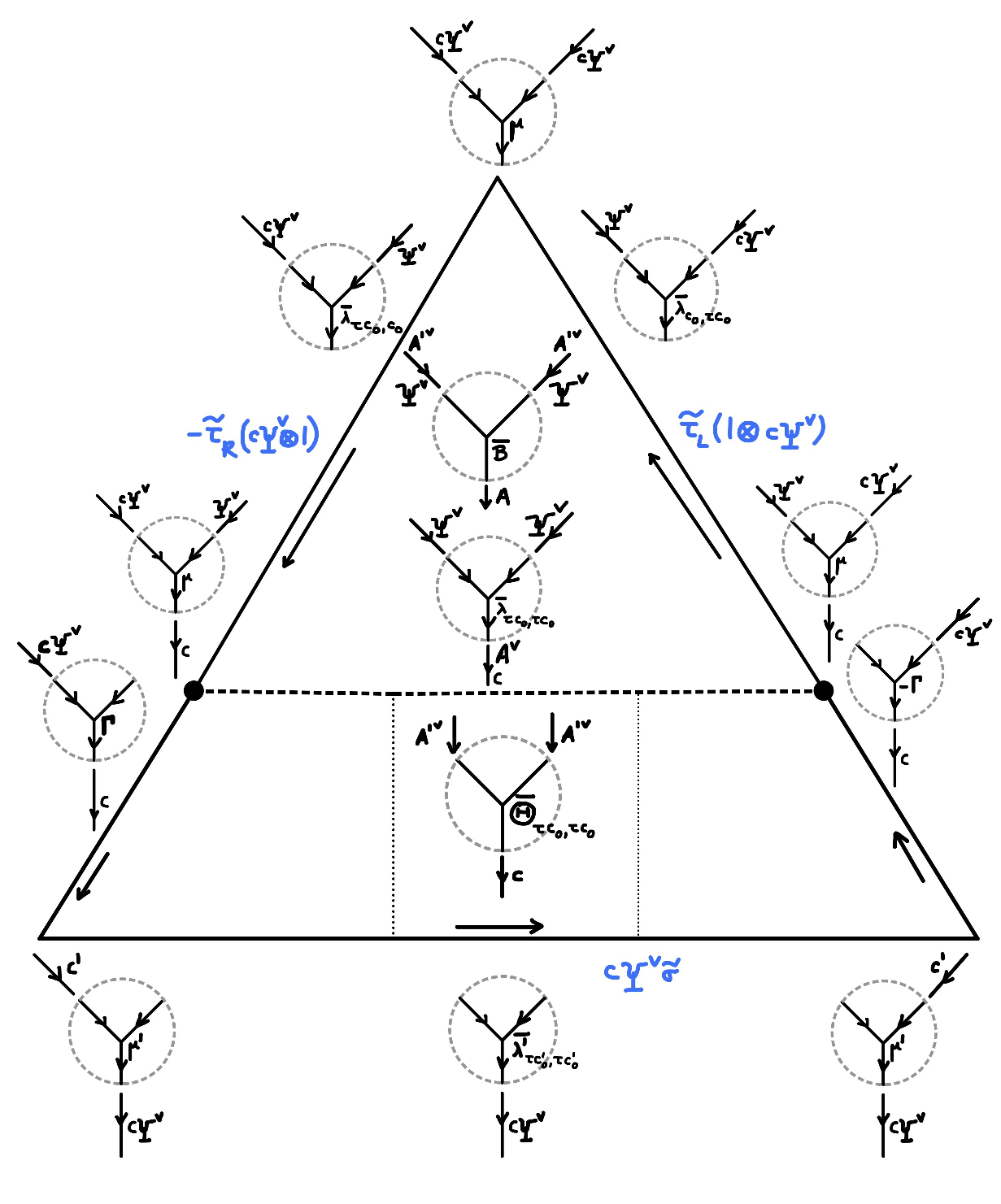}
\caption{Definition of $-\wt\beta$}
\label{fig:beta-tilde-new-Q}
\end{center}
\end{figure}

All the relations of an $A_2$-triple are clear from the figures. 
Note that, although we have $B=0$ and $\Theta c=0$, the map $\tilde \beta$ may be nonzero.
This finishes the construction of the $A_2$-structure on $(\cA'^\vee,c\Psi^\vee,\cA)$. 

We remark that Figure~\ref{fig:beta-tilde-new-Q} shows that the conditions $B=0$ and $\Theta c=0$ are not necessary in order to construct the $A_2$-triple structure on $(\cA'^\vee,c\Psi^\vee,\cA)$. However, we will use them below in defining the special morphisms of $A_2$-triples $(1,\Psi)$ and $(\Psi^\vee,1)$. 

%

Consider now the right hand map in~\eqref{eq:A2+toA2-mor} corresponding to
the commuting diagram
$$
\xymatrix{
  \cA'^\vee \ar[r]^{c\Psi^\vee} \ar[d]^{1} & \cA \ar[d]^\Psi \\
  \cA'^\vee \ar[r]^{c'} & \cA'.  
}
$$
The maps defining the $A_2$-triple $(\cA'^\vee,c\Psi^\vee,\cA)$ are
decorated with~$\widetilde{}$, while the maps defining the
$A_2$-triple $(\cA'^\vee,c',\cA')$ are decorated with~$'$.
We need to define maps $(\wh\mu,\wh m_L,\wh m_R,\wh\tau_L,\wh\tau_R,\wh\sigma,\wh\beta)$
making $(1,\Psi)$ a special morphism of $A_2$-triples. So we need to
verify conditions (1)--(5) in Definition~\ref{defi:pre-subalgebra-mor}
with $f=\Psi$, $g=1$, and $(\cA'^\vee,c\Psi^\vee,\cA)$ in place of $(\cM,c,\cA)$.

Condition (1) holds in view of $\Psi c\Psi^\vee=c'$ from~\eqref{eq:c-Psi}.
Since $[\p,\Gamma] = \mu'(\Psi\otimes\Psi) - \Psi\mu$, the first
equation in condition (2) holds with $\wh\mu:=\Gamma$. By definition of
$\wt m_L,\wt m_R,\wt\sigma$ we have 
\begin{gather*}
  m_L'(\Psi\otimes 1) = 1\circ\wt m_L,\qquad
  m_R'(1\otimes\Psi) = 1\circ\wt m_R,\cr
  \sigma'(1\otimes 1) = 1\circ\wt\sigma. 
\end{gather*}
Thus the remaining two equations in condition (2) and the first equation in condition (3) 
are satisfied with $\wh m_L=\wh m_R=\wh\sigma:=0$. 

To define the other maps we introduce $\Theta_{c_0,\tau c_0}$ and $\Theta_{\tau c_0,c_0}$ defined in Figures~\ref{fig:Thetabar-new-Q-c0tauc0} and~\ref{fig:Thetabar-new-Q-tauc0c0}, and we denote 
$$
\ol{\Theta}_{\boldsymbol{a},\boldsymbol{b}}=-\tau\Theta_{\boldsymbol{b},\boldsymbol{a}},\qquad \boldsymbol{a},\boldsymbol{b}=c_0 \,\, \mbox{or}\,\, \tau c_0.
$$

\begin{figure}
\begin{center}
\includegraphics[width=\textwidth]{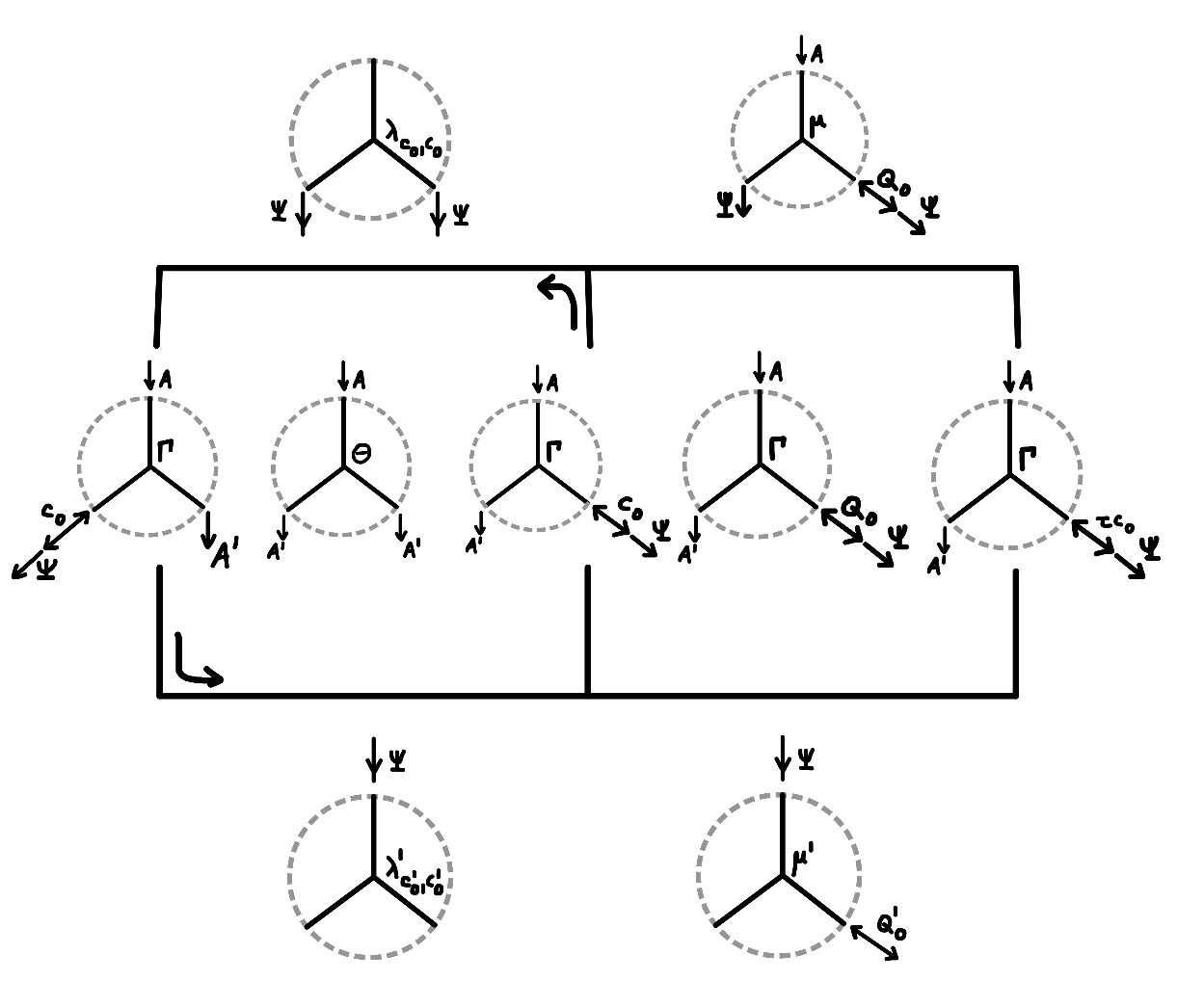}
\caption{Definition of $\Theta_{c_0,\tau c_0}$}
\label{fig:Thetabar-new-Q-c0tauc0}
\end{center}
\end{figure}

\begin{figure}
\begin{center}
\includegraphics[width=\textwidth]{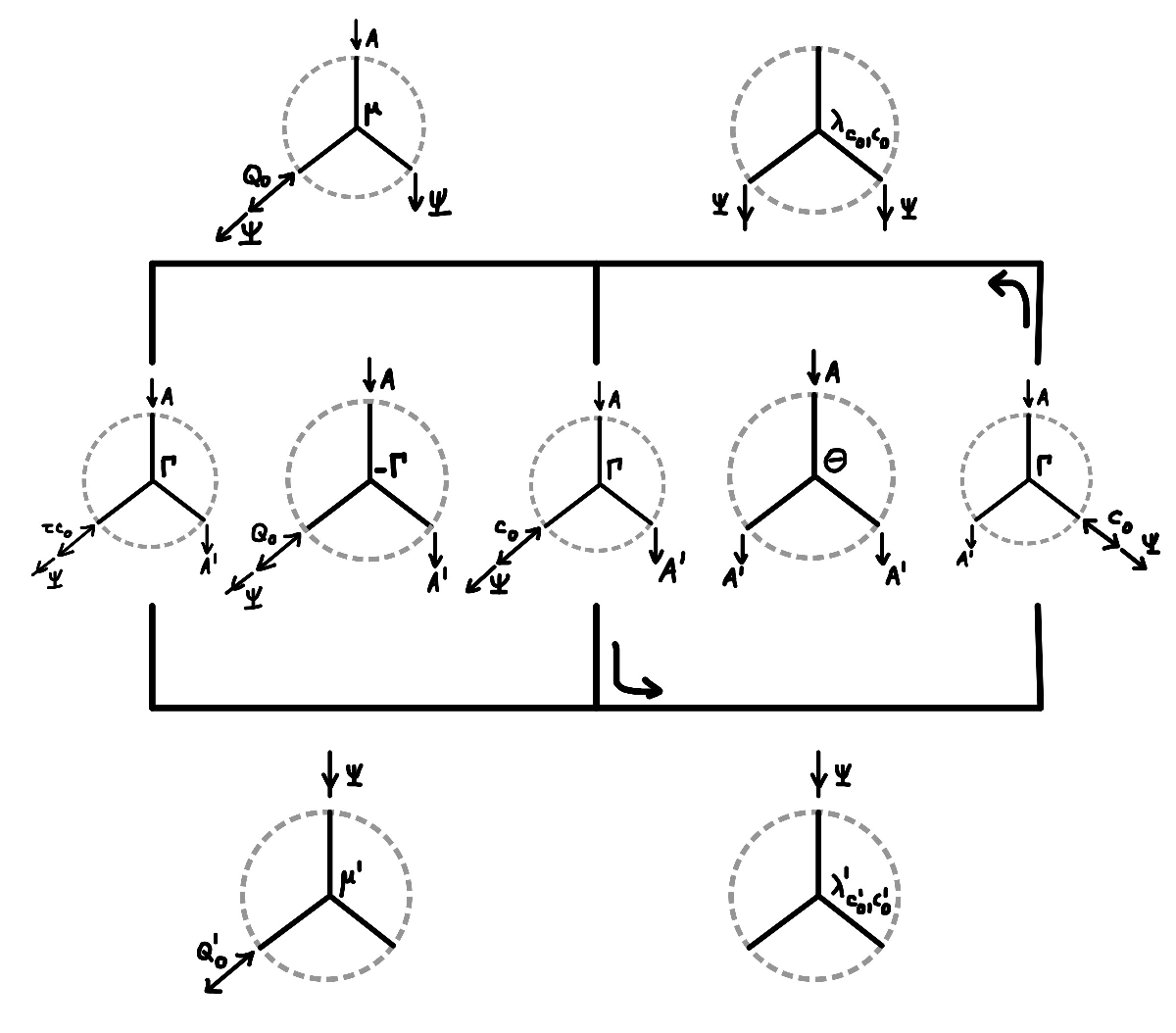}
\caption{Definition of $\Theta_{\tau c_0,c_0}$}
\label{fig:Thetabar-new-Q-tauc0c0}
\end{center}
\end{figure}

The definition of $\wh\tau_L$ is shown in
Figure~\ref{fig:tauL-hat-new-Q}. We read off that it satisfies the second
equation in condition (3), which with $f=\Psi$, $g=1$, $\wh\mu=\Gamma$
and $\wh m_L=0$ takes the form
$$
   [\p,\wh\tau_L] = \tau_L'(1\otimes\Psi) - \Psi\wt\tau_L -
   \Gamma(c\Psi^\vee\otimes 1).
$$
The definition of $\wh\tau_R$ such that it satisfies the last
equation in condition (3) is analogous. It is shown in Figure~\ref{fig:tauR-hat-new-Q}. 

\begin{figure}
\begin{center}
\includegraphics[width=\textwidth]{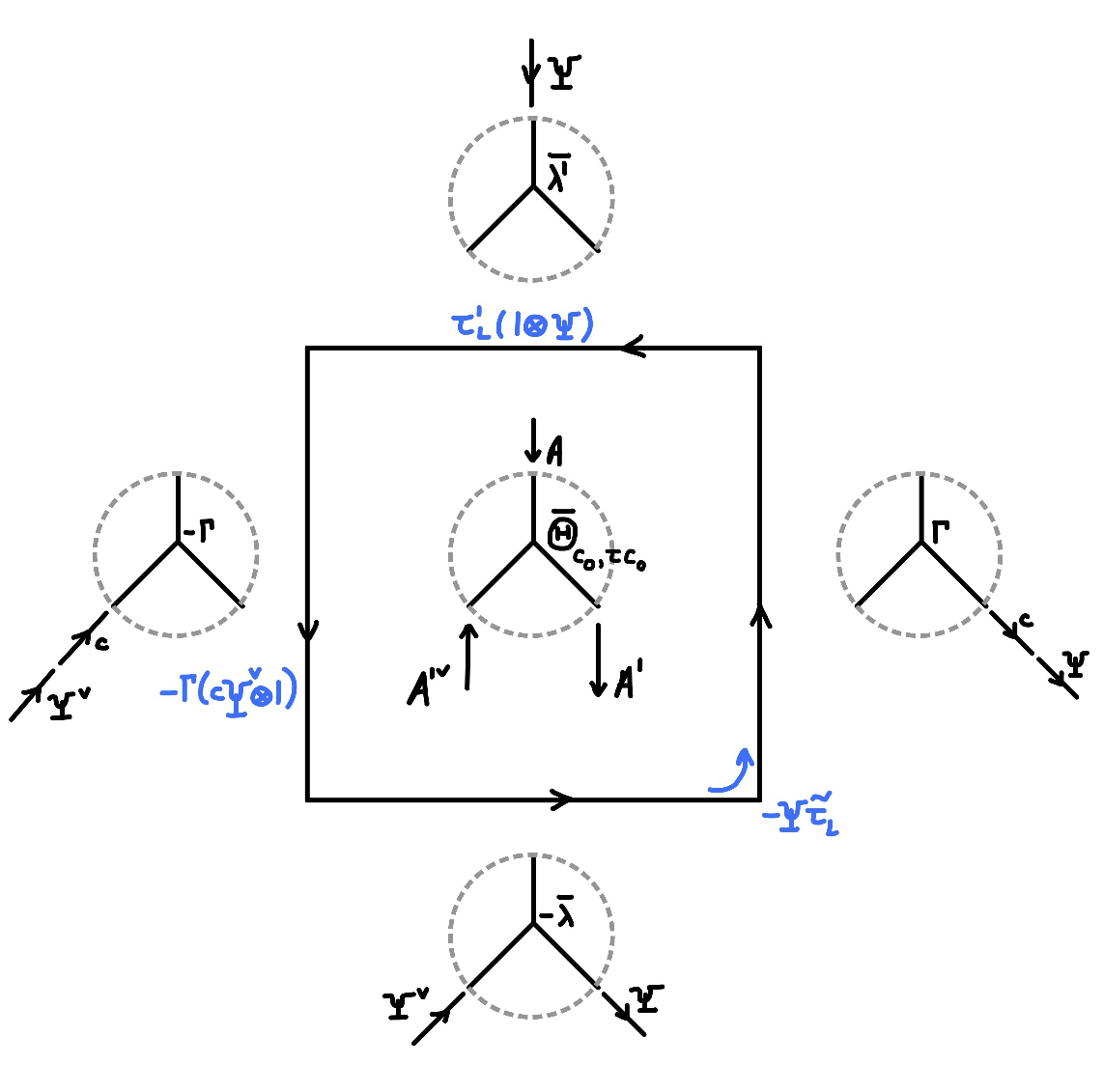}
\caption{Definition of $\wh\tau_L$}
\label{fig:tauL-hat-new-Q}
\end{center}
\end{figure}

\begin{figure}
\begin{center}
\includegraphics[width=\textwidth]{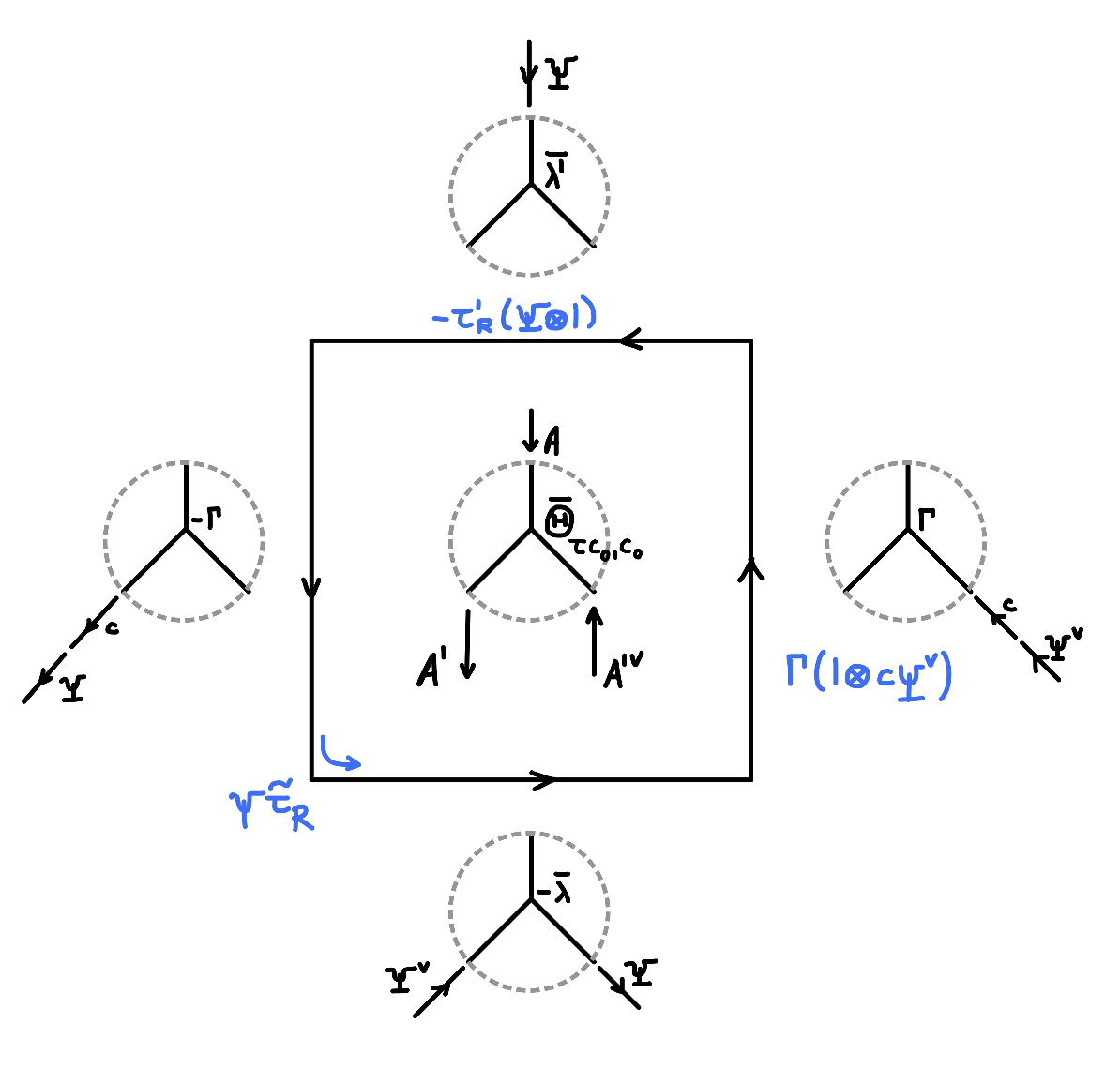}
\caption{Definition of $-\wh\tau_R$}
\label{fig:tauR-hat-new-Q}
\end{center}
\end{figure}

With $\wh\beta:=0$ and
$f=\Psi$, $g=1$ and $\beta'=\wh\sigma=0$, condition (4) becomes
$$
  -\Psi\wt\beta - \wh\tau_R(c\Psi^\vee\otimes 1) + \wh\tau_L(1\otimes c\Psi^\vee)=0.
$$
The condition $\Theta c=0$ in
Definition~\ref{def:A2+mor} implies that this expression reduces to a sum involving the side terms in the definition of $\Theta_{\tau c_0,\tau c_0}$ from Figure~\ref{fig:Thetabar-new-Q}, each appearing twice and with opposite signs. The above expression therefore vanishes. 
This finishes the construction of the right hand special $A_2$-morphism
in~\eqref{eq:A2+toA2-mor}. 


Consider now the left hand map in~\eqref{eq:A2+toA2-mor} corresponding to
the commuting diagram
$$
\xymatrix{
  \cA'^\vee \ar[r]^{c\Psi^\vee} \ar[d]^{\Psi^\vee} & \cA \ar[d]^{1} \\
  \cA^\vee \ar[r]^{c} & \cA.  
}
$$
The maps defining the $A_2$-triple $(\cA'^\vee,c\Psi^\vee,\cA)$ are
decorated with $\widetilde{}$, while the maps defining the
$A_2$-triple $(\cA^\vee,c,\cA)$ are not decorated with anything.
We need to define maps $(\wh\mu,\wh m_L,\wh m_R,\wh\tau_L,\wh\tau_R,\wh\sigma,\wh\beta)$
making $(\Psi^\vee,1)$ a special morphism of $A_2$-triples. So we need to
verify conditions (1)--(4) in Definition~\ref{defi:pre-subalgebra-mor}
with $f=1$, $g=\Psi^\vee$, $(\cA'^\vee,c\Psi^\vee,\cA)$ in place of
$(\cM,c,\cA)$, and $(\cA^\vee,c,\cA)$ in place of $(\cM',c',\cA')$.

Condition (1) holds trivially.
Since $\mu(1\otimes 1) = \mu = \wt\mu$, the first equation in
condition (2) holds with $\wh\mu:=0$.
The definition of $-\wh m_L$ is shown in Figure~\ref{fig:mL-hat}.
We read off that it satisfies the second equation in condition (2).
The definition of $\wh m_R$ such that it satisfies the third
equation in condition (2) is analogous.

\begin{figure}
\begin{center}
\includegraphics[width=.7\textwidth]{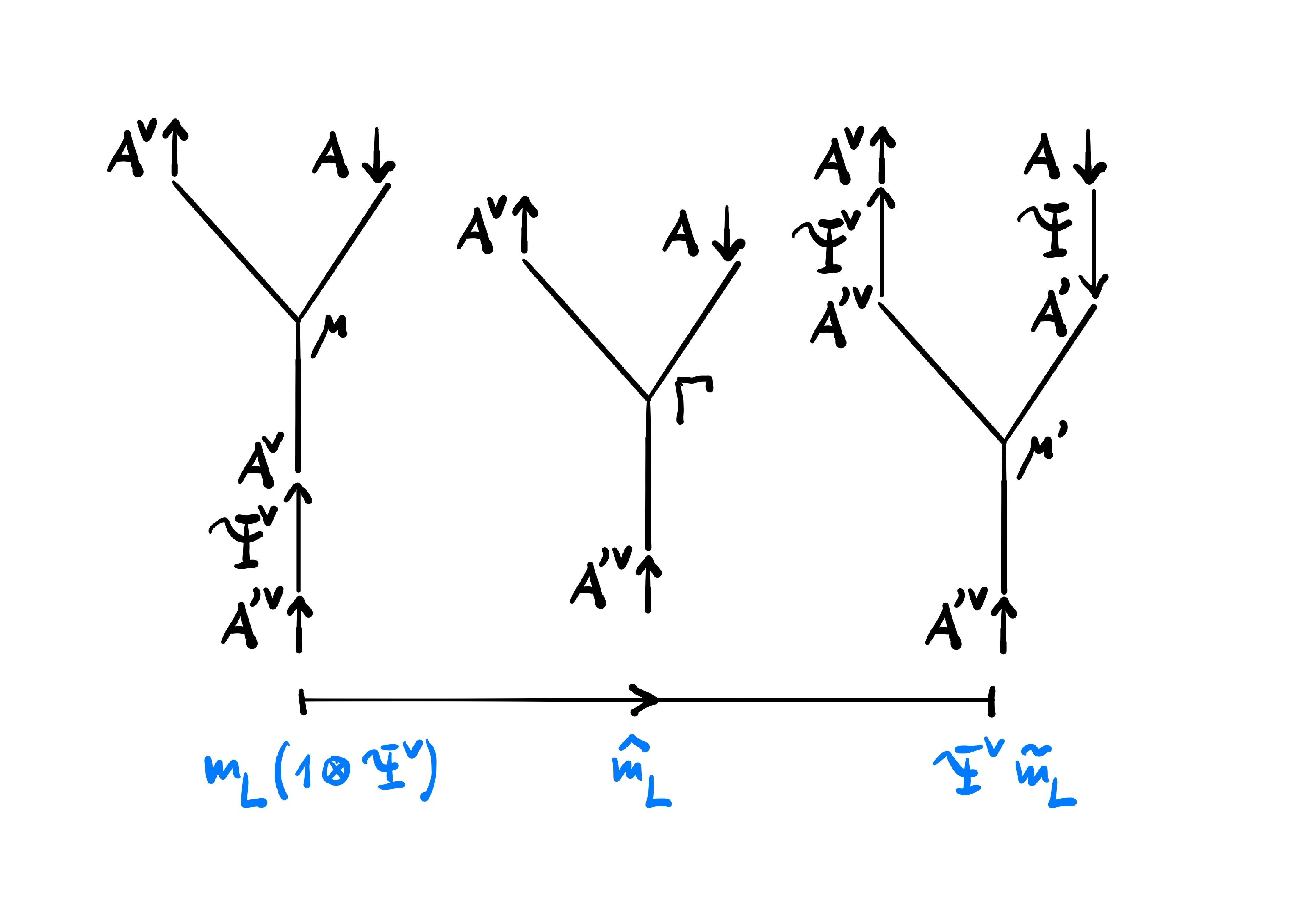}
\caption{The map $-\wh m_L$}
\label{fig:mL-hat}
\end{center}
\end{figure}

The definition of $-\wh\sigma$ is shown in Figure~\ref{fig:sigma-hat-new-Q} and uses the map $\ol{\Theta}_{\tau c_0,\tau c_0}$ from Figure~\ref{fig:Thetabar-new-Q}.
We read off that it satisfies condition (3) which takes the form
$$
  [\p,\wh\sigma] = \sigma(\Psi^\vee\otimes \Psi^\vee)-\Psi^\vee\wt\sigma-\wh m_R(1\otimes
c\Psi^\vee)+\wh m_L(c\Psi^\vee\otimes 1).
$$

\begin{figure}
\begin{center}
\includegraphics[width=\textwidth]{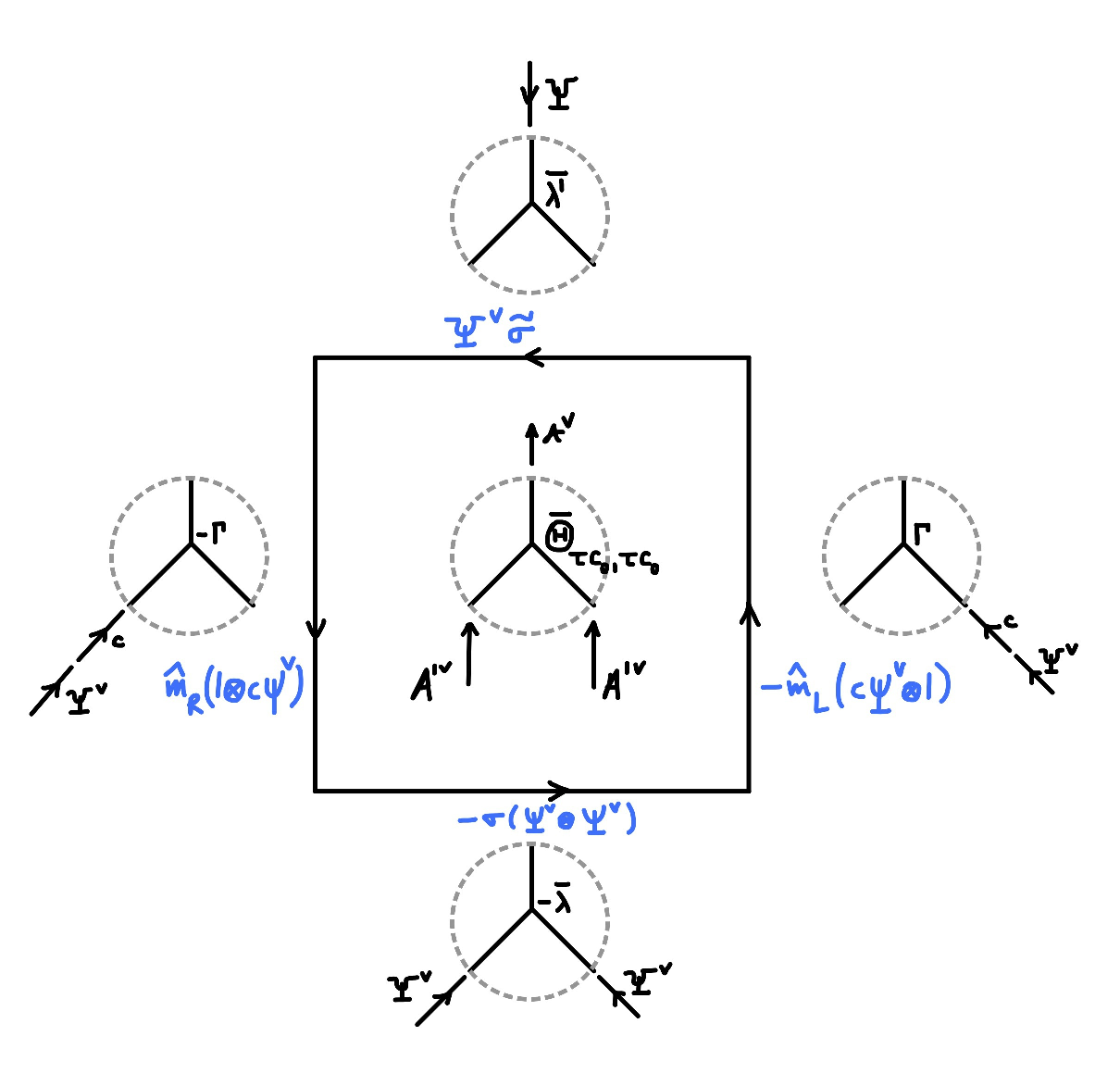}
\caption{The map $-\wh\sigma$}
\label{fig:sigma-hat-new-Q}
\end{center}
\end{figure}

With $\wh\tau_L:=0$ and $f=1$, $g=\Psi^\vee$ and $\wh\mu=0$, the second
equation in condition (3) becomes
$$
   \wt\tau_L = \tau_L(\Psi^\vee\otimes 1)+c\wh m_R.
$$
This equality is seen by inspection of Figure~\ref{fig:tauL-tilde-new-Q}
for $\wt\tau_L$, where the left half of the homotopy corresponds to $c\wh m_R$ and the right half corresponds to
$\tau_L(\Psi^\vee\otimes 1)$.
Similarly, the last equation in condition (3) holds with $\wh\tau_R:=0$. 

With $\wh\beta:=0$ and $f=1$, $g=\Psi^\vee$ and
$\beta=\wh\tau_L=\wh\tau_R=0$, 
condition (5) becomes 
$$
-\wt\beta + c\wh\sigma=0,
$$
which holds due to the condition $\Theta c=0$ in Definition~\ref{def:A2+mor}. Indeed, this expression reduces to a sum involving as before the side terms in the definition of $\Theta_{\tau c_0,\tau c_0}$ from Figure~\ref{fig:Thetabar-new-Q}, each appearing twice and with opposite signs. The expression therefore vanishes. 

This finishes the construction of the left hand special $A_2$-morphism
in~\eqref{eq:A2+toA2-mor}, and thus the proof of Proposition~\ref{prop:A2+toA2-mor}. 
\end{proof}

\begin{remark}
The proof of Proposition~\ref{prop:A2+toA2-mor} does not use the condition $\lambda c=0$ in the definition of a special $A_2^+$-algebra. We have chosen to include it in the definition in order to put it on an equal footing with the condition $\Theta c=0$ in the definition of a special morphism, and because it is satisfied in our main examples in~\cite{CHO-MorseFloerGH}.
\end{remark}

\subsection{Duality for $A_2^+$-algebras}\label{sec:PD-A2+}

Let $(\cA,\p,c_0,Q_0,\mu,\lambda,B)$ be an $A_2^+$-algebra and assume that $\cA$ is free and finite dimensional over $R$ in each degree. Then we have a canonical isomorphism $\cA\cong\cA^{\vee\vee}$ under which the maps $c,Q:\cA^\vee\to\cA$ induced by $c_0,Q_0$ satisfy $c^\vee-c=[\p,Q]$ and canonical chain isomorphisms
$$
  Cone(c)\cong Cone(c^\vee)\cong Cone(c)^\vee[-1].
$$
Here the first chain isomorphism is 
$$
  \begin{pmatrix}1 & -Q \\ 0 & 1\end{pmatrix}\colon
  \left(\cA\oplus\cA^\vee[-1],\begin{pmatrix}\p & c \\ 0 & -\p^\vee\end{pmatrix}\right) \to
  \left(\cA\oplus\cA^\vee[-1],\begin{pmatrix}\p & c^\vee \\ 0 & -\p^\vee\end{pmatrix}\right)
$$
and the second one simply interchanges to two direct summands,
$$
\left(\cA\oplus\cA^\vee[-1],\begin{pmatrix}\p & c^\vee \\ 0 & -\p^\vee\end{pmatrix}\right) \cong
  \left(\cA^\vee[-1]\oplus\cA,\begin{pmatrix}-\p^\vee & 0 \\ c^\vee & \p\end{pmatrix}\right).
$$
Let us denote by $\boldsymbol{\mu}$ the chain level product on $Cone(c)$ arising from Proposition~\ref{prop:TQFT+} and Proposition~\ref{prop:product_on_cone}. It is algebraically dual to a chain level coproduct $\tau \boldsymbol{\mu}^\vee$ 
on $Cone(c)^\vee[-1]$, which under the chain isomorphisms above corresponds to a coproduct $\boldsymbol{\lambda}$ on $Cone(c)$. This coproduct is in turn algebraically dual to a product $\boldsymbol{\lambda}^\vee\tau$ on $Cone(c)^\vee[-1]$, which under the chain isomorphisms above corresponds to the original product $\boldsymbol{\mu}$ on $Cone(c)$. 
On the level of homology we thus obtain the following algebraic counterpart of our main Poincar\'e duality theorem:

\begin{theorem}[Duality for $A_2^+$-algebras]\label{thm:PD-A2+}
Let $(\cA,\p,c_0,Q_0,\mu,\lambda,B)$ be an $A_2^+$-algebra which is free and finite dimensional over $R$ in each degree. Then $Cone(c)$ carries a canonical product $\boldsymbol{\mu}$ and coproduct $\boldsymbol{\lambda}$ and we have a canonical isomorphism 
$$
  \Bigl(H_*\bigl(Cone(c)\bigr),\boldsymbol{\mu},\boldsymbol{\lambda}\Bigr)\cong
  \Bigl(H_{*-1}\bigl(Cone(c)^\vee\bigr),\boldsymbol{\lambda}^\vee \tau,\tau \boldsymbol{\mu}^\vee\Bigr)
$$
intertwining the products and coproducts.
\qed
\end{theorem}


By construction, the components of the product $\boldsymbol{\mu}$ as well as the coproduct $\boldsymbol{\lambda}$ on the cone are obtained by dualizing the operations $\mu$, $\lambda$ of the $A_2^+$-structure at some of their inputs and outputs, together with contributions from the secondary copairing $Q_0$.

\appendix

\section{Multilinear shifts}\label{sec:shifts}

We describe in this section our convention for shifts of multilinear maps. It is inspired by Lef\`evre-Hasegawa's discussion of shifts in the context of $A_\infty$-algebras~\cite[\S1.2.2]{Lefevre-Hasegawa}. In this section the notation $|\cdot|$ stands for degree (of an element, resp. of a linear map). 

Given a chain complex $(V,\p)$, recall that $V[k]$ is the chain complex in which we shift the degree down by $k\in\Z$ and multiply the differential by $(-1)^k$, i.e.  
$$
V[k]_i=V_{i+k}, \qquad \p[k]=(-1)^k \p. 
$$
We denote 
$$
s_k:V\to V[k]
$$ 
the map induced by $\mathrm{Id}_V$. This is a chain isomorphism of degree $-k$, and we denote its inverse, which is a chain isomorphism of degree $k$, by  
$$
\omega_k:V[k]\to V.
$$ 

\begin{definition}
Let $V_1,\dots,V_\ell$ and $V$ be chain complexes and let 
$$
\alpha:V_1\otimes \dots \otimes V_\ell\to V
$$
be a multilinear map. Given $k_1,\dots,k_\ell,k\in\Z$ we denote 
$$
\overline k=(k_1,\dots,k_\ell;k)
$$ 
and define the $\overline k$-shift of $\alpha$ to be 
$$
\alpha[\overline k]:V_1[k_1]\otimes \dots \otimes V_\ell[k_\ell]\to V[k],
$$  
$$
\alpha[\overline k]=s_k\alpha \omega_{k_1}\otimes \dots\otimes \omega_{k_\ell}.
$$
\end{definition}
In an equivalent formulation, the shift $\alpha[\overline k]$ is defined from the commutative diagram 
$$
\xymatrix{
\bigotimes_i V_i[k_i]\ar[r]^-{\alpha[\overline k]}\ar[d]_{\otimes_i \omega_{k_i}}& V[k] \\
\bigotimes_i V_i \ar[r]_-\alpha & V \ar[u]^{s_k}
}
$$
Yet equivalently, given elements $x_i\in V_i$ and $x\in V$, denote $\overline x_i=s_{k_i}x_i\in V[k_i]$ and $\overline x=s_kx\in V[k]$. We then have  
$$
\alpha[\overline k](\bar x_1,\dots,\bar x_\ell)=(-1)^{|\overline x_1|(k_2+\dots+k_\ell) + \dots + |\overline x_{\ell-1}|k_\ell} \overline{\alpha(x_1,\dots,x_\ell)}. 
$$

The degrees of $\alpha$ and $\alpha[\overline k]$ are related by 
$$
|\alpha[\overline k]|=|\alpha|+k_1+\dots+k_\ell - k, 
$$
and we have 
$$
\big[\p,\alpha[\overline k]\big] = (-1)^k [\p,\alpha][\overline k]. 
$$
In particular, if $\alpha$ is a chain map then so is $\alpha[\overline k]$.

\begin{example} Let $\{\mu^d_\cA\}$, $d\ge 1$ be an $A_\infty$-algebra structure on $\cA$, with $\mu^d_\cA:\cA[-1]^{\otimes d}\to \cA[-1]$ of degree $-1$ and $[\mu,\mu]=0$. Denoting 
$$
\p=\mu^1_\cA[1]=-\mu^1_\cA
$$
the shifted differential on $\cA=\cA[-1][1]$, the shifted product 
$$
\mu=\mu^2_\cA[1,1;1]:\cA\otimes \cA\to \cA
$$
is a chain map with respect to $\p$ and is defined on elements as 
$$
\mu(a,b)=(-1)^{|a|}\mu^2_\cA(a,b). 
$$
This is exactly our convention~\eqref{eq:dga-from-Ainfty}. Moreover, formula~\eqref{eq:associator-from-Ainfty} shows that the associator for $\mu$ is precisely given by $f=\mu^3_\cA[1,1,1;1]$, i.e. 
$$
\mu(\mu(a_1,a_2),a_3)-\mu(a_1,\mu(a_2,a_3))= [\p,f](a_1,a_2,a_3). 
$$
\end{example}

\begin{example}
If $f:V\to V$ is a graded map, then $f[k]_i=f_{i+k}$. In particular $f[k][-k]=f$ (the shift is ``involutive" on linear maps).  
\end{example}

\begin{remark}[Non-involutivity] Given $\overline k=(k_1,\dots,k_\ell;k)$, denote $-\overline k=(-k_1,\dots,-k_\ell;-k)$. Then 
$$
\alpha[\overline k][-\overline k]=(-1)^{k_1(k_2+\dots+k_\ell) + \dots + k_{\ell-1}k_\ell} \alpha. 
$$
Therefore the shift on multilinear maps is only involutive up to sign. 
\end{remark} 

\begin{remark}[Non-commutativity]
Given $\overline k=(k_1,\dots,k_\ell;k)$ and $\overline t=(t_1,\dots,t_\ell;t)$, define 
$\overline k + \overline t = (k_1+t_1,\dots,k_\ell+t_\ell;k+\ell)$. Then 
\begin{align*}
(-1)^{t_1(k_2+\dots+k_\ell) + \dots + t_{\ell-1}k_\ell} & \alpha[\overline k][\overline t] \\
& =\alpha[\overline k+\overline t] = (-1)^{k_1(t_2+\dots+t_\ell) + \dots + k_{\ell-1}t_\ell}\alpha[\overline t][\overline k],
\end{align*}
and thus $\alpha[\overline k][\overline t]$ differs in general from $\alpha[\overline t][\overline k]$, and from $\alpha[\overline k + \overline t]$, by a sign. (This is also the reason for non-involutivity.) 
\end{remark}

\begin{remark}[Associativity]
Let us view the shift $[\overline k]$ as acting from the right on the space of multilinear maps $L(V_1\otimes\dots\otimes V_\ell,V)$. Then 
$$
([\overline k][\overline t])[\overline s]= [\overline k]([\overline t][\overline s]). 
$$
Indeed, a straightforward computation shows that both terms are equal to 
\begin{align*}
(-1)^{s_1(t_2+\dots+t_\ell)+ \dots s_{\ell-1}t_\ell} & (-1)^{s_1(k_2+\dots+k_\ell)+\dots s_{\ell-1}k_\ell} \\
& \times \, (-1)^{t_1(k_2+\dots+k_\ell)+\dots+t_{\ell-1}k_\ell} [\overline k + \overline t +\overline s].
\end{align*}
\end{remark}

\section{DG conventions} \label{sec:dg_linear_algebra}

In this section we discuss the Koszul sign rule, we deduce the differential graded (dg) sign conventions from the single rule for the differential on a tensor product, and discuss a useful pictorial representation.

Let $R$ be a commutative ring. The notions of left $R$-module, right $R$-module and symmetric $R$-bimodule coincide, so that we use the simplified terminology ``$R$-module". We denote generically the differential of a dg $R$-module by $\p$, and will further specify the notation only if necessary. We work in homological grading convention: the differential $\p$ has degree $-1$. The degree of elements or maps is denoted $|\cdot|$. 

\subsection{Fundamental conventions} \label{sec:conventions}
Let $\cA$, $\cB$, $\cC$, $\cD$ be dg $R$-modules.
The fundamental conventions of dg linear algebra are the following.

{\sc 1. Differential on the tensor product.} Let $\cA\otimes \cB$ be the $R$-module whose degree $k$ part is $\oplus_{i+j=k} \cA_i\otimes \cB_j$. One defines a differential by 
the formula\footnote{This formula has geometric roots. From a purely algebraic perspective, it is the essentially unique possible choice in order to achieve the relation $\p^2=0$. The other possibility would be $\p(a\otimes b)=(-1)^{|b|}\p a\otimes b + a\otimes \p b$, which is equivalent to~\eqref{eq:tensor_prod} under the twist isomorphism $A\otimes B\stackrel\simeq\longrightarrow B\otimes A$, $a\otimes b\mapsto (-1)^{|a|\cdot |b|}b\otimes a$.} 
\begin{equation}\label{eq:tensor_prod}
\p(a\otimes b) = \p a \otimes b + (-1)^{|a|}a\otimes \p b.
\end{equation} 
In view of the Koszul sign rule below, this can also be rewritten 
$$
\p_{\cA\otimes \cB} = \p_\cA\otimes \Id_{\cB}+ \Id_{\cA}\otimes \p_\cB.
$$

{\sc 2. Differential on the $\Hom$ module.} We denote $\Hom(\cA,\cB)$ the $R$-module whose degree $r$ part consists of $R$-module maps $f:\cA_*\to\cB_{*+r}$. The differential is defined to be  
$$
\p f = \p\circ f - (-1)^{|f|}f\circ \p. 
$$
We also write $\p f=[\p,f]$. 

{\sc 2bis. Differential on the dual module.} A particular case of the previous construction is that of the \emph{dual $R$-module}  
$$
\cA^\vee = \Hom(\cA,R).  
$$
Here the base ring $R$ is understood to be supported in degree $0$. As a consequence $\cA^\vee$ is graded as $\cA^\vee_*=\Hom(\cA_{-*},R)$ (note the minus sign!), and is endowed with the differential 
$$
\p f = -(-1)^{|f|}f\circ \p.
$$

{\sc 3. The evaluation map is a chain map.} The \emph{evaluation map} is
$$
\ev:\Hom(\cA,\cB)\otimes \cA\to \cB,\qquad f\otimes a\mapsto f(a). 
$$
We also write $f(a)=\langle f, a\rangle$. This has degree $0$ and is a chain map under conventions 1 and 2.

\begin{lemma}
Assuming any of the conventions 1, 2 or 3, the two other conventions are equivalent. (For the implication $2+3\Rightarrow 1$ we assume that $\p_{\cA\otimes\cB}$ has the form $a\otimes b\mapsto \pm\p a\otimes b\pm a \otimes \p b$.)
\qed
\end{lemma}

{\sc 3bis. The evaluation pairing is a chain map.} The \emph{evaluation pairing} is 
$$
\ev:\cA^\vee\otimes \cA\to R,\qquad f\otimes a\mapsto f(a) =: \langle f,a\rangle. 
$$
This has degree $0$ and is a chain map under conventions 1 and 2bis. 

\begin{lemma}Assuming convention 1, conventions 2bis and 3bis are equivalent.  \qed
\end{lemma}

{\sc 4. The twist isomorphism is a chain map.} The \emph{twist isomorphism} is the dg $R$-module isomorphism 
\begin{equation} \label{eq:twist_iso}
\tau:\cA\otimes \cB\stackrel\simeq\longrightarrow \cB\otimes \cA,\qquad a\otimes b \mapsto (-1)^{|a|\cdot |b|}b\otimes a. 
\end{equation}
Under convention 1, a chain isomorphism $\cA\otimes \cB\stackrel\simeq\longrightarrow\cB\otimes\cA$ of the form $a\otimes b\mapsto \varepsilon(|a|,|b|)b\otimes a$ with $\varepsilon(|a|,|b|)\in\{\pm 1\}$ is necessarily of the form~\eqref{eq:twist_iso} up to global sign change. As an example, while $a\otimes b\mapsto b\otimes a$ is an $R$-module isomorphism, it is not in general a chain map.

{\sc 5. The Koszul sign rule.} There is a chain map 
$$
T:\Hom(\cA,\cB)\otimes \Hom(\cC,\cD)\to \Hom(\cA\otimes \cC,\cB\otimes \cD)
$$
defined by the commutative diagram 
\begin{equation}\label{eq:HomABCD}
{\scriptsize 
\xymatrix
@C=40pt
{
\Hom(\cA,\cB)\otimes \Hom(\cC,\cD) \otimes \cA\otimes \cC\ar[r]^-{1\otimes \tau_{23}\otimes 1}\ar[d]_{T\otimes 1\otimes 1}& \Hom(\cA,\cB)\otimes \cA \otimes \Hom(\cC,\cD) \otimes \cC \ar[d]_{\ev\otimes\ev}\\
\Hom(\cA\otimes \cC,\cB\otimes\cD)\otimes \cA\otimes\cC\ar[r]_-\ev& \cB\otimes \cD.
}
}
\end{equation}
Here $\tau_{23}$ is the twist isomorphism applied on the 2nd and 3rd factors of the tensor product. With a slight abuse of notation we denote $T(f\otimes g)$ by $f\otimes g$. Then the previous commutative diagram amounts to the familiar Koszul sign rule 
$$
\langle f\otimes g,a\otimes c\rangle = (-1)^{|g|\cdot |a|}f(a)\otimes g(c).
$$
We thus see that the Koszul sign rule is a consequence of convention 4 for the twist map, which in turn is a consequence of convention 1 for the differential on the tensor product.

The bottom line of the discussion is the following: 
\begin{itemize}
\item conventions 4 and 5 are consequences of convention 1. 
\item assuming any one of the conventions 1, 2 or 3, the two other ones are equivalent. 
\end{itemize}

At this point it is instructive to quote Loday-Vallette~\cite[1.5.3]{Loday-Vallette}: {\it ``[The Koszul convention] permits us to avoid complicated signs in the formulas provided one works with the maps (or functions) \emph{without} evaluating them on the elements. When all the involved operations are of degree $0$, the formulas in the nongraded case apply mutatis mutandis to the graded case."} The point of the previous discussion was to deduce the convention for the twist isomorphism and for the Koszul sign rule from convention 1, which is ultimately of geometric nature. 

\subsection{Consequences of the fundamental conventions} \label{sec:consequences}

From this point on we will be able to avoid all signs by working exclusively with functional equalities, or commutative diagrams. 

1. We have a canonical chain map 
$$
T:\cB\otimes \cC^\vee\longrightarrow \Hom(\cC,\cB)
$$
obtained by specializing~\eqref{eq:HomABCD} with $\cA=\cD=R$. This is equivalently defined by the commutative diagram 
$$
\xymatrix
@C=60pt
{
\cB\otimes \cC^\vee \otimes \cC \ar[r]^-{1\otimes \ev}\ar[d]_{T\otimes 1}  & \cB \ar@{=}[d] \\
\Hom(\cC,\cB)\otimes \cC \ar[r]_-{\ev}& \cB.
}
$$
If $\cC$ is free and finite dimensional over $R$ in each degree the map $T$ is a chain isomorphism. We have $T(b\otimes f)(c)=b\cdot f(c)$.

There is of course also a canonical chain map 
$$
\widetilde{T}:\cC^\vee\otimes \cB \longrightarrow \Hom(\cC,\cB) 
$$
obtained by suitably specializing~\eqref{eq:HomABCD}. This is alternatively defined by the commutative diagram 
$$
\xymatrix
@C=60pt
{
\cC^\vee\otimes\cB\otimes \cC \ar[r]^-{\tau_{12}\otimes 1}\ar[d]_{\widetilde{T}\otimes 1}  & \cB\otimes\cC^\vee\otimes \cC \ar[r]^-{1\otimes \ev} & \cB \ar@{=}[d] \\
\Hom(\cC,\cB)\otimes \cC \ar[rr]_-{\ev}& & \cB.
}
$$
Thus $\widetilde{T}=T\circ\tau$ and $\widetilde{T}(f\otimes b)(c) = (-1)^{|b|\cdot |f|}b\cdot f(c)$. As expected, there is no sign involved when expressing the operation in terms of a commutative diagram, i.e. in functional terms. The sign appears only when explicitly evaluating on elements. 

2. We have canonical \emph{adjunction isomorphisms} 
$$
F:\Hom(\cA,\Hom(\cB,\cC))\stackrel\simeq\longleftrightarrow \Hom(\cA\otimes \cB,\cC):G.
$$
The map $F$ acts by $F(f)(a\otimes b)=f(a)(b)$. Its inverse $G=F^{-1}$ acts by $G(g)(a)(b)=g(a\otimes b)$. The fact that the adjunction isomorphisms are chain maps ties together conventions 1 and 2 in the same manner as does convention 3. 

3. There is a canonical map
$$
\cA^\vee\otimes \cB^\vee\to (\cA\otimes \cB)^\vee
$$
realized by the composition $F\widetilde{T}:\cA^\vee\otimes \cB^\vee\to \Hom(\cA,\cB^\vee)\stackrel\simeq\to (\cA\otimes \cB)^\vee$. Also, there is a canonical map
$$
\cB^\vee\otimes \cA^\vee\to (\cA\otimes \cB)^\vee
$$
obtained from the previous one by pre-composing with $\tau$, and thus realized by $FT:\cB^\vee\otimes \cA^\vee\to \Hom(\cA,\cB^\vee)\stackrel\simeq\to(\cA\otimes \cB)^\vee$. These maps are isomorphisms if either $\cA$ or $\cB$ is free and finite dimensional over $R$. 

4. For sequel use, it is useful to consider the \emph{coevaluation map} 
$$
\ev^\vee:R\to(\cA^\vee\otimes \cA)^\vee, 
$$
dual to the evaluation map, and also the canonical map 
$$
\iota : \cA\to \cA^{\vee\vee}
$$
defined by the commutative diagram 
$$
\xymatrix
@C=60pt
{\cA\otimes \cA^\vee\ar[r]^\tau \ar[d]_{\iota\otimes 1}& \cA^\vee\otimes\cA \ar[r]^\ev & R \ar@{=}[d]\\
\cA^{\vee\vee}\otimes \cA^\vee \ar[rr]_{\ev}& & R.
} 
$$
Explicitly, the map $\iota$ acts as $a\mapsto \left( f\mapsto (-1)^{|a|\cdot |f|} \langle f,a\rangle\right)$. (Formally we have $\langle a,f\rangle = (-1)^{|a|\cdot |f|} \langle f,a\rangle$.)

If $\cA$ is free and finite dimensional over $R$ then the map $\iota: \cA\to \cA^{\vee\vee}$ is an isomorphism, the modules $\cA^\vee$ and $\cA^{\vee\vee}$ are also free and finite dimensional over $R$, and $\ev^\vee$ can be interpreted as a chain map\footnote{In this situation, given a basis $e_1,\dots,e_\ell$ of $\cA$ and denoting $e_1^*,\dots,e_\ell^*$ the dual basis of $\cA^\vee$, we have $\ev^\vee(1)=\sum_i e_i\otimes e_i^*$. This is the so-called \emph{Casimir element}.}  
$$
\ev^\vee:R\to \cA\otimes\cA^\vee.
$$ 
We then have $(\ev\otimes 1)(1\otimes \ev^\vee)=\text{Id}_{\cA^\vee}$ and $(1\otimes \ev)(\ev^\vee\otimes 1)=\text{Id}_\cA$.

\subsection{The language of trees} \label{sec:trees}

We depict operations involving multiple inputs and outputs in $\cA$ and $\cA^\vee$ by trees whose half-edges carry two different labels: 1. \emph{input} or \emph{output}, signified by an ingoing or outgoing arrow, and 2. $\cA$ or $\cA^\vee$. In our graphical representation, the input edges labeled $\cA$ or $\cA^\vee$ are directed ``upwards", resp. ``downwards", and the output edges labeled $\cA$ or $\cA^\vee$ are directed ``downwards", resp. ``upwards". The composition of operations is represented by stacking trees one on top of the other. 

For example, the graphical representation of the maps $\ev$ and $\ev^\vee$ in the case where $\iota$ is an isomorphism is 
\begin{center}
\begin{tikzpicture}
\draw (0,0) -- (0,2); 
 \node [below] at (0,0) {$\cA^\vee$}; 
 \draw [->] (0.25,0) -- (0.25,0.25);
 \draw [->] (0.25,2) -- (0.25,1.75);
 \node [above] at (0,2) {$\cA$};
 \node at (0.35,1) {$\ev$};
\draw (4,0) -- (4,2);
 \node [below] at (4,0) {$\cA$}; 
 \draw [<-] (4.25,0) -- (4.25,0.25);
 \draw [<-] (4.25,2) -- (4.25,1.75);
 \node [above] at (4,2) {$\cA^\vee$};
 \node at (4.45,1) {$\ev^\vee$};
\end{tikzpicture}
\end{center}
The relations $(\ev\otimes 1)(1\otimes \ev^\vee)=\text{Id}_{\cA^\vee}$ and $(1\otimes \ev)(\ev^\vee\otimes 1)=\text{Id}_\cA$ become graphically 
\begin{center}
\begin{tikzpicture}
\draw (0,0) -- (0,2); 
 \node [below] at (0,0) {$\cA^\vee$}; 
 \draw [->] (0.25,0) -- (0.25,0.25);
 \draw [->] (0.25,2) -- (0.25,1.75);
 \node [left] at (0,2) {$\cA$};
 \node at (0.35,1) {$\ev$};
\draw (0,2.25) -- (0,4.25);
 \draw [<-] (0.25,2.25) -- (0.25,2.5);
 \draw [<-] (0.25,4.25) -- (0.25,4);
 \node [above] at (0,4.25) {$\cA^\vee$};
 \node at (0.45,3.25) {$\ev^\vee$};
 
 \node at (1,2) {$=$};
 
 \draw (2,1) -- (2,3);
 \node [below] at (2,1) {$\cA^\vee$}; 
 \draw [->] (2.25,1) -- (2.25,1.25);
 \draw [<-] (2.25,3) -- (2.25,2.75);
 \node [above] at (2,3) {$\cA^\vee$};
 \node at (2.55,2) {$\text{Id}_{\cA^\vee}$};
  
\draw (6,0) -- (6,2);
 \node [below] at (6,0) {$\cA$}; 
 \draw [<-] (6.25,0) -- (6.25,0.25);
 \draw [<-] (6.25,2) -- (6.25,1.75);
 \node [left] at (6,2) {$\cA^\vee$};
 \node at (6.45,1) {$\ev^\vee$};
\draw (6,2.25) -- (6,4.25);
 \draw [->] (6.25,2.25) -- (6.25,2.5);
 \draw [->] (6.25,4.25) -- (6.25,4);
 \node [above] at (6,4.25) {$\cA$};
 \node at (6.35,3.25) {$\ev$}; 
 
\node at (7,2) {$=$};
 
 \draw (8,1) -- (8,3);
 \node [below] at (8,1) {$\cA$}; 
 \draw [<-] (8.25,1) -- (8.25,1.25);
 \draw [->] (8.25,3) -- (8.25,2.75);
 \node [above] at (8,3) {$\cA$};
 \node at (8.45,2) {$\text{Id}_{\cA}$}; 
\end{tikzpicture}
\end{center}
The graphical representation is more suggestive than the formulas.

Using the evaluation map, an input (output) in $\cA$ can be converted to an output (input) in $\cA^\vee$: add a tensor factor $\cA$ for each new output in $\cA^\vee$, and apply the evaluation map as dictated by the tree. For example, suppose we are given a bilinear product $\mu$ on $\cA$ (which is not assumed to be commutative or associative) depicted by the trivalent tree with one vertex:
\begin{center}
\begin{tikzpicture}
\draw (0,0) -- (-2,2); 
	\node [right] at (-1.5,1.5) {\tiny $1$};
\draw (0,0) -- (2,2); 
	\node [left] at (1.5,1.5) {\tiny $2$};
\draw (0,0) -- (0,-2);
\node [right] at (0,0) {$\mu$};
\draw [->] (-2,2.5) -- (-2,2.25); \node [right] at (-2,2.5) {$\cA$}; 
\draw [->] (2,2.5) -- (2,2.25); \node [right] at (2,2.5) {$\cA$}; 
\draw [->] (0,-2.25) -- (0,-2.5); \node [right] at (0,-2.25) {$\cA$};
\end{tikzpicture}
\end{center} 
Here the labels $1$ and $2$ on the incoming edges express the order of the arguments for the product $\mu:\cA\otimes \cA\to \cA$. The same tree then defines left and right multiplications
$$m_L:\cA\otimes \cA^\vee\to \cA^\vee\quad\text{and}\quad m_R:\cA^\vee\otimes \cA\to \cA^\vee$$
by converting inputs and outputs as in Figure~\ref{fig:mLmR}:
\begin{figure}
\begin{center}
\includegraphics[width=.7\textwidth]{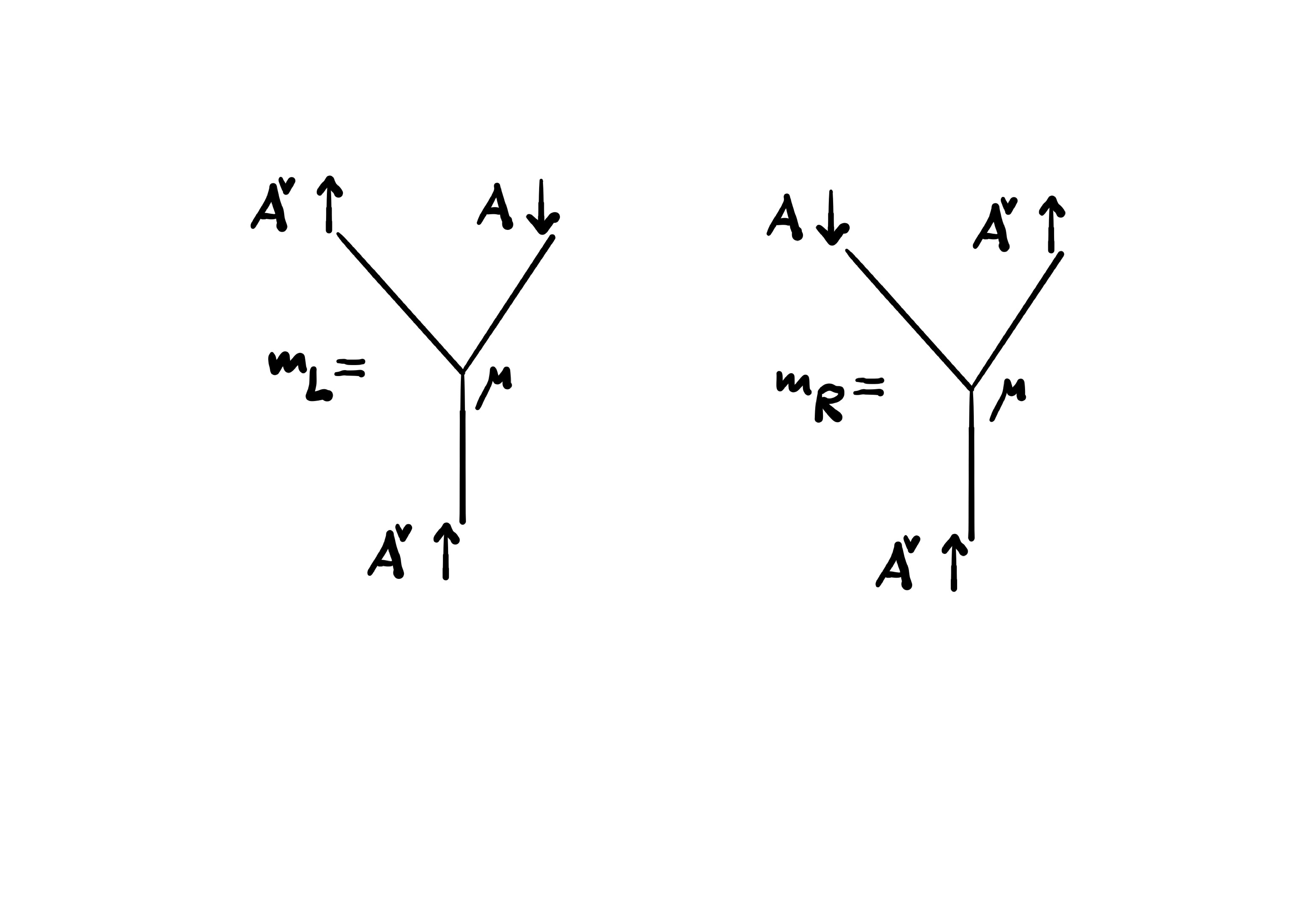}
\caption{Definition of $m_L$ and $m_R$}
\label{fig:mLmR}
\end{center}
\end{figure}
According to our rules, the operation $m_L$ is defined by the commuting diagram
$$
\xymatrix{
\cA\otimes \cA^\vee\otimes \cA \ar[d]^{\tau_{23}\tau_{12}}\ar[r]^{\ \ \ m_L\otimes 1} & \cA^\vee\otimes\cA \ar[dr]^{\ev} \\ 
\cA^\vee \otimes \cA \otimes \cA \ar[r]^{\ \ \ 1\otimes\mu} & \cA^\vee\otimes\cA \ar[r]^{\ \ \ev} & R, \\ 
}
$$
where $\tau_{23}\tau_{12}$ is the permutation moving the first tensor factor to the third position. This writes out as 
$$
\ev (m_L\otimes 1)=\ev(1\otimes\mu)\tau_{23}\tau_{12},
$$
which means $\langle m_L(a,f),b\rangle=(-1)^{|a|\cdot (|f|+|b|)}\langle f,\mu(b,a)\rangle$ in terms of elements. In accordance with the discussion in~\S\ref{sec:conventions}, the defining formula does not involve signs when written in functional terms: the signs appear only when the formula is evaluated on elements. Using the notation $\langle a,f\rangle = (-1)^{|a|\cdot |f|} \langle f,a\rangle$ from above, the last formula can also be written in equivalent form without signs as
\begin{equation}\label{eq:mL-mu}
\langle b,m_L(a,f)\rangle =\langle \mu(b,a),f\rangle.
\end{equation}
Similarly, the right multiplication reads
$$
\ev(m_R\otimes 1)=\ev(1\otimes\mu),
$$
which in terms of elements means 
\begin{equation}\label{eq:mR-mu}
\langle m_R(f,a),b\rangle =\langle f,\mu(a,b)\rangle.
\end{equation}

\begin{remark}\label{rmk:finite-dim}
Let us emphasize that in the preceding discussion as well as in the sequel we use only the evaluation map and not its dual, so everything works in the infinite dimensional case. If $\cA$ is free and finite dimensional, then we can view the conversion of inputs/outputs in $\cA$ to outputs/inputs in $\cA^\vee$ as adding the pictures for $\ev$ resp.~$\ev^\vee$ at the corresponding edge. 
\end{remark}

\begin{remark}\label{rmk:TQFT} One remarkable aspect about the language of trees is that the definitions are forced upon us by TQFT-type relations: two composite expressions which give rise to the same tree by formal gluing are equal. In the situation at hand, this requirement leaves as an available choice the order on the input half-edges of $\mu$, but this is fixed by the requirement that $m_L$ defines a \emph{left} module structure, and $m_R$ defines a \emph{right} module structure, in the case that the product $\mu$ is associative (so in this case $\cA^\vee$ becomes a bimodule over $\cA$). 
\end{remark}

\bibliographystyle{abbrv}
\bibliography{000_SHpair}

\end{document}